\documentclass{amsart}

\usepackage{varioref}
\usepackage[bookmarks=false]{hyperref}
\usepackage
{cleveref}
\crefname{thm}{Theorem}{Theorems}
\crefname{claim}{Claim}{Claims}

\usepackage{amsthm, amssymb, amsmath, a4wide}
\usepackage[sans]{dsfont} 
\usepackage{xifthen}
\usepackage{enumitem}

\usepackage{array}
\makeatletter
\newcommand{\thickhline}{%
    \noalign {\ifnum 0=`}\fi \hrule height 1pt
    \futurelet \reserved@a \@xhline
}
\newcolumntype{B}{@{\hskip\tabcolsep\vrule width 1pt\hskip\tabcolsep}}
\makeatother

\author{Pinaki Mondal}
\title{Analytic compactifications of $\cc^2$ part II - one irreducible curve at infinity}

\usepackage{xcolor}
\hypersetup{
    colorlinks,
    linkcolor={red!50!black},
    citecolor={blue!50!black},
    urlcolor={blue!80!black}
}

\DeclareMathOperator{\ind}{ind}

\makeatletter

\newcommand{\Rmnum}[1]{\expandafter\@slowromancap\romannumeral #1@}
\makeatother

\DeclareMathOperator\Div{div}

\DeclareMathOperator\ord{ord} 
 
\DeclareMathOperator\pole{pole}
\DeclareMathOperator\proj{Proj}

\DeclareMathOperator\sing{Sing} 

\DeclareMathOperator\supp{Supp}

\newcommand{\scrB}{\ensuremath{\mathcal{B}}}
\newcommand{\scrC}{\ensuremath{\mathcal{C}}}

\newcommand{\scrE}{\ensuremath{\mathcal{E}}}

\newcommand{\scrG}{\ensuremath{\mathcal{G}}}

\newcommand{\scrI}{\ensuremath{\mathcal{I}}}

\newcommand{\scrY}{\ensuremath{\mathcal{Y}}}

\newcommand{\cc}{\ensuremath{\mathbb{C}}}

\newcommand{\GG}{\ensuremath{\mathbb{G}}}
\newcommand{\kk}{\ensuremath{\mathbb{K}}}

\newcommand{\pp}{\ensuremath{\mathbb{P}}}
\newcommand{\qq}{\ensuremath{\mathbb{Q}}}

\newcommand{\zz}{\ensuremath{\mathbb{Z}}}

\newcommand{\sheaf}{\ensuremath{\mathcal{O}}}
\newcommand{\ntorus}{(\cc^*)^n}

\newcommand{\WP}{\ensuremath{{\bf W}\pp}}

\newcommand{\im}{\ensuremath{\Rightarrow}}

\newcommand{\into}{\ensuremath{\hookrightarrow}}

\newtheorem{thm}{Theorem}[section]
\newtheorem*{thm*}{Theorem}
\newtheorem{lemma}[thm]{Lemma}
\newtheorem*{lemma*}{Lemma}

\newtheorem{prop}[thm]{Proposition}
\newtheorem*{prop*}{Proposition}
\newtheorem{cor}[thm]{Corollary}
\newtheorem{claim}[thm]{Claim}
\newtheorem*{claim*}{Claim}
\newtheorem{proclaim}{Claim}[thm]
\newtheorem*{conjecture*}{Conjecture}

\theoremstyle{definition} 
\newtheorem{algorithm}[thm]{Algorithm}

\newtheorem*{constrinition*}{Construction-Definition}

\newtheorem*{convention*}{Convention}
\newtheorem{defn}[thm]{Definition}
\newtheorem*{defn*}{Definition}

\newtheorem*{definotation*}{Definition-Notation}
\newtheorem{example}[thm]{Example}
\newtheorem*{example*}{Example}

\newtheorem*{fact*}{Fact}

\newtheorem*{facts*}{Facts}

\newtheorem{notation}[thm]{Notation}

\newtheorem*{bold-note*}{Note}
\newtheorem{problem}[thm]{Problem}
\newtheorem*{problem*}{Problem}
\newtheorem{bold-question}[thm]{Question}
\newtheorem*{bold-question*}{Question}
\newtheorem{rem}[thm]{Remark}
\newtheorem{reminition}[thm]{Remark-Definition}
\newtheorem*{reminition*}{Remark-Definition}

\newtheorem{remexample*}{Remark-Example}

\newtheorem*{remtation*}{Remark-Notation}

\newtheorem*{remuestion*}{Remark-Question}

\newtheorem*{remvention*}{Remark-Convention}

\theoremstyle{remark}
\newtheorem*{rem*}{Remark}
\newtheorem*{note*}{Note}
\newtheorem*{notation*}{Notation}
\newtheorem*{question*}{Question}

\newtheorem*{questions*}{Questions}

\newcounter{UnorderedProofTempCtr}
\newcommand{\tempcommand}{}

\newcommand{\baromegastar}{\bar \omega^*}

\newcommand{\hot}{\text{h.o.t.}}
\newcommand{\lot}{\text{l.o.t.}}
\newcommand{\ldt}{\text{l.d.t.}}

\newcommand{\dpsx}[2]{{#1} \langle \langle #2 \rangle \rangle }
\newcommand{\dpsxc}{\dpsx{\cc}{x}}

\newcommand{\xomegatheta}{\bar X_{\vec \omega,\vec \theta}}
\newcommand{\xomegaprimetheta}{\bar X_{\vec \omega',\vec \theta}}
\newcommand{\xomegathetaprime}{\bar X_{\vec \omega,\vec \theta'}}

\newcommand{\xomegaprimeomegadtheta}{\bar X_{\vec \omega'_{\vec \omega, d},\vec \theta}}

\newcommand{\phiomegatheta}{\phi_{\vec \omega,\vec \theta}}
\newcommand{\tildephiomegatheta}{\tilde \phi_{\vec \omega,\vec \theta}}
\newcommand{\iomegatheta}{\iota_{\vec \omega,\vec \theta}}

\newcommand{\Xxy}{\bar X_{(x,y)}}

\newcommand{\deltaomegatheta}{\delta_{\vec \omega,\vec \theta}}
\newcommand{\deltaomegathetaprime}{\delta_{\vec \omega,\vec \theta'}}

\newcommand{\yomega}{\scrY_{\vec \omega}}
\newcommand{\yomegaprime}{\scrY_{\vec \omega'}}
\newcommand{\yomegaess}{\scrY^{e}_{\vec \omega}}
\newcommand{\yomegaessalg}{\scrY^{e, alg}_{\vec \omega}}

\newcommand{\gthetatheta}{g_{\vec \theta, \vec{\check \theta}}}
\newcommand{\gthetathetapprime}{g_{\vec \theta', \vec{\check {\theta'}}}}
\newcommand{\cthetathetac}{C_{\vec \theta, \vec{\check \theta}, c}}
\newcommand{\cthetathetacpprime}{C_{\vec \theta', \vec{\check {\theta'}}, c'}}

\newcommand{\cd}{\scrC_{\vec d}}
\newcommand{\cdbar}{\bar \scrC_{\vec d}}

\newcommand{\Etildephi}{\scrE_{\tilde \phi}}
\newcommand{\Eomega}{\scrE_{\vec \omega}}
\newcommand{\Eomegaprime}{\scrE_{\vec \omega'}}

\newlist{prooflist}{enumerate}{3}
\setlist[prooflist,1]{label=(\roman*)}
\setlist[prooflist,2]{label=(\arabic*)}
\setlist[prooflist,3]{label=(\alph*)}

\newlist{defnlist}{enumerate}{3}
\setlist[defnlist,1]{label=(\alph*)}
\setlist[defnlist,2]{label=(\arabic*), ref=(\alph{defnlisti}.\arabic*)}
\setlist[defnlist,3]{label=(\roman*), ref=(\alph{defnlisti}.\arabic{defnlistii}.\roman*)}

\newcommand{\skewness}{index}
\newcommand{\reachable}[1]{key compatible with {#1}}

\begin{document}

\begin{abstract} 
We classify {\em primitive normal compactifications} of $\cc^2$ (i.e.\ normal analytic surfaces containing $\cc^2$ for which the curve at infinity is irreducible), compute the moduli space of these surfaces and their groups of auomorphisms. In particular we show that in `most' of these surfaces $\cc^2$ is `rigidly embedded'. As an application we give a description of `embedded isomorphism classes' of planar curves with one place at infinity. We also compute the canonical divisor of these surfaces; it turns out that their log discrepancy is related to the Frobenius number of the semigroup of poles along the curve at infinity. We  use the computation to classify Gorenstein primitive compactifications of $\cc^2$ with rational and minimally elliptic singularities, extending a result of Brenton, Drucker and Prins \cite{brenton-graph-1}. As another application we characterize weighted projective spaces of the form $\pp^2(1,1,q)$ in terms of their {\em log discrepancy} and {\em \skewness}, generalizing a characterization of $\pp^2$ due to Borisov \cite{borisov}. 
\end{abstract}

\maketitle

\section{Introduction} \label{sec-intro}
{\em Normal analytic compactifications} (henceforth to be denoted by simply `compactifications')  of $\cc^2$ are arguably the simplest, and consequently a well-studied, class of surfaces: see e.g. \cite{morrow}, \cite{brentonfication}, \cite{brenton-singular}, \cite{brenton-graph-1}, \cite{miyanishi-zhang}, \cite{furushima}, \cite{ohta}, \cite{kojima}, \cite{koji-hashi}, \cite{favsson-dynamic}. In \cite{sub2-1} we studied singularities of these surfaces and of the curves at infinity (i.e.\ the complement of $\cc^2$), and in \cite{algebraicity} we gave an effective algorithm to determine the algebraicity of the simplest among these surfaces, namely those for which the curve at infinity (i.e.\ the complement of $\cc^2$) is irreducible; following \cite{ohta} we call these {\em primitive} compactifications.\\

Combining results of \cite{algebraicity} and \cite{sub2-1} yields a description of singularities of primitive compactifications (\cref{general-structure-thm}), and in the case of algebraic primitive compactifications, an explicit description of the defining equations (\cref{projective-embedding}) and the singularities of the curve at infinity (\cref{projllary}). In particular, all primitive algebraic compactifications of $\cc^2$ are {\em weighted complete intersections} and the curve at infinity has at most one singular point, which is at worst a {\em toric} (or monomial) singularity. Starting from these observations, we undertake in this article a detailed study of primitive compactifications.\\

The order of vanishing of polynomials along the curve at infinity on a primitive compactification of $\cc^2$ is a discrete valuation on $\cc[x,y]$, and it completely determines the compactification. The valuation is in turn determined by a finite sequence of polynomials called the {\em key forms} - these are natural analogues of the {\em key polynomials} (introduced by MacLane \cite{maclane-key} and widely used in valuation theory) of valuations centered at the origin. The {\em key sequence} associated to the compactification is the sequence consisting of orders of poles of these key forms along the curve at infinity. The notion of key sequences generalize the notion of {\em $\delta$-sequences} which are ubiquitous (see e.g.\ \cite[Section 2.1]{sathaye}, \cite[Section 3]{suzuki}) in the theory of plane curves with one place at infinity. The main technical observation of this article is that under appropriate choices of coordinates on $\cc[x,y]$, the key sequence associated to a primitive compactification can be brought to a (unique) normal form (\cref{normal-thm-0}). Moreover, if the key sequence is in the normal form, then the automorphisms of $\cc[x,y]$ that preserve the key sequence or the corresponding valuation can be explicitly described (\cref{normal-morphism,preserving-thm}). Using these properties of normal forms of key sequences, we establish a number of results, namely:\\

\begin{itemize}[wide]
\item We show that the embeddings of $\cc^2$ in `most' primitive compactifications are `rigid'; more precisely, if $\bar X$ is a primitive compactification of $\cc^2$ which is not isomorphic to a weighted projective surface of the form $\pp^2(1,1,q)$ (for some positive integer $q$), then $\bar X$ has only one subset isomorphic to $\cc^2$ (\cref{rigid-prop}).\\

\item We compute the groups of automorphisms of primitive compactifications of $\cc^2$ (\cref{aut-cor}). In particular, it turns out that `most' primitive compactifications, including all the non-algebraic ones, admit only finitely many automorphisms (\cref{aut-cor1}).\\

\item We explicitly describe the moduli space of primitive compactifications, i.e.\ the space of (isomorphism classes of) compact analytic surfaces of Picard rank $1$ which contain a copy of $\cc^2$. The moduli space of primitive compactifications with a fixed key sequence (in the normal form) turns out to be of the form $(\cc^*)^k \times \cc^l$ for some $k,l \geq 0$ (\cref{moduli-cor}). As an application of our result, we give a description (originally due in another form to Oka \cite{oka-moduli}) of the moduli space of {\em embedded isomorphism classes} of planar curves with one place at infinity with a fixed $\delta$-sequence (in the normal form) (\cref{one-place-cor}); in particular, the latter space is a quotient of $(\cc^*)^k \times \cc^l$ under an (explicitly described) action of $(\cc^*)^2$ . \\
\end{itemize}

The remaining main result of this article is a computation of the canonical divisors of primitive compactifications in terms of associated key sequences (\cref{canonical-thm}). This leads to a characterization (\cref{singularity-cor}) of the primitive compactifications with rational and elliptic singularities and those which are Gorenstein. As an immediate application we recover the classification in \cite{brenton-graph-1} of primitive Gorenstein compactifications of $\cc^2$ with rational singularities (\cref{rational-cor}), and in addition classify primitive Gorenstein compactifications of $\cc^2$ with minimally elliptic singularities (\cref{elliptic-cor}). \\



Two invariants of divisorial valuations (corresponding to curves at infinity on compactifications of $\cc^2$) play a special role in this article: {\em log discrepancy}\footnote{We follow \cite{jonsson-dykovich} to call it ``log discrepancy''; it is called ``$\bar K$ label'' in \cite{borisov}.} (\cref{log-discrepancy}) and {\em index}\footnote{It is called ``determinant label'' in \cite{borisov}; in \cite{jonsson-dykovich} it is defined, but not given any name; the motivation for our choice of the term `index' is explained in \cref{index}} (\cref{index}). As an application of our computation of the canonical class, we give a characterization of weighted projective spaces of the form $\pp^2(1, 1, q)$, $q \geq 0$, in terms of the log discrepancy and index of the valuation at infinity (\cref{11q-thm}). As a special case we recover Borisov's \cite{borisov} characterization of $\pp^2$ as (up to isomorphism) the unique compactification of $\cc^2$ such that the valuation corresponding to the curve at infinity has index $1$ and log discrepancy $-2$. \\

It follows from our computation of the canonical divisor (\cref{canonical-thm}) that the log discrepancy $A_\nu$ and the index $\alpha_\nu$ of a divisorial valuation $\nu$ centered at infinity satisfy the following identity:
\begin{align}
A_\nu + \alpha_\nu
	 &= \sum_{k=1}^{n+1}\alpha_k \omega_k - \sum_{k=0}^{n+1} \omega_k
	 \label{log-index}
\end{align}
where $(\omega_0, \ldots, \omega_{n+1})$ is the key sequence associated to $\nu$, and for each $k = 1, \ldots, n+1$, $\alpha_k\omega_k$ is the smallest positive multiple of $\omega_k$ which is in the (additive) group generated by $\omega_0, \ldots, \omega_{k-1}$. It turns out (see e.g.\ \cref{basis-lemma}) that $\omega_0, \ldots, \omega_{n+1}$ generate the {\em semigroup of poles} of polynomials along the corresponding curve at infinity. In the case that each $\alpha_k\omega_k$ is the semigroup generated by $\omega_0, \ldots, \omega_{k-1}$ (which is the case e.g.\ when $\nu$ comes from a primitive algebraic compactification of $\cc^2$ and the curve at infinity is non-singular - see assertion \eqref{curve-non-singularity} of \cref{projllary}), it follows from a result of Herzog \cite[Proposition 2.1]{herzog} that the semigroup of poles is {\em symmetric}\footnote{\label{symmetric-footnote} A sub-semigroup $S$ of $\zz$ is called {\em symmetric} if there exists $m$ in the group $\tilde S$ generated by $S$ such that for all $s \in \tilde S$, $s \in S$ iff $m - s \not\in S$.} and the expression on the right hand side of \eqref{log-index} computes its {\em Frobenius number}\footnote{The Frobenius number of a sub-semigroup of $\zz$ is the greatest integer not belonging to it.}!\\


In \cite{g2a-unpublished} we use the results of this article to classify $\GG^2_a$-surfaces (i.e.\ equivariant compactifications of $\cc^2$ with the additive group structure) of Picard rank one and answer some questions posed by Hassett and Tschinkel \cite{hassett-tschinkel}.  \\

We finish the introduction with a mention of one of the things {\em not} achieved in this work, namely a description of non-algebraic primitive compactifications of $\cc^2$ as explicit as the description of algebraic ones.
%
More precisely, in \cref{structure-section} we see that a primitive compactification $\bar X$ of $\cc^2$ which is not isomorphic to a weighted projective surface has a unique point $P_\infty$ with non-cyclic quotient singularity. The complement of $P_\infty$ in $\bar X$ is quasi-projective, and can be explicitly described in terms of key forms of the valuation associated to the curve $C_\infty$ at infinity (assertion \eqref{general-assertion} of \cref{projective-embedding}). However, if $\bar X$ is non-algebraic, then unlike the algebraic case we do not have an explicit description of any neighborhood of $P_\infty$ in $\bar X$. In particular, we believe that answer(s) to any of the following questions would be interesting:

\begin{problem}
Assume $\bar X$ is non-algebraic.
\begin{enumerate}
\item Find an explicit description of the (analytic) local ring $\sheaf_{\bar X, P_\infty}$ of $\bar X$ at $P_\infty$. 
\item In particular, describe when $C_\infty$ is singular at $P_\infty$ and the corresponding singularity type.
\end{enumerate}
\end{problem}  

\subsection{Organization}
In \cref{presection} we recall and introduce some preliminary notions and results necessary for the statement of the results of this article. In particular, we introduce {\em descending Puiseux series} (which in our setting are more convenient than usual Puiseux series for studying valuations centered at infinity), {\em key sequences}, and describe their relation with divisorial valuations centered at infinity. In \cref{structure-section} we present basic results about singularities of primitive compactifications and the curves at infinity, and in the case of algebraic primitive compactifications, their embedding into (weighted) projective spaces. The results of this section are mostly reformulation or simple applications of results from \cite{sub2-1,algebraicity}. \Cref{normal-section} is the technical heart of this work; here we define the normal forms of key sequences, and state their basic properties. The proof of these properties are differed to the appendices, essentially because of the length. \Cref{automorphic-section,moduli-section} consist of applications of normal forms. In \cref{automorphic-section} we establish the `rigidity' of $\cc^2$ in `most' primitive compactifications and compute their groups of automorphism. We use the description of these groups of automorphisms in \cref{moduli-section} to compute moduli spaces of primitive compactifications (modulo isomorphisms) and of curves in $\cc^2$ with one place at infinity (modulo automorphisms of $\cc^2$). In \cref{canonical-section} we compute the canonical divisor of primitive compactifications and apply it to characterize weighted projective spaces $\pp^2(1,1,q)$ in terms of log discrepancy and index, and to classify (Gorenstein) primitive compactifications with rational and elliptic singularities. The main tool in our computation of canonical divisor is a result of Kuo and Parusi\'nski \cite{kuo-parusinski-pol-arc} on multiplicities of Puiseux roots of polynomials. \Cref{key-appendix} collects some technical results on key forms that are used throughout the article. The remaining four appendices are devoted to the proof of the properties of normal forms stated in \cref{normal-section}. In particular, we study in \cref{theta-appendix} the effect of changes of the coefficients of a descending Puiseux series on the associated key forms, and in \cref{change-appendix} the effect of change of coordinates on $\cc^2$ on the coefficients of a descending Puiseux series. Finally, in \cref{normal-proof-1,normal-proof-2} we combine these results to prove the properties of key forms.  


\section{Preliminaries} \label{presection}

\subsection{Divisorial semidegrees and descending Puiseux series}
Let $\Xxy$ be a copy of $\pp^2$ such that $X := \cc^2$ is embedded into $\Xxy$ via the map $(x,y) \mapsto [1:x:y]$. 

 \label{key-section}
\begin{defn}[Divisorial discrete valuations] \label{divisorial-defn}
A discrete valuation on $\cc(x,y)$ is a map $\nu: \cc(x,y)\setminus\{0\} \to \zz$ such that for all $f,g \in \cc(x,y)\setminus \{0\}$,
\begin{itemize}
\item $\nu(f+g) \geq \min\{\nu(f), \nu(g)\}$,
\item $\nu(fg) = \nu(f) + \nu(g)$.
\end{itemize}
A discrete valuation $\nu$ on $\cc(x,y)$ is called {\em divisorial} iff there exists a normal algebraic surface $Y_\nu$ equipped with a birational map $\sigma: Y_\nu \to \Xxy$ and a curve $C_\nu$ on $Y_\nu$ such that for all non-zero $f \in \cc[x,y]$, $\nu(f)$ is the order of vanishing of $\sigma^*(f)$ along $C_\nu$. The {\em center} of $\nu$ on $\Xxy$ is $\sigma(C_\nu)$. $\nu$ is said to be {\em centered at infinity} (on $\Xxy$) iff the center of $\nu$ on $\Xxy$ is contained in $\Xxy \setminus X$; equivalently, $\nu$ is centered at infinity iff there is a non-zero polynomial $f \in \cc[x,y]$ such that $\nu(f) < 0$. 
\end{defn}

\begin{defn}[Divisorial semidegrees] \label{semi-defn}
A {\em divisorial semidegree} on $\cc[x,y]$ is a map $\delta : \cc(x,y)\setminus\{0\} \to \zz$ such that $-\delta$ is a divisorial discrete valuation centered at infinity. 
\end{defn}

\begin{defn}[Descending Puiseux series] \label{dpuiseuxnition}
The field of {\em descending Puiseux series} in $x$ is 
$$\dpsxc := \bigcup_{p=1}^\infty \cc((x^{-1/p})) = \left\{\sum_{j \leq k} a_j x^{j/p} : k,p \in \zz,\ p \geq 1 \right\},$$
where for each integer $p \geq 1$, $\cc((x^{-1/p}))$ denotes the field of Laurent series in $x^{-1/p}$. Let $\phi$ be a descending Puiseux series in $x$. The {\em polydromy order} (terminology taken from \cite{casas-alvero}) of $\phi$ is the smallest positive integer $p$ such that $\phi \in \cc((x^{-1/p}))$. For any $r \in \qq$, let us denote by $[\phi]_{>r}$ (resp.\ $[\phi]_{\geq r}$) sum of all terms of $\phi$ with order greater than (resp.\ greater than or equal to) $r$. Then the {\em Puiseux pairs} of $\phi$ are the unique sequence of pairs of relatively prime integers $(q_1, p_1), \ldots, (q_k,p_k)$ such that the polydromy order of $\phi$ is $p_1\cdots p_k$, and for all $j$, $1 \leq j \leq k$,
\begin{itemize}
\item $p_j \geq 2$,
\item $[\phi]_{>\frac{q_j}{p_1\cdots p_j}} \in \cc((x^{-\frac{1}{p_0\cdots p_{j-1}}}))$ (where we set $p_0 := 1$), and 
\item $[\phi]_{\geq \frac{q_j}{p_1\cdots p_j}} \not\in \cc((x^{-\frac{1}{p_0\cdots p_{j-1}}}))$.
\end{itemize}
The exponents $q_j/(p_1\cdots p_j)$, $1 \leq j \leq k$, are called the {\em characteristic exponents} of $\phi$. Let $\phi = \sum_{q \leq q_0} a_q x^{q/p}$, where $p$ is the polydromy order of $\phi$. Then the {\em conjugates} of $\phi$ are $\phi_j := \sum_{q \leq q_0} a_q \zeta^q x^{q/p}$, $1 \leq j \leq p$, where $\zeta$ is a primitive $p$-th root of unity. 
\end{defn}

The relation between descending Puiseux series and semidegrees is given by the following proposition, which is a reformulation of the corresponding result for Puiseux series and valuations \cite[Proposition 4.1]{favsson-tree}.

\begin{prop}[cf.\ {\cite[Proposition 4.1]{favsson-tree}}] \label{phi-delta}
Let $\delta$ be a divisorial semidegree on $\cc(x,y)$ such that $\delta(x) > 0$. Then there exists a {\em descending Puiseux polynomial} (i.e.\ a descending Puiseux series with finitely many terms) $\phi_\delta \in \dpsxc$ (unique up to conjugacy) and a unique rational number $r_\delta < \ord_x(\phi_\delta)$ such that for every polynomial $f \in \cc[x,y]$, 
\begin{align}
\delta(f) = p\deg_x\left( f(x,y)|_{y = \phi_\delta(x) + \xi x^{r_\delta}}\right), \label{phi-delta-defn}
\end{align}
where $\xi$ is an indeterminate and $p$ is the smallest positive integer such that all exponents of $x$ appearing in $\phi_\delta(x^p) + \xi x^{rp}$ are integers.
Conversely, given a Puiseux polynomial $\phi_\delta$ and a rational number $r_\delta < \ord_x(\phi_\delta)$, \eqref{phi-delta-defn} defines a divisorial semidegree $\delta$ on $\cc(x,y)$ such that $\delta(x) > 0$.  \qed
\end{prop}

\begin{defn} \label{generic-descending-defn}
If $\delta$, $\phi$ and $r$ are as in \cref{phi-delta}, then we say that $\tilde \phi_\delta(x,\xi):= \phi(x) + \xi x^{r}$ is the {\em generic descending Puiseux series} associated to $\delta$ and we call the integer $p$ from \eqref{phi-delta-defn} the {\em polydromy order} of $\tilde \phi_\delta$. Let the Puiseux pairs of $\phi$ be $(q_1, p_1), \ldots, (q_l,p_l)$. Then $p_1 \cdots p_l$ divides $p$. Let $p_{l+1} := p/(p_1 \cdots p_l)$. Then $r$ can be expressed as $q_{l+1}/p$, $q_{l+1} \in \zz$, $\gcd(q_{l+1}, p_{l+1}) = 1$. Then the {\em formal Puiseux pairs} (resp.\ {\em formal characteristic exponents}) of $\tilde \phi_\delta$ are $(q_j, p_j)$ (resp.\ $q_j/(p_1\cdots p_j)$), $1 \leq j \leq l+1$. Note that 
\begin{enumerate}
\item $\delta(x) = p_1 \cdots p_{l+1}$,
\item it is possible that $p_{l+1} = 1$, whereas $p_k \geq 2$ for $1 \leq k \leq l$. 
\end{enumerate}
\end{defn}

\begin{defn} \label{determining-defn}
Given a divisorial semidegree $\delta$ on $\cc[X]$, $X := \cc^2$, we say that $\delta$ is {\em primitive} iff there exists a (necessarily unique) normal analytic compactification $\bar X$ of $X$ such that $C_\infty := \bar X \setminus X$ is an irreducible curve and $\delta$ is the order of pole along $C_\infty$.
\end{defn}

The following result, which is an immediate corollary of \cite[Proposition 3.5]{sub2-1}, gives a connection between descending Puiseux series of a primitive semidegree with the geometry of the associated compactification.  

\begin{prop} \label{main-curve-prop}
Let $\bar X$ be a primitive compactification of $X$, $C_\infty := \bar X \setminus X$ be the curve at infinity, $\delta$ be the associated semidegree on $\cc[x,y]$ and let $\tilde \phi_\delta(x,\xi):= \phi_\delta(x) + \xi x^{r_\delta}$ be the generic descending Puiseux series associated to $\delta$. Then there is a unique point $P_\infty \in C_\infty$ such that for all $P \in C_\infty \setminus P_\infty$ and all $f \in \cc[x,y]\setminus \{0\}$, $P$ is on the curve (on $\bar X$) defined by $f$ iff there is a descending Puiseux root\footnote{A {\em descending Puiseux root} of $f$ is a descending Puiseux series $\phi(x)$ such that $f(x,\phi(x)) \equiv 0$. It follows from the standard theory of Puiseux series that if $\deg_y(f) = d$, then $f$ has $d$ descending Puiseux roots.} $\phi(x)$ of $f$ of the form 
\begin{align}
\phi(x) =  \phi_\delta(x) + c_P x^{r_\delta} + \ldt  \label{generic-expansion}
\end{align}
for some $c_P \in \cc$ (where $\ldt$ denotes terms with lower degree in $x$). \qed
\end{prop}

\subsection{Key sequences and key forms}

\begin{defn}[Key sequences]
\label{key-seqn}
A sequence $\vec\omega := (\omega_0, \ldots, \omega_{n+1})$, $n \in \zz_{\geq 0}$, of integers is called a {\em key sequence} if it has the following properties: 
\begin{enumerate}
\item \label{positive-1-property} $\omega_0 \geq 1$.
\item \label{gcd-1-property} Let $e_k := \gcd(|\omega_0|, \ldots, |\omega_k|)$, $0 \leq k \leq n+1$ and $\alpha_k := e_{k-1}/e_k$, $1 \leq k \leq n+1$. Then $e_{n+1} = 1$, and
\item \label{smaller-property}  $\omega_{k+1} < \alpha_k\omega_k$, $1 \leq k \leq n$.
\end{enumerate} 
Moreover, $\vec \omega$ is called {\em primitive} if $\omega_{n+1} > 0$ (or equivalently, $\omega_k > 0$ for all $k$, $0 \leq k \leq n+1$), and it is called {\em algebraic} if 
\begin{enumerate}
\setcounter{enumi}{3}
\item \label{algebraic-semigroup-property} $\alpha_k\omega_k \in \zz_{\geq 0} \langle \omega_0, \ldots, \omega_{k-1} \rangle$, $1 \leq k \leq n$.
\end{enumerate}
Finally, $\vec \omega$ is called {\em essential} if $\alpha_k \geq 2$ for $1 \leq k \leq n$. Note that
\begin{enumerate}[label=(\alph{enumi})]
\item Given an arbitrary key sequence $(\omega_0, \ldots, \omega_{n+1})$, it has an associated {\em essential subsequence} $(\omega_0, \omega_{i_1}, \ldots, \omega_{i_l}, \omega_{n+1})$ where $\{i_j\}$ is the collection of all $k$, $1 \leq k \leq n$, such that $\alpha_k \geq 2$.
\item If $\vec \omega$ is an algebraic key sequence, then its essential subsequence is also algebraic.
\end{enumerate} 
\end{defn}

\begin{rem}\label{unique-remark}
Let $\vec\omega := (\omega_0, \ldots, \omega_{n+1})$ be a key sequence. It is straightforward to see that property \ref{smaller-property} implies the following: for each $k$, $1 \leq k \leq n$, $\alpha_k\omega_k$ can be {\em uniquely} expressed in the form $\alpha_k\omega_k = \beta_{k,0}\omega_0 + \beta_{k,1} \omega_1 + \cdots + \beta_{k,k-1}\omega_{k-1}$, where $\beta_{k,j}$'s are integers such that $0 \leq \beta_{k,j} < \alpha_j$ for all $j \geq 1$. If $\vec \omega$ is in additional {\em algebraic}, then $\beta_{k,0}$'s of the preceding sentence are {\em non-negative}.
\end{rem}

\begin{defn}[Key forms]\label{key-form}
Let $\vec \omega := (\omega_0, \ldots, \omega_{n+1})$ be a key sequence and $\vec \theta := (\theta_1, \ldots, \theta_n) \in \ntorus$. The corresponding {\em key forms} are elements $g_0, \ldots, g_{n+1}$ in $\cc[x,x^{-1},y]$ defined as follows: $g_0 := x$, $g_1 := y$, and
\begin{align*}
g_{k+1} = g_k^{\alpha_k} - \theta_k g_0^{\beta_{k,0}} \cdots g_{k-1}^{\beta_{k,k-1}},\quad 1 \leq k \leq n,
\end{align*}
where $\alpha_k$'s and $\beta_{k,j}$'s are as in Remark \ref{unique-remark}. 
\end{defn}

Key forms are clear-cut analogues of the notion of {\em key polynomials} corresponding to discrete valuations introduced by MacLane \cite{maclane-key}. It follows from the standard theory of discrete valuations (as developed in e.g.\ \cite[Chapters 2, 4]{favsson-tree}) that there are pairwise correspondences among the following families:
\begin{enumerate}[label=(\Alph{enumi})]
\item \label{omega-theta} Pairs $(\vec \omega, \vec \theta)$ where $\vec \omega := (\omega_0, \ldots, \omega_{n+1})$ is a key sequence and $\vec \theta := (\theta_1, \ldots, \theta_n) \in \ntorus$. 
\item \label{delta} Divisorial semidegrees $\delta$ on $\cc[x,y]$ such that $\delta(x) > 0$.
\item \label{phi-r} Pairs $(\phi,r)$ where $\phi$ is a {\em descending Puiseux polynomial}, i.e.\ an element of $\dpsxc$ with finitely many terms, $r \in \qq$, $r < \ord_x(\phi)$.
\end{enumerate}
The correspondence between \ref{delta} and \ref{phi-r} is given by \cref{phi-delta}. The mapping $(\omega, \theta) \mapsto \deltaomegatheta$ defined in \cref{omega-theta-to-delta} below gives the correspondence \ref{omega-theta} $\to$ \ref{delta}. 

\begin{thm} \label{omega-theta-to-delta}
Let $(\vec \omega, \vec \theta)$ be as in \ref{omega-theta}, $y_1, \ldots, y_{n+1}$ be indeterminates and $\omega$ be the {\em weighted degree} on $S := \cc[x,x^{-1},y_1, \ldots, y_{n+1}]$ corresponding to weights $\omega_0$ for $x$ and $\omega_j$ for $y_j$, $0 \leq j \leq n+1$ (i.e.\ the value of $\omega$ on a polynomial is the maximum `weight' of its monomials). For every polynomial $f \in \cc[x,x^{-1},y]$, define
\begin{align}
\delta(f) := \min\{\omega(F): F(x,y_1, \ldots, y_{n+1}) \in S,\ F(g_0,g_1, \ldots, g_{n+1}) = f\}. \label{generating-eqn}
\end{align}
where $g_0, \ldots, g_{n+1}$ are key forms corresponding to $(\vec \omega, \vec \theta)$. Then $\delta$ is a divisorial semidegree on $\cc[x,y]$ such that $\delta(g_k) = \omega_k$, $0 \leq k \leq n+1$; in particular, $\delta(x) = \omega_0 > 0$. 
\end{thm}

\begin{proof}
Follows from \cite[Theorem 3.17]{algebraicity}. 
\end{proof}

Combining \cref{omega-theta-to-delta} with \cref{phi-delta} gives the correspondence \ref{omega-theta} $\to$ \ref{phi-r}. The correspondence \ref{phi-r} $\to$
 \ref{omega-theta} is described in detail in \cite[Algorithm 5.1]{algebraicity}; we summarize it below. 
 
 \begin{algorithm}[Construction of $(\vec \omega, \vec \theta)$ and the key forms from $(\phi,r)$]
 \label{phi-omega-algorithm} 
 Define $\tilde \phi(x, \xi) := \phi(x) + \xi x^r$. Set $\omega_0$ to be the polydromy order (\cref{generic-descending-defn}) of $\tilde \phi$, $\omega_1 := \omega_0\deg_x(\tilde \phi)$, $g_0 := x$, $g_1 := y$. We construct the rest of the elements inductively: assume $\omega_0, \ldots, \omega_k$, $\theta_1, \ldots, \theta_{k-1}$, $g_0, \ldots, g_k$ have been constructed for some $k \geq 1$ such that $\deg_x\left(g_{j}|_{y = \tilde \phi}\right) = \omega_j/\omega_0$, $0 \leq j \leq k$. Let $\tilde a_k$ be the coefficient of $x^{\omega_k/\omega_0}$ in $g_k|_{y = \tilde \phi}$. If $\tilde a_k\in \cc[\xi]\setminus \cc$, then stop. Otherwise there exist unique integers $\beta_{k,0}, \ldots, \beta_{k,k-1}$ as in \cref{unique-remark} and a unique $\theta_k \in \cc^*$ such that with $g_{k+1} := g_k^{\alpha_k} - \theta_k g_0^{\beta_{k,0}} \cdots g_{k-1}^{\beta_{k,k-1}}$, we have $\tilde \omega_{k+1} := \deg_x\left(g_{k+1}|_{y = \tilde \phi}\right) < \alpha_k \omega_k/\omega_0$, where $\alpha_k$ is defined as in \cref{key-seqn}. Define $\omega_{k+1}  := \omega_0\tilde \omega_{k+1}$, and repeat the process.
 \end{algorithm}
 
 \begin{example} 
In the table below each pair $p_j,q_j$ are relatively prime integers with $p_j > 0$:
 \begin{align*}
 \begin{array}{|p{2.3cm}|lBp{2.5cm}|l|p{3.2cm}|}
 \hline
 $\phi(x)$ & r & $\vec \omega$ & \vec \theta & \text{key forms} \\
 \hline 
 $0$ & q_1/p_1 &$(p_1, q_1)$ & \text{--} & $x, y$ \\
 \hline
 $ax^{q_1/p_1}$, $c \in \cc^* $ & q_2/(p_1p_2) & $(p_1p_2, q_1p_2,$ \newline $(p_1-1)q_1p_2 + q_2) $ & a^{p_1}
	 &$ x, y, y^{p_1} - a^{p_1}x^{q_1}$ \\
 \hline 
 $a_1x^{2/3} + a_2x^{-1},$  $a_1, a_2 \in \cc^*$ & -2 & $(3, 2, 1, -2)$ & (a_1^3, 3a_2)
	 &  $x, y,  y^3 - a_1^3x^2$, \newline $y^3 - a_1^3x^2 - 3a_2x^{-1}y^2$\\
\hline
 \end{array}
  \end{align*}
 \end{example} 
 
\begin{rem}\label{approximate-remark}
{\em $\delta$-sequences} of plane curves with one place at infinity are special cases of key sequences; e.g.\ the definition of $\delta$-sequences in \cite[p.\ 1116]{sathaye} is equivalent to the following definition: a sequence $\vec\omega := (\omega_0, \ldots, \omega_{n+1})$, $n \in \zz_{\geq 0}$, of integers is a $\delta$-sequence if it is a primitive algebraic key sequence such that 
\begin{itemize}
\item 
$\alpha_k \geq 2$ for $k = 2, \ldots, n$. In other words, the essential subsequence of $\vec \omega$ is either $\vec \omega$ or $(\omega_0, \omega_2, \omega_3, \ldots, \omega_{n+1})$. 
\item $\alpha_{n+1}\omega_{n+1} \in \zz_{\geq 0} \langle \omega_0, \ldots, \omega_n \rangle$, i.e.\ the identity from property \eqref{algebraic-semigroup-property} of primitive algebraic key sequences holds for $k = n+1$ as well. 
\end{itemize} 
Geometrically the relation between $\delta$-sequences of plane curves and key sequences of semidegrees can be seen in the following way: let $C = V(f) \subseteq \cc^2$ be a curve with one place at infinity, where $f \in \cc[x,y]$ is an irreducible polynomial. Let $C_\xi := V(f - \xi) \subseteq \cc^2$ for each $\xi \in \cc$. It is a classical result of Moh \cite{moh-analytic-irreducibility} that each $C_\xi$ has one place at infinity, call it $P_\xi$. Let $\delta_C$ be the semidegree on $\cc[x,y]$ that assigns to each $h \in \cc[x,y]$ the order of pole at $P_\xi$ of $h|_{C_\xi}$ for generic $\xi$. Let $\vec\omega := (\omega_0, \ldots, \omega_{n+1})$ be the key sequence of $\delta_C$ and $\vec \omega_e := (\omega_{i_0}, \ldots, \omega_{i_{l+1}})$ be its essential subsequence, where $0 = i_0 < \cdots < i_{l+1} = n+1$. Then $\omega_{i_{l+1}} = 0$, and the $\delta$-sequence of $C$ in $(x,y)$-coordinates is:
\begin{align}
\vec d 
	:=	\begin{cases}
			(\omega_{i_0}, \ldots, \omega_{i_l}) &\text{if}\ i_1 = 1,\\
			(\omega_{i_0}, \omega_1, \omega_{i_2}, \ldots, \omega_{i_l}) &\text{if}\ i_1 > 1.
		 \end{cases} \label{delta-from-key}
\end{align}
\end{rem}

\section{Basic structure of primitive compactifications of $\cc^2$} \label{structure-section}
\begin{notation}\label{omega-theta-defn}
Let $(x,y)$ be a fixed system of coordinates on $X := \cc^2$. Let $\vec \omega := (\omega_0, \ldots, \omega_{n+1})$ be a key sequence, $\vec \theta := (\theta_1, \ldots, \theta_n) \in \ntorus$, and $\deltaomegatheta$ be the associated semidegree on $\cc[x,y]$ as in \cref{omega-theta-to-delta}. In case $\deltaomegatheta$ is primitive (\cref{determining-defn}), we write $\xomegatheta$ for the corresponding primitive compactification of $X$.
\end{notation}

\begin{thm} \label{key-primitive-correspondence}
\mbox{}
\begin{enumerate}
\item $\deltaomegatheta$ is primitive iff $\vec \omega$ is primitive.
\item Assume $\deltaomegatheta$ is primitive. Then $\xomegatheta$ is algebraic iff $\vec \omega$ is algebraic.  
\end{enumerate}
\end{thm}

\begin{proof}
The first assertion is a reformulation of \cite[Corollary 6.3]{sub2-1}. The second assertion follows from \cite[Theorem 4.1]{algebraicity}. 
\end{proof}

\cite[Theorem 4.5]{sub2-1} gives an explicit description of dual graphs of minimal resolutions of singularities of primitive compactifications of $\cc^2$. \Cref{general-structure-thm} below collects some information about the singularities of primitive compactifications and the curves on infinity on these surfaces; it follows in a straightforward manner from an examination of these dual graphs. 

\begin{prop} \label{general-structure-thm}
Let $\vec\omega:= (\omega_0, \ldots, \omega_{n+1})$ be a primitive key sequence and $\vec \theta \in \ntorus$. Let $\bar X:= \xomegatheta$, $C_\infty$ be the curve at infinity on $\bar X$, and $P_\infty \in C_\infty$ be as in \cref{main-curve-prop}. 
\begin{enumerate}
\item $\bar X$ has at most two singular points. 
\begin{enumerate}
\item \label{case-projective} If $\bar X$ is isomorphic to a weighted projective surface of the form $\pp^2(1, p, q)$, with $1 \leq p \leq q$, then 
\begin{enumerate}
\item If $p = q$, then $\bar X \cong \pp^2$ and therefore $\bar X$ is nonsingular.
\item Otherwise let $p' := p/\gcd(p,q)$ and $q' := q/\gcd(p,q)$. If $p' = 1$, then $\bar X$ has only one singular point, which is a cyclic quotient singularity\footnote{Let $P$ be an isolated singular point on a surface $U$, and $a,b,c$ be positive integers. Then $U$ is said to have a {\em cyclic quotient singularity} at $P$ of {\em type $\frac{1}{a}(b,c)$} iff $P$ has a neighborhood in $U$ isomorphic to a neighborhood of the origin in the quotient of $\cc^2$ under the action of the group of $a$-th roots of unity given by $\zeta \cdot (u,v) = (\zeta^b u, \zeta^c v)$, where $\zeta$ is a primitive $a$-th root of unity.} of type $\frac{1}{q'}(1,1)$.
\item Otherwise $\bar X$ has two singular points, both cyclic quotient singularities, one of type $\frac{1}{q'}(1,p')$ and the other of type $\frac{1}{p'}(1,q')$. 
\end{enumerate}
\item Otherwise $\bar X$ has a non-cyclic quotient singularity at $P_\infty$. Let $ \omega' := \gcd(\omega_0,\ldots, \omega_n)$. 
\begin{enumerate}
\item If $\omega' = 1$, then $\bar X$ has no other singular points, i.e.\ $\sing(\bar X) = \{P_\infty\}$. 
\item Otherwise $\bar X$ has precisely one singular point $P_0$ other than $P_\infty$. $P_0$ is a cyclic quotient singularity of type $\frac{1}{\omega'}(1,\omega_{n+1})$. 
\end{enumerate}
\end{enumerate} 
\item \label{non-singinfinity} $C_\infty$ is non-singular off $P_\infty$. In particular $C_\infty \setminus P_\infty \cong \cc$. Moreover, in case \eqref{case-projective}, $C_\infty$ is non-singular, i.e.\ $C_\infty \cong \pp^1$. \qed
\end{enumerate}
\end{prop}

Our next result is about embeddings of primitive compactifications $\xomegatheta$ into projective spaces. If $\xomegatheta$ is not algebraic, then of course such embeddings do not exist. However, $\xomegatheta \setminus \{P_\infty\}$ (where $P_\infty$ is as in \cref{general-structure-thm}) remains a quasi-projective variety (since $\xomegatheta \setminus \{P_\infty\}$ is either non-singular or has only one rational singular point, and therefore is quasi-projective due to the algebraicity criterion of Artin \cite{artractability}). In particular, $\xomegatheta\setminus V(x)$ (where $V(x)$ is the closure of the $y$-axis) is quasi-projective (since \cref{main-curve-prop} implies that $P_\infty \in  V(x)$). The first part of the next proposition shows that $\xomegatheta\setminus V(x)$ is in fact an affine surface, with a closed embedding defined by the key forms of $\deltaomegatheta$. In the case that $\xomegatheta$ is algebraic, the embedding extends to all of $\xomegatheta$, and realizes it as a weighted complete intersection in a weighted projective space, whose defining equations can be explicitly described. 

\begin{prop}\label{projective-embedding}
Let $\vec\omega:= (\omega_0, \ldots, \omega_{n+1})$ be a key sequence and $\vec \theta \in \ntorus$. Let $g_0, \ldots, g_{n+1}$ be the key forms associated to $(\vec\omega,\vec \theta)$ and $\WP$ be the weighted projective space $\pp^{n+2}(1, \omega_0, \omega_1, \ldots, \omega_{n+1})$ with (weighted) homogeneous coordinates $[w:y_0: \cdots :y_{n+1}]$. 
\begin{enumerate}
\item \label{general-assertion} The map $\iomegatheta: X\setminus V(x) \into \WP\setminus V(y_0)$ given by
\begin{align*}
(x,y) \mapsto [1: g_0(x,y): g_1(x,y): \cdots : g_{n+1}(x,y)]
\end{align*}
induces an open embedding of $X \setminus V(x)$ into the closure $Y$ in $\WP\setminus V(y_0)$ of its image. The complement $C$ of $X\setminus V(x)$ in $Y$ is isomorphic to $\cc$ and the order of pole along $C$ is precisely $\deltaomegatheta$. If $\vec \omega$ is primitive, then $Y \cong \xomegatheta \setminus V(x)$.
\item If $\vec\omega$ is primitive algebraic, then  $\iomegatheta$ induces an isomorphism of $\xomegatheta$ with the subvariety of $\WP$ defined by weighted homogeneous polynomials $G_k$, $1 \leq k \leq n$, given by
\begin{gather} \label{projective-definition}
G_k := w^{\alpha_k\omega_k - \omega_{k+1}}y_{k+1} - \left( y_k^{\alpha_k} - \theta_k \prod_{j=0}^{k-1}y_j^{\beta_{k,j}}\right)
\end{gather}
where $\alpha_k$'s and $\beta_{k,j}$'s are as in Remark \ref{unique-remark}. 
\end{enumerate}

\end{prop}

\begin{proof}
The first assertion is a consequence of \cref{basis-lemma}, and the second of \cite[Theorem 4.1]{algebraicity}. 
\end{proof}

The curve at infinity on primitive algebraic compactifications can be completely described as well: 

\begin{cor} \label{projllary}
Assume $\vec \omega$ is primitive algebraic. Identify $\xomegatheta$ with the subvariety of $\WP$ from \cref{projective-embedding}. Let $C_\infty := \xomegatheta \setminus X$ be the curve at infinity and $P_\infty$ (resp.\ $P_0$) be the point on $C_\infty$ with coordinates $[0: \cdots :0:1]$ (resp.\ $[0:1: \bar \theta_1: \cdots : \bar \theta_n : 0]$, where $\bar \theta_k$ is an $\alpha_k$-th root of $\theta_k$, $1 \leq k \leq n$). Then
\begin{enumerate}
\item \label{curve-1} Let $S$ be the subsemigroup of $\zz^2$ generated by $\{(\omega_k,0): 0 \leq k \leq n\} \cup \{(0,\omega_{n+1})\}$. Then $C_\infty \cong \proj \cc[S]$, where $\cc[S]$ is the semigroup algebra generated by $S$, and the grading in $\cc[S]$ is induced by the sum of coordinates of elements in $S$.  
\item \label{curve-singularity} Let $\tilde S:= \zz_{\geq 0} \langle \alpha_{n+1}\omega_{n+1} \rangle \cap \zz_{\geq 0} \langle \omega_0, \ldots, \omega_n \rangle$. Then $\cc[C_\infty \setminus P_0] \cong \cc[\tilde S]$,  In particular,
\begin{enumerate}
\item  $C_\infty$ has at worst a (non-normal) {\em toric} singularity at $P_\infty$;
\item \label{curve-non-singularity} $C_\infty$ is non-singular iff $\alpha_{n+1}\omega_{n+1} \in \zz_{\geq 0} \langle \omega_0, \ldots, \omega_n \rangle$, i.e.\ iff the essential subsequence of $\vec \omega$ is a {\em $\delta$-sequence} (\cref{approximate-remark}).  
\end{enumerate}

\end{enumerate}
\end{cor} 

\begin{rem} \label{curve-remark}
\mbox{}
\begin{defnlist}
\item $P_0$ and $P_\infty$ of \cref{projllary} are the same as $P_0$ and $P_\infty$ from \cref{general-structure-thm}. 
\item \label{symmetric-remark} Assertion \eqref{curve-non-singularity} of \cref{projllary}, assertion \eqref{delta-basis} of \cref{basis-lemma} and \cite[Proposition 2.1]{herzog} implies that if $C_\infty$ is non-singular, then the semigroup of poles of polynomials along $C_\infty$ is {\em symmetric} (see \cref{symmetric-footnote} for the definition of symmetric semigroups). 
\end{defnlist}
\end{rem}

\begin{proof}[Proof of \cref{projllary}]
Since $C_\infty = V(w) \cap \xomegatheta$, \cref{projective-embedding} implies that 
\begin{align*}
C_\infty &\cong \proj \cc[w, y_0, \ldots, y_{n+1}]/\langle w, \bar G_1, \ldots, \bar G_n \rangle \cong \proj \cc[y_0, \ldots, y_{n+1}]/\langle \bar G_1, \ldots, \bar G_n \rangle,
\end{align*}
where $\bar G_k := y_k^{\alpha_k} - \theta_k \prod_{j=0}^{k-1}y_j^{\beta_{k,j}}$, $1 \leq k \leq n$. \cite[Lemma B.1]{algebraicity} then implies that there is an isomorphism of graded $\cc$-algebras of the form
\begin{align*}
\cc[y_0, \ldots, y_{n+1}]/\langle \bar G_1, \ldots, \bar G_n \rangle \cong \cc[t^{\omega_0}, \ldots, t^{\omega_n}, y_{n+1}],
\end{align*} 	
where $t$ is an indeterminate and the grading on the right hand side is given by assigning the degrees of $t$ and $y_{n+1}$ to be respectively $1$ and $\omega_{n+1}$. Moreover, the isomorphism maps $y_k \mapsto t^{\omega_k}$ for $0 \leq k \leq n$, and $y_{n+1} \mapsto y_{n+1}$. This immediately implies assertion \eqref{curve-1}. Moreover, since $P_0 = C_\infty \cap V(y_{n+1})$, it follows that 
\begin{align*}
\cc[C_\infty \setminus P_0] \cong \cc\left[\frac{t^{\sum \beta_k \omega_k}}{y_{n+1}^{\beta_{n+1}}}: \beta_k \geq 0\ \text{for all}\ k,\ \sum_{k=0}^n \beta_k \omega_k = \beta_{n+1}\omega_{n+1}\right].
\end{align*}
Assertion \ref{curve-singularity} now follows from the definition of $\alpha_{n+1}$.
\end{proof}

The {\em index} of a $\qq$-Cartier divisor $D$ is the smallest positive integer $m$ such that $mD$ is Cartier. The following result describes the index of the divisor at infinity on a primitive algebraic compactification. 

\begin{prop} \label{index-prop}
Let $\vec \omega$ and $C_\infty$ be as in \cref{projllary}. Then the index of $[C_\infty]$ is $\alpha_{n+1}\omega_{n+1}$, where $\alpha_{n+1} := \gcd(\omega_0, \ldots, \omega_n)$ and $[C_\infty]$ is the Weil divisor corresponding to $C_\infty$. 
\end{prop}

\begin{proof}
We have to show that $m[C_\infty]$ is a Cartier divisor iff $m$ is divisible by $\alpha_{n+1}\omega_{n+1}$. At first we show the ($\Leftarrow$) implication. Since $\alpha_{n+1} = \gcd(\omega_0,\ldots, \omega_n)$, it follows that $|m|  = k\omega_{n+1} = m_0\omega_0 - \sum_{j=1}^n m_j\omega_j$ for non-negative integers $k, m_0, \ldots, m_n$. Consider the set up of \cref{projllary}. Then it is straightforward to see that $1/g_{n+1}^k$ and $(\prod_{j=1}^n g_j^{m_j})/g_0^{m_0}$ defines the Cartier divisor $|m|[C_\infty]$ respectively near $P_\infty$ and $P_0$. Since $\xomegatheta \setminus \{P_0, P_\infty\}$ is non-singular, this complete the proof of ($\Leftarrow$) direction.\\

Now we prove the ($\im$) implication. Let $h_\infty = h_{\infty,1}/h_{\infty,2}$ define $m[C_\infty]$ near $P_\infty$, with $h_{\infty,1}, h_{\infty,2} \in \cc[x,y]$. Then for each $i$, the closure in $\bar X$ of the curve on $X$ defined by $h_{\infty, i} = 0$ does {\em not} go through $P_\infty$. \cref{main-curve-prop} and assertion \eqref{last-factorization} of \cref{last-curve-prop} then imply that $\delta(h_{\infty, i})$ is a multiple of $\omega_{n+1}$ for each $i$, where $\delta$ is the order of pole along $C_\infty$. It follows that $\omega_{n+1}$ divides $m= \delta(h_{\infty,2}) - \delta(h_{\infty,1})$. Now assume $h_0 = h_{01}/h_{02}$ defines $m[C_\infty]$ near $P_0$, with $h_{01}, h_{02} \in \cc[x,y]$. Fix $i$, $1 \leq i \leq 2$. Let $h_{0i} = \sum a_{\beta} g_0^{\beta_0} g_1^{\beta_1} \cdots g_{n+1}^{\beta_{n+1}}$ be the expansion of $h_{0i}$ in terms of the basis $\scrB$ of \cref{basis-lemma}. Since $P_0$ is not on the closure on $\bar X$ of the curve $h_{0i}=0$, it follows that among all $\beta$ such that $a_\beta \neq 0$ and $\sum \omega_j\beta_j = \delta(h_{0i})$, there must be some $\beta$ such that $\beta_{n+1} = 0$. Since $\alpha_{n+1}$ divides $\omega_j$ for every $j$, $0 \leq j \leq n$, it follows that $\alpha_{n+1}$ divides $\delta(h_{0i})$. Consequently, $\alpha_{n+1}$ also divides $m= \delta(h_{02}) - \delta(h_{01})$. Since $\gcd(\alpha_{n+1}, \omega_{n+1}) = 1$ (by definition of key sequences), it follows that $\alpha_{n+1}\omega_{n+1}$ divides $m$, as required.
\end{proof}

\section{Normal forms of key sequences} \label{normal-section}

In this section we introduce normal forms of key sequences associated with divisorial semidegrees, and state their basic properties. The proofs of these properties are in \cref{normal-proof-1,normal-proof-2}. We start with some examples which motivate the definition of normal forms. 

\subsection{Motivation and definition of normal forms}

\begin{example} \label{exchange-example}
Let $\vec \omega = (p,q)$ for some $p,q \in \zz_{> 0}$. Then $\delta := \deltaomegatheta$ is the weighted degree on $\cc[x,y]$ corresponding to weights $p,q$ for $x$ and $y$ respectively, and the associated generic descending Puiseux series is $\tilde \phi(x,\xi) = \xi x^{q/p}$. 
After the change of coordinate $(x,y) \mapsto (y,x)$, the  generic descending Puiseux series of $\delta$ becomes $\xi x^{p/q}$ and the key sequence becomes $(q,p)$.  
\end{example}

\begin{example} \label{change-example-0}
Let $\delta$ be the semidegree on $\cc[x,y]$ defined by $\delta(f) := 5\deg_x(f(x,\tilde \phi))$, where $$\tilde \phi:= x^{5} +2x^{4} + 3x^{3/2} + \xi x^{-1}$$
\Cref{phi-omega-algorithm} shows that the key forms of $\delta$ are $x$, $y$, $y - x^5$, $y - x^5 - 2x^4$,  $(y - x^5 - 2x^4)^2 - 9x^3$, and the key sequence is $\vec \omega := (2,10, 8,3,1)$. However, after the change of coordinate $(x,y) \mapsto (x,y- x^5-2x^4)$, the generic descending Puiseux series of $\delta$ becomes $3x^{3/2} + \xi x^{-1}$ with key forms $x,y, y^2 - 9x^3$ and key sequence $(2, 3, 1)$. 
\end{example}

\begin{example} \label{change-example-1}
Let $\delta$ be the semidegree on $\cc[x,y]$ defined by $\delta(f) := 5\deg_x(f(x,\tilde \phi))$, where $$\tilde \phi:= x^{3/5} +2x^{2/5} + 3x^{1/5} + \xi$$
\Cref{phi-omega-algorithm} shows that the key forms of $\delta$ are $x$, $y$, $y^5 - x^3$, $y^5- x^3 - 10xy^3$,  $y^5- x^3 - 10xy^3 + 5x^2y$, and the key sequence is $\vec \omega := (5,3,14,13,12)$. The table below shows the result of substitution $y = \tilde \phi(x)$ into the key forms; the term $\ldt$ denotes terms with lower degree in $x$. 
\begin{align*}
\begin{array}{|l|l|}
\hline
g(x,y) & g(x,y)|_ {y = \tilde \phi}\\
\hline 
y &  x^{3/5} +2x^{2/5} + 3x^{1/5} + \xi\\
y^5 - x^3 & 10x^{14/5} + 55x^{13/5} + 5\xi x^{12/5} + \ldt \\
y^5 - x^3 - 10xy^3 & -5x^{13/5} + 5\xi x^{12/5} + \ldt \\
y^5 - x^3 - 10xy^3 + 5x^2y &5\xi x^{12/5} + \ldt \\
\hline
\end{array}
\end{align*}
Now consider a change of coordinates of the form $(x,y) \mapsto (x-sy, y)$, $s \in \cc$. It is not hard to see that this converts the generic descending Puiseux series into 
\begin{align*}
\tilde \phi_s
	&= x^{3/5} +2x^{2/5} + (3 + 3s/5)x^{1/5} + \xi
\end{align*}
It is straightforward to see that the terms with highest $x$-degree in $y|_{y = \tilde \phi_s}$ and $(y^5 - x^3)|_{y = \tilde \phi_s}$ are the same as those in $y|_{y = \tilde \phi}$ and $(y^5 - x^3)|_{y= \tilde \phi}$ respectively, but the term with highest $x$-degree in $(y^5 - x^3 - 10xy^3)|_{y = \tilde \phi_s}$ is $(-5 + 3s)x^{13/5}$. \Cref{phi-omega-algorithm} then implies that the key forms and the key sequence of $\delta$ in the new coordinates behave differently depending on whether $-5 + 3s$ is zero or not. More precisely, 
\begin{itemize}
\item If $s \neq 5/3$, then the key forms of $\delta$ in the new coordinates are $x$, $y$, $y^5 - x^3$, $y^5- x^3 - 10xy^3$,  $y^5- x^3 - 10xy^3 -(-5 + 3s)x^2y$ and the key sequence is $(5,3,14,13,12)$. 
\item If $s = 5/3$, then the key forms of $\delta$ in the new coordinates are $x$, $y$, $y^5 - x^3$, $y^5- x^3 - 10xy^3$, and the key sequence is $(5,3,14,12)$.
\end{itemize}
\end{example}

Let $\delta$ be a divisorial semidegree on $\cc[x,y]$ with associated key sequence $(\omega_0, \ldots, \omega_{n+1})$. \Cref{exchange-example} shows that after a change of coordinate we may assume that
\begin{prooflist} 
\item \label{0>=1}  $\omega_0 \geq \omega_1$.
\end{prooflist}
Moreover \cref{change-example-0} suggests (and the discussion in \cref{normal-proof-1} shows) that 
\begin{prooflist} [resume]
\item \label{non-multiple}  if $n \geq 1$, then it is possible to ensure that neither of $\omega_0$ and $\omega_1$ is a (non-negative) integer multiple of the other.
\end{prooflist}
Finally, \cref{change-example-1} suggests that sometimes it is possible to kill off some interior terms of the key sequence without changing $\omega_0$ or $\omega_1$; we now precisely describe those terms which can be eliminated in this way.\\

Let $\vec \omega := (\omega_0, \ldots, \omega_{n+1})$ be a key sequence, and $\vec \omega_e := (\omega_{i_0}, \ldots, \omega_{i_{l+1}})$, where $0 = i_0 < i_1 < \cdots < i_{l+1} = n+1$, be the essential subsequence of $\vec \omega$. Let $\vec \theta \in \ntorus$ and $\deltaomegatheta$ be the semidegree associated to $(\vec \omega, \vec \theta)$. Identity \eqref{omega-and-chi-0} implies that the formal characteristic exponents of the generic descending Puiseux series $\tildephiomegatheta$ of $\deltaomegatheta$ are
\begin{align}
\chi_j := \frac{1}{\omega_0}(\omega_{i_j} - \sum_{k = 1}^{j-1}(\alpha_{i_k} - 1)\omega_{i_k}),
							\quad 1 \leq j \leq l+1. \label{beta-i-j}
\end{align}
where $\alpha_1, \ldots, \alpha_{n+1}$ are as in \cref{key-seqn}. Let 
\begin{align}
\scrE_{\vec \omega} 
	:=
	\begin{cases}
	 \{k\frac{\omega_1}{\omega_0}-1: k \in \zz,\  \max\{0, (\chi_{l+1}+1)\frac{\omega_0}{\omega_1}\} < k < \frac{\omega_0}{\omega_1} + 1\} \cup\{0\}
		 & \text{if}\ \omega_1 > 0,\\
	 \{k\frac{\omega_1}{\omega_0}-1: k \in \zz,\  0 < k < (\chi_{l+1}+1)\frac{\omega_0}{\omega_1}\}
	 		 & \text{if}\ \omega_1 < 0. 
	\end{cases} \label{exponents-omega}
\end{align}
We show in \cref{normal-proof-1} that 
\begin{prooflist}[resume]
\item Under conditions \ref{0>=1} and \ref{non-multiple}, $\scrE_{\vec \omega}$ consists of precisely those $\beta$ such that the coefficient $c_\beta$ of $x^\beta$ in $\tildephiomegatheta$ can be independently varied by changes of coordinates of $\cc[x,y]$ without changing $\omega_0$ or $\omega_1$.\\
\end{prooflist}


Let $\beta \in \qq$. We now describe the effect of changing $c_\beta$ on the key sequence and key forms. Let
\begin{align}
\hat k(\beta) 
	&:= 
	\begin{cases}
		0 &\text{if}\ \beta \geq \chi_1, \\
		\max\{k: 1 \leq k \leq l+1,\ \beta <  \chi_k\} & \text{otherwise}.
	\end{cases} 
	\label{hat-k-beta} \\
\hat \omega_\beta	
	&:= \omega_0 \beta + \sum_{j=1}^{\hat k(\beta)} (\alpha_{i_j}-1)\omega_{i_j} 
	\label{hat-omega-beta} \\
\hat I_\beta 
	&= 
	\begin{cases}
	\{i: i_{\hat k(\beta)} < i < i_{\hat k(\beta) + 1}\} 
		& \text{if}\ \hat k(\beta) \leq l\\
	\emptyset 
		&\text{if}\ \hat  k(\beta) = l+1.
	\end{cases}	
\end{align}
Note that $\hat \omega_{\chi_j} = \omega_{i_j}$, $1 \leq j \leq l+1$. \Cref{c-thm} gives the following interpretation for $\hat \omega_\beta$ and $\hat I_\beta$: pick $\beta$ distinct from each of the $\chi_j$'s. For each $c \in \cc$, let $\tilde \psi_c := \tildephiomegatheta + cx^\beta$ and $\delta_c$ be the corresponding semidegree. 
%
Pick the largest integer $\hat i_\beta$ such that the key forms $g_0, \ldots, g_{\hat i_\beta}$ remain unchanged for all values of $c$. Then $\hat i_\beta \in \hat I_\beta$ and $\hat \omega_\beta = \delta_c(g_{\hat i_\beta})$ for generic $c \in \cc$.

\begin{example}[\Cref{change-example-1} continued]
In the situation of \cref{change-example-1}, $\hat i_{1/5} = 3$, $\hat I_\beta = \{2,3\}$, $\hat \omega_{1/5} = 13$. 
\end{example} 
If $\hat \omega_\beta = \omega_{\hat i_\beta}$, then \cref{c-thm} implies that
\begin{prooflist}[resume]
\item \label{dropping-key} as in \cref{change-example-1}, it is possible to ensure that $\omega_{\hat i_\beta}$ is less than $\hat \omega_\beta$ by changing the value of $c_\beta$ while keeping fixed all $c_{\beta'}$ with $\beta' > \beta$. Moreover, it turns out that the property in the preceding sentence holds iff $\omega_i \neq \hat \omega_\beta$ for each $i \in \hat I_\beta$ (\cref{c-lemma}).\\
\end{prooflist}

Observations \ref{0>=1}--\ref{dropping-key} suggest the following definition of normal forms:
\begin{defn} \label{normal-key-defn}
We say that a key sequence $\vec \omega = (\omega_0, \ldots, \omega_{n+1})$ is in the {\em normal form} if it satisfies one of the following (mutually exclusive) conditions: 
\begin{enumerate}[label = (N\arabic{enumi}), ref=N\arabic{enumi}]
\setcounter{enumi}{-1}

\item \label{trivially-normal} 
\begin{enumerate} 
\item\label{trivially-zero-length} $n = 0$.
\item \label{trivially-omega} $\omega_0 \geq \omega_1$. 
\end{enumerate}
\item \label{generally-normal}
\begin{enumerate}
\item \label{generaly->=1} $n \geq 1$.
\item \label{generally-less} $\omega_0 > \omega_1$.
\item \label{generally-alpha} $\frac{\omega_1}{\omega_0} \not\in \{\frac{1}{k}: k \in \zz,\ k \geq 1\} \cup \{0\}$. 
\item \label{generally-omega} For each $\beta \in \scrE_{\vec \omega}$, there does not exist $i \in \hat I_\beta$ such that $\omega_i = \hat \omega_\beta$. 
\end{enumerate}
\end{enumerate}
\end{defn}

\subsection{Basic properties of normal forms}
The results in this section state the fundamental properties of normal forms. The proofs of \cref{normal-thm-0,normal-morphism} are in \cref{normal-proof-1} and the proof of \cref{preserving-thm} is in \cref{normal-proof-2}. 

\begin{thm} \label{normal-thm-0}
Let $X := \cc^2$ and $\delta$ be a divisorial semidegree on $\cc[X]$. Then 
\begin{enumerate}
\item \label{existence-assertion} There exist polynomial coordinates $(x,y)$ on $X$ such that $\delta(x) > 0$ and the key sequence $\vec \omega$  of $\delta$ with respect to $(x,y)$-coordinates is in the normal form. 
\item \label{unique-assertion} If $(x',y')$ is another system of coordinates on $\cc[X]$ such that $\delta(x') > 0$ and  the key sequence $\vec \omega'$  of $\delta$ with respect to $(x',y')$-coordinates is in the normal form, then $\vec \omega' = \vec \omega$. 
\end{enumerate}
\end{thm}


Let $X := \cc^2$ and $\delta$ be a divisorial semidegree on $\cc[X]$. Pick a system of coordinates $(x,y)$ on $X$ such that $\delta(x)  > 0$ and the corresponding  key sequence $\vec \omega := (\omega_0, \ldots, \omega_{n+1})$ of $\delta$ is in normal form. The following theorem describes the changes of coordinates on $X$ which preserves the key sequence of $\delta$. Let $F: = (F_1, F_2): X \to X$ be an isomorphism. Set $(x',y') := (F_1(x,y), F_2(x,y))$. Assume $\delta(x') > 0$. Let $\vec \omega'$ be the key sequence of $\delta$ with respect to $(x',y')$-coordinates.  We write $\tilde \phi_\delta(x,\xi)$ and $\tilde \psi_\delta(x',\xi)$ for the generic descending Puiseux series of $\delta$ respectively in $(x,y)$ and $(x',y')$ coordinates. 

\begin{thm} \label{normal-morphism}
Assume $\vec \omega' = \vec \omega$. Then
\begin{enumerate}
\item \label{trivial-F} If $\omega_0 = \omega_1$, then $F$ is an affine automorphism. 
\item \label{trivial-F-zero} If $\omega_1 = 0$, then $F$ is an automorphism of the form $(x,y) \mapsto (ax+f(y),by + c)$ for some $a,b \in \cc^*$, $c\in \cc$, and $f(y) \in \cc[y]$. 
\item \label{positive-F} If $\omega_0 > \omega_1 > 0$, then $F$ is of the form below:  
\begin{align*}
&F : (x,y) \mapsto (\bar a x + f(y), \bar b y + c), \ \text{where}\\
&\bar a, \bar b \in \cc^*, \ 
	c = \begin{cases}
			\text{an arbitrary element in \cc} & \text{if $\chi_{l+1} \geq 0$,} \\
			0 & \text{otherwise.}
		 \end{cases}\\
& f(y) \in \cc[y],\ 
\deg(f) \leq \frac{\omega_0}{\omega_1}(\chi_{l+1} +1)-1
\end{align*}
where $\chi_{l+1}$ is the last formal characteristic exponent of $\tilde \phi_\delta(x,\xi)$; see \eqref{beta-i-j}.
\item \label{negative-F} If $\omega_0 > 0 > \omega_1$, then $F$ is of the form below:  
\begin{align*}
&F : (x,y) \mapsto (\bar a x + f(y), \bar b y), \ \text{where}\\
&\bar a, \bar b \in \cc^*,\ f(y) \in \cc[y],\ 
\ord(f) \geq \frac{\omega_0}{\omega_1}(\chi_{l+1} +1)  - 1.
\end{align*}
\item \label{unique-puiseux} Let $\tilde \phi_\delta(x,\xi) = \phi(x) + \xi x^r = \sum_\beta a_\beta x^{\beta} + \xi x^r$. Then 
\begin{enumerate}
\item $ \tilde \psi_\delta(x',\xi) = b\sum_\beta a_\beta a^{-\bar \omega_0\beta} x'^{\beta} + \xi x'^{r}$ for some $a,b\in \cc^*$, where $\bar \omega_0$ is the polydromy order of $\phi(x)$.
\item \label{unique-puiseux-b} If $F$ is as in \eqref{positive-F} or \eqref{negative-F}, then $\bar a = a^{\bar \omega_0}$ and $\bar b = b$. 
\end{enumerate}
\end{enumerate}
\end{thm}

\begin{rem} 
Given a fixed divisorial semidegree $\delta$ on $\cc[x,y]$, it can be shown that among all key sequences $\vec \omega$ of $\delta$ with respect to different coordinate systems on $\cc[x,y]$, $\omega_0$ is the minimum when $\vec \omega$ is in the normal form. 
\end{rem}

The final result of this section describes the polynomial automorphisms which {\em preserves} a divisorial semidegree; given an automorphism $F$ of $X$, we say that {\em $F$ preserves $\delta$}, or that $F^*(\delta) = \delta$ iff $\delta(f) = \delta(f \circ F)$ for all $f \in \cc[X]$.

\begin{thm} \label{preserving-thm}
Let $X$, $\delta$, $\vec \omega$ be as in \cref{normal-morphism} (in particular, $\vec \omega$ is in the normal form) and let $F : X \to X$ be an automorphism. 
\begin{enumerate}
\item \label{n=0} If $n = 0$, then $F^*( \delta) = \delta$ iff $F$ is as in one of assertions \eqref{trivial-F}--\eqref{negative-F} of \cref{normal-morphism}. 
\item \label{n>=1} If $n \geq 1$, then $F^*( \delta) = \delta$ iff $F$ is as in assertions \eqref{positive-F} or \eqref{negative-F} of \cref{normal-morphism} subject to the following additional constraints on $\bar a$ and $\bar b$: let $\bar\omega_k := \omega_{k}/\alpha_{n+1}$, $0 \leq k \leq n$, and $\baromegastar_k := \alpha_1 \bar \omega_1 + \sum_{j=2}^{k-1}(\alpha_j-1)\bar \omega_j - \bar \omega_k$, $2 \leq k \leq n$, where $\alpha_1, \ldots, \alpha_{n+1}$ are as in Definition \ref{key-seqn}. Set $\baromegastar := \gcd(\baromegastar_2, \ldots, \baromegastar_n)$ (note that $\baromegastar$ is defined only if $n \geq 2$). Then $\bar a = a^{\bar \omega_0}$ and $\bar b = a^{\bar \omega_1}$, where
\begin{align*}
a &= \begin{cases}
		\text{an arbitrary element of $\cc^*$} 	& \text{if $n = 1$,}\\	
		\text{an $\baromegastar$-th root of	unity}	& \text{if $n \geq 2$.}
	 \end{cases} 
\end{align*}
\end{enumerate}
\end{thm}

\section{Automorphisms of primitive compactifications} \label{automorphic-section}

Our first application of normal forms is the `rigidity' of $\cc^2$ in a primitive compactification:

\begin{prop} \label{rigid-prop}
Let $\bar X$ be a primitive compactification of $X := \cc^2$. Let $U$ be an (open) subset of $\bar X$ isomorphic to $\cc^2$.
\begin{enumerate}
\item \label{X-bar-general} If $\bar X \not\cong \pp^2(1,1,q)$ for any integer $q$, then $U = X$, i.e.\ there exists only one open subset of $\bar X$ isomorphic to $\cc^2$.
\item \label{X-bar-proj-1-q} Assume $\bar X \cong \pp^2(1,1,q)$, $q \geq 1$, with weighted homogeneous coordinates $[z:y:x]$. 
\begin{enumerate}
\item \label{X-bar-proj-0} If $q = 1$, then $U = \bar X \setminus V( ax + by +cz)$ for some $(a,b,c) \in \cc^3\setminus\{(0,0,0)\}$. 
\item \label{X-bar-proj-1} If $q > 1$, then $U = \bar X \setminus V( by +cz)$ for some $(b,c) \in \cc^2\setminus\{(0,0)\}$. 
\end{enumerate} 
\end{enumerate}
\end{prop}

\begin{proof}
Let $X_0 := \bar X \setminus \{P_\infty\}$, where $P_\infty$ is as in \cref{general-structure-thm}. Then $X_0$ is a quasi-projective variety (see the discussion following \cref{general-structure-thm}). We start with an (obvious!) observation: 
\begin{align} \label{degree-group-generator}
\parbox{.66\textwidth}{every irreducible curve on $X_0$ is linearly equivalent (as a Weil divisor) to an integer multiple of $C_\infty \cap X_0$, where $C_\infty := \bar X \setminus X$.}
\end{align}
Let $\delta$ be the semidegree on $\cc[X]$ corresponding to $C_\infty := \bar X \setminus X$. Assume there exists an open subset $U \subseteq \bar X$ such that $U \cong X$, but $U \neq X$. We will show that assertion \eqref{X-bar-proj-1-q} of the proposition holds.\\

Indeed, under our assumption $C := \bar X \setminus U$ is the closure (in $\bar X$) of an irreducible curve on $X$ defined by some $h \in \cc[x,y]$. Since $C$ is the {\em curve at infinity} on $\bar X$ with respect to $U$, assertion \eqref{non-singinfinity} of \cref{general-structure-thm} implies that 
\begin{enumerate}[label=(\roman{enumi})]
\item \label{C-non-singular} $C \setminus \{P_\infty\}$ is non-singular; in particular, $C \cap X$ is a non-singular rational curve.
\end{enumerate} 
On the other hand, observation \eqref{degree-group-generator} applied to $U$ implies  that $C_\infty \cap X_0$ is linearly equivalent to an integer multiple of $C \cap X_0$. This implies that $\delta(h) = 1$ (since $\Div(h) = C - \delta(h) C_\infty$). \cref{delta-1-cor} implies that 
\begin{enumerate}[label=(\roman{enumi})]
\setcounter{enumi}{1}
\item \label{h-linear} $h$ is a linear combination of some key forms of $\delta$ and constants, and
\item \label{h-one-place} the curve $h=0$ has only one place at infinity.
\end{enumerate} 
Combining \ref{C-non-singular} and \ref{h-one-place} yields:
\begin{enumerate}[label=(\roman{enumi})]
\setcounter{enumi}{3}
\item \label{c-c} $C\cap X \cong \cc$.
\end{enumerate} 
Choose coordinates $(x,y)$ on $X$ such that the key sequence $\vec \omega = (\omega_0, \ldots, \omega_{n+1})$ of $\delta$ is in the normal form. Observation \ref{c-c} and the Abhyankar-Moh-Suzuki theorem \cite{abhya-moh-line} imply that
\begin{enumerate}[label=(\roman{enumi})]
\setcounter{enumi}{4}
\item \label{h-x-y} either $h$ is a polynomial of degree $1$, or
\item \label{division} one of the integers among $\{\deg_x(h), \deg_y(h)\}$ divides the other.
\end{enumerate} 

\begin{proclaim} \label{wt-proj-claim}
\eqref{trivially-normal} holds with $\omega_1 = 1$. In particular, $\bar X \cong \pp^2(1,1,\omega_0)$.
\end{proclaim}
\begin{proof}
If \ref{h-x-y} occurs, then since $\delta(h) = 1$, we have either $\omega_0 = \delta(x) = 1$ or $\omega_1 = \delta(y) = 1$. An inspection of normal forms shows that only possibility is \eqref{trivially-normal} with $\omega_1 = 1$. This implies that $\bar X \cong \pp^2(1,1,\omega_0)$. On the other hand, if \ref{h-x-y} does not hold, then \ref{h-linear} and defining properties of key forms imply that $\deg_x(h)/\deg_y(h) = \omega_1/\omega_0$. Assertion \ref{division} and the defining properties of normal forms then imply again that $\omega_1 = 1$ and $\bar X \cong \pp^2(1,1,\omega_0)$. 
\end{proof}
Assertion \ref{X-bar-proj-1-q} follows from combining \cref{wt-proj-claim} with the observation that $\delta(h) = 1$. 
\end{proof}

Combining \cref{preserving-thm} with \cref{rigid-prop} immediately yields a complete description of groups of automorphisms of primitive compactifications of $\cc^2$. 

\begin{thm} \label{aut-cor}
Fix a system of coordinates $(x,y)$ on $X := \cc^2$. Let $\vec \omega := (\omega_0, \ldots, \omega_{n+1})$ be a primitive key sequence in normal form, $\vec\theta := (\theta_1, \ldots,\theta_n) \in \ntorus$, and $\bar X := \xomegatheta$ be the corresponding primitive compactification of $X$. Let $\scrG$ be the group of automorphisms of $\bar X$. 
\begin{enumerate}	
\item \label{aut-wt} If  \eqref{trivially-normal} holds, then $\bar X \cong \pp^2(1,\omega_0,\omega_1)$. Fix (weighted) homogeneous coordinates $[z:x:y]$ on $\bar X$. 
\begin{enumerate}
\item \label{aut-wt-0} If $\omega_0 = \omega_1 = 1$, then $\bar X \cong \pp^2$ and $\scrG \cong PGL(3,\cc)$. 
\item \label{aut-wt-1} If $\omega_0 > \omega_1 = 1$, then $\scrG = \{[z:x:y] \mapsto [az+by:cx+f(y,z):dz + ey] : a,b,d,e\in \cc,\ ad -be \neq 0,\ c \in \cc^*,\ f$ is a homogeneous polynomial in $(y,z)$ of degree $\omega_0\}$.
\item \label{aut-wt-2} If $\omega_0 > \omega_1 > 1$, then $\scrG = \{[z:x:y] \mapsto [z:ax+ f(y,z): by + cz^{\omega_1}] : a,b\in \cc^*,\ c\in \cc,\ f$ is a weighted homogeneous polynomial in $(y,z)$ of weighted degree $\omega_0\}$.
\end{enumerate}
\item \label{aut-general} If \eqref{generally-normal} holds, define $\bar \omega_0, \ldots, \bar \omega_n, \baromegastar$ as in assertion \eqref{n>=1} of \cref{preserving-thm}, and set 
$$k_{\bar X} := -\left(\omega_0 + \omega_{n+1} + 1 - \sum_{k=1}^n (\alpha_k -1)\omega_k\right),$$
where $\alpha_1, \ldots, \alpha_{n+1}$ are as in Definition \ref{key-seqn}. Then $\scrG$ consists of all $F:\bar X \to \bar X$ such that $F|_X : (x,y) \mapsto (a^{\bar \omega_0} x + f(y), a^{\bar \omega_1} y + c)$, where
\begin{align*}
a &= \begin{cases}
		\text{an arbitrary element of $\cc^*$} 	& \text{if $n = 1$,}\\	
		\text{an $\baromegastar$-th root of	unity}	& \text{if $n \geq 2$.}
	 \end{cases} \\
c &= \begin{cases}
		0 & \text{if}\ \omega_0 + k_{\bar X} \geq 0, \\
		\text{an arbitrary element in \cc} & \text{otherwise.}
	 \end{cases}
\end{align*}
and $f(y) \in \cc[y]$ is a polynomial such that $\deg(f) \leq -(k_{\bar X} + \omega_1 + 1)/\omega_1$. \qed
\end{enumerate}
\end{thm}

\begin{proof}
We only need to check that
\begin{prooflist}
\item \label{autoN1:1} $\frac{\omega_0}{\omega_1}(\chi_{l+1}+1) - 1 = -(k_{\bar X} + \omega_1 + 1)/\omega_1$, and
\item \label{autoN1:2} $\chi_{l+1} \geq 0$ iff $\omega_0 + k_{\bar X} < 0$. 
\end{prooflist}
Identity \eqref{omega-and-chi-0} implies that 
\begin{align} 
\omega_{n+1} = \sum_{k=1}^n(\alpha_k - 1)\omega_k + \chi_{l+1}\omega_0 \label{omega-l+1}
\end{align}
It follows that 
\begin{align*}
(k_{\bar X} + \omega_1 + 1)/\omega_1 
	&= (-\omega_0+\omega_1 -  \chi_{l+1}\omega_0)/\omega_1 = 1 - \omega_0(\chi_{l+1}+1)/\omega_1
\end{align*}
which proves \ref{autoN1:1}. Now \eqref{omega-l+1} implies that 
\begin{align*}
\omega_0 + k_{\bar X} = -1 - \chi_{l+1}\omega_0
\end{align*}
Since $\omega_0 + k_{\bar X}$ is an integer, it follows that $\chi_{l+1} \geq 0$ 
iff $\omega_0 + k_{\bar X}  < 0$, as required to prove \ref{autoN1:2}. 
\end{proof}

\begin{cor} \label{aut-cor1}
Adopt the notations of \cref{aut-cor}. If $k_{\bar X} \geq -1$ and $n \geq 2$, then $\bar X$ admits only finitely many automorphisms. In particular, every non-algebraic primitive compactification of $\cc^2$ admits only finitely many automorphisms.
\end{cor}

\begin{proof}
The first statement follows from assertion \eqref{aut-general} of \cref{aut-cor}. For the last statement, we show that if $\bar X$ is non-algebraic, then $k_{\bar X} > -1$ and $n \geq 2$. Indeed, if $k_{\bar X} \leq -1$, then \cref{singularity-cor} implies that $\bar X$ has only rational singularities, so that Artin's contraction criterion \cite{artin-rational} implies that $\bar X$ is algebraic. On the other hand, if $n \leq 1$, then it is straightforward to see that $\vec \omega$ is an algebraic key sequence, so that $\bar X$ is algebraic (\cref{key-primitive-correspondence}).
\end{proof}

\section{Moduli spaces} \label{moduli-section}
In this section we compute moduli spaces of two kinds of objects:
\begin{itemize}
\item primitive normal compactifications of $\cc^2$ modulo isomorphisms (in the category of normal analytic surfaces), and
\item curves $C \subseteq \cc^2$ with one place at infinity modulo automorphisms of $\cc^2$. 
\end{itemize}

\begin{defn}
Let $\vec \omega$ be a primitive key sequence in normal form. We denote by $\yomega$ the moduli space of primitive normal compactifications of $\cc^2$ with key sequence $\vec \omega$. More precisely, $\yomega$ is the set of (isomorphism classes of) compact normal analytic surfaces $Y$ of Picard rank $1$ such that
\begin{enumerate}
\item $Y$ has a subset $X$ isomorphic to $\cc^2$, and
\item if $\delta$ is the semidegree on $\cc[X]$ corresponding to $C_\infty := Y \setminus X$, then there exists a system of coordinates $(x,y)$ on $X$ such that $\delta(x) > 0$ and $\vec \omega$ is the key sequence of $\delta$ in $(x,y)$ coordinates.
\end{enumerate}
In case that $\vec \omega$ is also an {\em essential} key sequence, we denote by $\yomegaess$ the union of all $\yomegaprime$ such that
\begin{enumerate}
\setcounter{enumi}{2}
\item $\vec \omega'$ is a primitive key sequence in normal form, and
\item the essential subsequence of $\vec \omega'$ is $\vec \omega$.
\end{enumerate} 
Finally $\yomegaessalg$ is the subset of $\yomegaess$ consisting of all $Y \in \yomegaess$ such that $Y$ is {\em algebraic}.
\end{defn}

\begin{rem}
The algebraicity of $\yomega$ depends only on $\vec \omega$ (\cref{key-primitive-correspondence}), i.e.\ if $\vec \omega$ is algebraic, then all elements of $\yomega$ are algebraic surfaces, and if $\vec \omega$ is not algebraic, then no element of $\yomega$ is algebraic. Similarly, when $\vec \omega$ is essential, the algebraicity of $\vec \omega$ determines if $\yomegaessalg$ is non-empty.
\end{rem}

Let $\vec \omega := (\omega_0, \ldots, \omega_{n+1})$ be a primitive key sequence in normal form. Let $X$ be a {\em fixed} copy of $\cc^2$ with fixed system of coordinates $(x,y)$. Then the map 
\begin{align}
 \ntorus \ni \vec \theta \mapsto \xomegatheta \in \yomega \label{moduli-map}
\end{align}
is surjective. Consequently, $\yomega$ is isomorphic to the quotient space of $\ntorus$ modulo the equivalence relation $\vec \theta \sim \vec \theta'$ iff $\xomegatheta \cong \xomegathetaprime$. \Cref{moduli-cor} below describes this equivalence relation; we now set up necessary notations. Define $\alpha_i$'s and $\beta_{i,j}$'s as in Remark \ref{unique-remark}. Moreover, set $\alpha_0 := 1$. Let $\vec \omega_e := (\omega_{i_0}, \ldots, \omega_{i_{l+1}})$ be the essential subsequence of $\vec \omega$. Recall that $i_0 = 0$ and $i_{l+1} = n+1$. The normality of $\vec \omega$ further implies that $i_1 = 1$. Define $\mu_1, \ldots, \mu_n \in \zz$ as follows: for each $i$, $1 \leq i \leq n$, pick the unique $k$ such that $i_k \leq i < i_{k+1}$, and set
\begin{align*}
\mu_i := \alpha_{i_0} \cdots \alpha_{i_k} - \sum_{j=1}^k \alpha_{i_0} \cdots \alpha_{i_{j-1}} \beta_{i,i_{j}}
\end{align*}
(we note that $\mu_i$'s are the same as in \cref{theta-thm}). 

\begin{thm} \label{moduli-cor}
\mbox{}
\begin{enumerate}
\item \label{yomega-singleton} If $n = 0$, then $\yomega$ is a point (corresponding to the weighted projective surface $\pp^2(1,\omega_0,\omega_1)$).
\item \label{yomega-general} If $n= 1$, then the map from \eqref{moduli-map} induces an isomorphism $\yomega \cong \ntorus/(\cc^*)^2$, where the action of $(\cc^*)^2$ is given by
\begin{align}
(\lambda_1, \lambda_2)\cdot (\theta_1, \ldots, \theta_n) := 
(\lambda_1^{-\beta_{1,0}}\lambda_2^{\mu_1}\theta_1, \ldots, 
\lambda_1^{-\beta_{n,0}}\lambda_2^{\mu_n}\theta_n). \label{yomega-general-action}
\end{align}

In particular, $\yomega \cong (\cc^*)^{\max\{n-2,0\}}$.
\end{enumerate} 
\end{thm}

\begin{proof}
Assertion \ref{yomega-singleton} follows from \cref{aut-cor}. We now prove assertion \eqref{yomega-general}. Pick $\vec\theta, \vec \theta' \in \ntorus$ such that there exists an isomorphism $F:\xomegatheta \cong \xomegathetaprime$. \cref{rigid-prop} implies that $F|_X$ is an automorphism of $X$. It follows that $\deltaomegathetaprime = F^*(\deltaomegatheta)$. Let $\tilde \phi_\delta(x,\xi) = \phi(x) + \xi x^r = \sum_{\beta \leq \beta_0} a_\beta x^{\beta} + \xi x^r$ be the generic descending Puiseux series of $\deltaomegatheta$ with respect to $(x,y)$-coordinates. Assertion \ref{unique-puiseux} of \cref{normal-morphism} then implies that the generic descending Puiseux series of $\deltaomegathetaprime$ in $(x,y)$-coordinates is $b\sum_{\beta \leq \beta_0} a_\beta a^{-\bar \omega_0\beta} x^{\beta} + \xi x^{r}$ for some $a, b \in \cc^*$, where $\bar \omega_0$ is the polydromy order of $\phi(x)$. \Cref{theta-thm} then implies that $(\theta'_1, \ldots, \theta'_n) = (b^{\mu_1}a^{-\omega_0\beta_{1,0}}\theta_1, \ldots, b^{\mu_n}a^{-\omega_0\beta_{n,0}}\theta_n)$. This proves assertion \eqref{yomega-general} and finishes the proof of the theorem.
\end{proof}

We continue to use the notations of \cref{moduli-cor}. Assume in addition that $\vec \omega$ is essential, i.e.\ $\vec \omega_e = \vec \omega$. We now describe $\yomegaess$ and $\yomegaessalg$. \\

For each $k$, $1 \leq k \leq n$, define 
\begin{align*}
\check \Omega_k 
	&:= \left(\zz\langle \omega_0, \ldots, \omega_k \rangle \cap (\omega_{k+1}, \alpha_k\omega_k)\right) \setminus \Omega^\times_k,\\
\check \Omega^{alg}_k 
	&:= \left(\zz_{\geq 0}\langle \omega_0, \ldots, \omega_k \rangle \cap (\omega_{k+1}, \alpha_k\omega_k)\right) \setminus \Omega^\times_k,	
\end{align*}
where $(\omega_{k+1}, \alpha_k\omega_k)$ is the open interval between $\omega_{k+1}$ and $\alpha_k\omega_k$; $\zz\langle \omega_0, \ldots, \omega_k \rangle$, $\zz_{\geq 0}\langle \omega_0, \ldots, \omega_k \rangle$ denote respectively the group and the semigroup generated by $\omega_0, \ldots, \omega_k$, and
\begin{align*}
\Omega^\times_k &:= \left\{ \sum_{j=1}^{k} (\alpha_j-1)\omega_j + m\omega_0: m \geq 0 \right\} \bigcup \left\{ \sum_{j=2}^{k} (\alpha_j-1)\omega_j + \alpha_1\omega_1 - \omega_0 \right\}.			
\end{align*}
Let $\check \Omega := \bigcup_{k=1}^n \check \Omega_k$, $\check \Omega^{alg} := \bigcup_{k=1}^n \check \Omega^{alg}_k$, and $m_k := |\bigcup_{j=1}^k \check \Omega_j|$, $1 \leq k \leq n$; note that $|\check \Omega| = m_n$. Set $m_0 := 1$. For each $k$, $1 \leq k \leq n$, denote the elements of $\check \Omega_k$ in decreasing order by $\check \omega_{m_{k-1} + 1} > \check \omega_{m_{k-1} + 2} > \cdots > \check \omega_{m_k}$. Then $\check \Omega = \{\check \omega_1, \ldots, \check \omega_{m_n}\}$. Let 
\begin{align}
\vec{\hat \omega} := (\omega_1, \check \omega_1, \ldots, \check \omega_{m_1}, \omega_2, \check \omega_{m_1+1}, \ldots, \check \omega_{m_2}, \ldots, \omega_n, \check \omega_{m_{n-1}+1}, \ldots, \check \omega_{m_n}) \label{hat-omega}
\end{align}
Given a subset $S$ of $\{1, \ldots, m_n\}$, let $\pi_S(\vec {\hat \omega})$ be the element formed from $\vec {\hat \omega}$ by omitting all $\check \omega_i$ such that $i \not\in S$. 
\begin{claim}
Let $\vec \Omega'$ be the set of all key sequences in normal form with essential subsequence $\vec \omega$. Then 
\begin{align*}
\vec \Omega' = \{\pi_S(\vec {\hat \omega}):  S \subseteq \{1, \ldots, m_n\}\}.
\end{align*}
\end{claim}

\begin{proof}
The claim follows in a straightforward manner from the definition of normal forms and essential subsequences, once we make the following observation: if $\vec \omega'$ is a primitive key sequence with essential subsequence $\vec \omega$, say $\omega_k = \omega'_{i'_k}$, $1 \leq  k \leq n+1$. Then $\vec \omega'$ satisfies condition \eqref{generally-omega} of normal forms iff for each $k = 1, \ldots, n$, $\omega'_i \not\in \Omega^\times_k$ for each $i = i'_k + 1, \ldots, i'_{k+1} - 1$.
\end{proof}

Let $\Theta: \ntorus \times \cc^{m_n} \to \cc^{m_n+n}$ be the map defined by
\begin{align}
((\theta_1, \ldots, \theta_n), (\check \theta_1, \ldots, \check \theta_{m_n}))
	\mapsto (\theta_1, \check \theta_1, \ldots, \check \theta_{m_1}, \theta_2, \check \theta_{m_1+1}, \ldots, \check \theta_{m_2}, \ldots, \theta_n, \check \theta_{m_{n-1}+1}, \ldots, \check \theta_{m_n})
\end{align}
Pick $(\vec \theta, \vec{\check \theta}) \in \ntorus \times \cc^{m_n}$. Let $S:= \{i:  \check \theta_i \neq 0\} \subseteq \{1, \ldots, m_n\}$, and $m := |S|$. Let $\vec \theta'_{(\vec \theta, \vec{\check \theta})} \in (\cc^*)^{m+n}$ be the element formed by dropping all the zero coordinates of $\Theta(\vec \theta, \vec{\check \theta})$, and let $\vec \omega'_{(\vec \theta, \vec{\check \theta})} := \pi_S(\vec{\hat \omega})$ be the corresponding key sequence. Then the map 
\begin{align}
\ntorus \times \cc^{m_n} \ni (\vec \theta, \vec{\check \theta}) \mapsto
	 \bar X_{\vec \omega'_{(\vec \theta, \vec{\check \theta})} , \vec \theta'_{(\vec \theta, \vec{\check \theta})}}
	 \in \yomegaess \label{moduli-map-alg}
\end{align}
is a surjection, and \cref{moduli-cor} combined with \cref{key-primitive-correspondence} immediately gives the following description of $\yomegaess$ and $\yomegaessalg$.

\begin{cor} \label{moduli-essential-cor}
 Fix $i$, $1 \leq i \leq m_n$. Pick (the unique) $k$, $1 \leq k \leq n$, such that $m_{k-1} < i \leq m_k$. Then there are unique integers $\check \beta_{i,0}, \ldots, \check \beta_{i,k}$ such that $0 \leq \check \beta_{i,j} < \alpha_j$ for $1 \leq j \leq k$ and $\check \omega_i = \sum_{j=0}^k \check \beta_{i,j}\omega_j$. Define $\check \mu_i := \alpha_{0} \cdots \alpha_{k} - \sum_{j=1}^k \alpha_{0} \cdots \alpha_{j-1}\check \beta_{i,j}$. 
\begin{enumerate}
\item \label{yomega-essentially-general} $\yomegaess \cong \left(\ntorus \times \cc^{m_n}\right) /(\cc^*)^2$, where the action of $(\cc^*)^2$ is given by
\begin{align}
(\lambda_1, \lambda_2) \cdot (\vec \theta, \vec{\check \theta})  
	:= 
(\lambda_1^{-\beta_{1,0}}\lambda_2^{\mu_1} \theta_1, \ldots, 
\lambda_1^{-\beta_{n,0}} \lambda_2^{\mu_n} \theta_n, 
\lambda_1^{-\check \beta_{1,0}} \lambda_2^{\check \mu_1} \check \theta_1, \ldots, 
\lambda_1^{-\check \beta_{m_n,0}} \lambda_2^{\check \mu_{m_n}} \check \theta_{m_n}), \label{yomega-essential-action}
\end{align}
where $(\vec \theta, \vec{\check \theta}) := (\theta_1, \ldots, \theta_n),(\check \theta_1, \ldots, \check \theta_{m_n}) \in \ntorus \times \cc^{m_n}$.
\item If $\vec \omega$ is not algebraic then $\yomegaessalg = \emptyset$.
\item \label{algebraic-moduli} If $\vec \omega$ is algebraic, then $\yomegaessalg \cong  Y^{e,alg}_{\vec \omega}/(\cc^*)^2$, where
$$Y^{e,alg}_{\vec \omega} := \{(\theta, \check \theta) \in \ntorus \times \cc^{m_n}: \check \theta_i = 0\ \text{for all $i$, $1 \leq i \leq m_n$, such that}\ \check \omega_i \not\in \check \Omega^{alg}\},$$
and the action of $(\cc^*)^2$ on $Y^{e,alg}_{\vec \omega}$ is induced from \eqref{yomega-essential-action}.
\end{enumerate}
\end{cor}

As an application of \cref{moduli-cor} we describe the moduli space of embedded isomorphism classes of planar curves with one place at infinity. \Cref{approximate-remark} describes a {\em coordinate-free} correspondence between plane curves $C$ with one place at infinity and divisorial semidegrees $\delta_C$ on $\cc[x,y]$. If the coordinates are chosen in a way that the key sequence of $\delta_C$ is in the normal form, identity \eqref{delta-from-key} implies that the $\delta$-sequence of $C$ in these coordinates is also in the normal form (when viewed as a key sequence).

\begin{defn} \label{curve-moduli}
Let $\vec d$ be a $\delta$-sequence  in normal form. We denote by $\cd$ the space of all pairs $(C,U)$ such that $U \cong \cc^2$ and $C$ is a curve on $U$ with one place at infinity such that the $\delta$-sequence of $C$ is $\vec d$ with respect to some systems of coordinates on $U$. For $(C,U), (C',U') \in \cd$, we write $(C,U) \sim (C',U')$ iff there is an isomorphism $F:U \to U'$ such that $C' = F(C)$. The `embedded isomorphism classes' of planar curves with one place at infinity is the quotient $\cdbar$ of $\cd$ by the equivalence relation $\sim$. 
\end{defn}

Let $\vec d :=  (d_0, \ldots, d_n)$ be a delta sequence in normal form. \Cref{approximate-remark} implies that $\vec \omega := (d_0, \ldots, d_n,0)$ is an essential algebraic key sequence. Construct $\check \Omega$, $\check \Omega^{alg}$, $\vec{\hat \omega}$ and define $m_n, \vec \omega'_{(\vec \theta, \vec{\check \theta})}$, $\vec \theta'_{(\vec \theta, \vec{\check \theta})}$ as in the paragraphs following  \cref{moduli-cor}. Let $\gthetatheta$ be the last key form of $\delta_{\vec \omega'_{(\vec \theta, \vec{\check \theta})} , \vec \theta'_{(\vec \theta, \vec{\check \theta})}}$.

\begin{thm} \label{one-place-cor}
Adopt the notations of \cref{moduli-essential-cor}. Recall that $X$ is a fixed copy of $\cc^2$ with fixed system of coordinates $(x,y)$. Consider the map 
\begin{align*}
\Psi: Y^{e,alg}_{\vec \omega} \times \cc \ni (\vec \theta, \vec{\check \theta}, c) \mapsto (\cthetathetac, X) \in \cdbar
\end{align*}
where $Y^{e,alg}_{\vec \omega}$ is as in assertion \eqref{algebraic-moduli} of \cref{moduli-essential-cor}, and $\cthetathetac := \{\gthetatheta - c = 0\} \subseteq X$. Then $\Psi$ induces an isomorphism $\cdbar \cong (Y^{e,alg}_{\vec \omega}  \times \cc)/(\cc^*)^2$, where the action of $(\cc^*)^2$ on $Y^{e,alg}_{\vec \omega}  \times \cc$ is given by 
\begin{align}
(\lambda_1, \lambda_2)\cdot (\vec \theta, \vec{\check \theta}, c) 
	&= ( (\lambda_1, \lambda_2)\cdot (\vec \theta, \vec{\check \theta}), \lambda_2^{d_0}c) \label{moduli-curve-action}
\end{align}
where $(\lambda_1, \lambda_2) \cdot (\vec \theta, \vec{\check \theta})$ is as in \eqref{yomega-essential-action}.
\end{thm}

\begin{proof}
The discussion in Remark \ref{approximate-remark} implies that $\Psi$ is surjective, so that we only have to determine when two points correspond to the same embedded isomorphism class. At first consider the case $n = 0$, i.e.\ $\vec d = (1)$. Then $\vec \omega = (1,0)$, $m_n = 0$, $Y^{e,alg}_{\vec \omega}$ is a singleton, and \eqref{moduli-curve-action} shows that $(Y^{e,alg}_{\vec \omega}  \times \cc)/(\cc^*)^2$ is also a singleton. On the other hand, $\cthetathetac$'s are simply the curves $\{y - c = 0\}$, which are all isomorphic. It follows that $\cdbar$ is also a singleton, and the theorem holds. \\

So assume $n \geq 1$, and pick $(\vec \theta, \vec{\check \theta}, c) , (\vec \theta', \vec{\check {\theta'}}, c')$ such that there is an isomorphism $\phi: X \to X$ such that $F(\cthetathetac) = \cthetathetacpprime$. It follows that $\delta_{\cthetathetacpprime} = F^*(\delta_{\cthetathetac})$, where $\delta_{\cthetathetac}$ and $\delta_{\cthetathetacpprime} $ are defined as in Remark \ref{approximate-remark}. Since $\vec \omega$ is the essential subsequence of both $\delta_{\cthetathetac} $ and $\delta_{\cthetathetacpprime}$, and since $n \geq 1$ and $\omega_{n+1} = 0$, assertion \eqref{positive-F} of \cref{normal-morphism} and observation \ref{autoN1:2} from the proof of \cref{aut-cor} imply that $F : (x,y) \mapsto (\bar a x,  \bar b y)$ for some $\bar a, \bar b \in \cc^*$. The arguments from the proof of \cref{moduli-cor} and assertion \eqref{unique-puiseux-b} of \cref{normal-morphism} imply that $ (\vec \theta', \vec{\check {\theta'}}) = (\lambda_1, \lambda_2) \cdot (\vec \theta, \vec{\check \theta})$, with $\bar a = \lambda_1^{d_0}$ and $\bar b = \lambda_2$. It follows that 
\begin{align*}
V(\gthetathetapprime - c') 
	= F(V(\gthetatheta - c)) 
	= V(g(\lambda_1^{-d_0}x, \lambda_2^{-1}y) - c)
\end{align*}
Since $\gthetatheta$ and $\gthetathetapprime$ are monic of degree $d_0$ in $y$ (assertion \eqref{deg-y-g-n+1} of \cref{last-curve-prop}), it follows that $\gthetathetapprime = \lambda_2^{d_0}g(\lambda_1^{-\bar \omega_0}x, \lambda_2^{-1}y)$ and $c' = \lambda_2^{d_0}c$, as required to complete the proof. 
\end{proof}

\section{Canonical divisor} \label{canonical-section}
In this section we compute the canonical divisor of a primitive compactifications of $\cc^2$ in terms of the associated key sequence (\cref{canonical-thm}), and give some of its applications. In \cref{canonical-0-section} we state \cref{canonical-thm} and use it to characterize $\pp^2(1,1,q)$ in terms of log discrepancy and skewness of the curve at infinity. In \cref{singular-section} we characterize when a primitive compactification has simple types of singularities or when it is Gorenstein. Finally, in \cref{canonical-subsection} we give the proof of \cref{canonical-thm}.

\subsection{The formula for canonical divisor and a characterization of $\pp^2(1,1,q)$}
\label{canonical-0-section}
Let $X := \cc^2$. Throughout \cref{canonical-0-section} $\delta$ is a divisorial semidegree on $\cc[X]$. 

\begin{defn}[{\cite[Section 9.3.3]{jonsson-dykovich}}] \label{log-discrepancy}
Let $\bar X$ be a normal analytic compactification of $X$ such that $\delta$ is centered at a curve $C$ at infinity on $\bar X$. Let $K_{\bar X}$ be the (unique) Weil divisor representing the canonical divisor of $\bar X$ such that $\supp(K_{\bar X}) \subseteq \bar X \setminus X$. The {\em log discrepancy} $A_\delta$ of $\delta$ is one plus the coefficient of $[C]$ in $K_{\bar X}$ (where $[C]$ is the Weil divisor corresponding to $C$). 
\end{defn}

\begin{thm} \label{canonical-thm}
Let $\vec \omega := (\omega_0, \ldots, \omega_{n+1})$ be the key sequence of $\delta$ in some system of coordinates on $X$. Then 
\begin{align}
	A_\delta &= -\omega_0 - \omega_{n+1} + \sum_{k=1}^n (\alpha_k -1)\omega_k \label{log-discrepancy-formula}
\end{align}
where $\alpha_1, \ldots, \alpha_{n+1}$ are as in Definition \ref{key-seqn}.
In particular, if $\delta$ is primitive algebraic and $\bar X$ is the corresponding primitive algebraic compactification of $X$, then the canonical divisor of $\bar X$ is
\begin{align}
	K_{\bar X}	&= -\left(\omega_0 + \omega_{n+1} + 1 - \sum_{k=1}^n (\alpha_k -1)\omega_k \right)[C_\infty], \label{canonical-formula}
\end{align}
where $[C_\infty]$ is the Weil divisor corresponding to $C_\infty$. 
\end{thm}

%
%

\begin{reminition}[{\cite[Section 9.3.3]{jonsson-dykovich}}] \label{index}
Let $\bar X$ be a normal analytic compactification of $X$ such that $\delta$ is centered at a curve $C$ at infinity on $\bar X$. Let $\check C$ be the unique curve supported at $\bar X \setminus X$ such that $(\check C, C) = 1$ and $(\check C, D) = 0$ for all irreducible curve $D \neq C$ at infinity on $\bar X$ (here we consider the intersection product on normal surfaces defined by Mumford \cite{mumford-normal}). The {\em \skewness} of $\delta$ is $\alpha_\delta := (\check C, \check C)$. It is straightforward to see that $\alpha_\delta$ is independent of the choice of $\bar X$. Moreover,
\begin{itemize}
\item \cite[Theorem 1.5 and Remark 1.6]{sub2-1} imply that if $(\omega_0, \ldots, \omega_{n+1})$ is the key sequence of $\delta$ in a system of coordinates on $X$, then
\begin{align}
\alpha_\delta = \alpha_{n+1}\omega_{n+1} \label{alpha_delta}
\end{align}
where $\alpha_{n+1} := \gcd(\omega_0, \ldots, \omega_n)$. 
\item Identity \eqref{alpha_delta} implies in particular that $\delta$ is primitive iff $\alpha_\delta > 0$. 
\item \Cref{index-prop} and identity \eqref{alpha_delta} imply that if $\bar X$ is primitive algebraic, then $\alpha_\delta$ is precisely the index of the Weil divisor at infinity (this is the motivation for our terminology `index' for $\alpha_\delta$). 
\end{itemize}

\end{reminition}

\begin{cor} \label{11q-thm}
The following are equivalent:
\begin{enumerate}
\item \label{A-alpha-condition} $\alpha_\delta \geq 0$ and $A_\delta \leq -\alpha_\delta$.
\item \label{1,1,q} there is a system of coordinates $(u,v)$ on $X$ and a non-negative integer $q$ such that $\delta$ is the weighted degree corresponding to weight $1$ for $u$ and $q$ for $v$. 
\end{enumerate}
If either of these conditions holds, then $q = \alpha_\delta$ and $A_\delta = -(\alpha_\delta+1)$. 
\end{cor}

\begin{proof}
If condition \eqref{A-alpha-condition} holds, then the key sequence of $\delta$ in $(u,v)$ coordinates is $\vec \omega = (1,q)$. Condition \eqref{1,1,q} then follows immediately from identities \eqref{log-discrepancy-formula} and \eqref{alpha_delta}. Now we verify the implication \eqref{A-alpha-condition} $\im$ \eqref{1,1,q}. Choose a system of coordinates $(x,y)$ on $X$ such that the corresponding key sequence $\vec \omega := (\omega_0, \ldots, \omega_{n+1})$ of $\delta$ is in the normal form. Identities \eqref{log-discrepancy-formula} and \eqref{alpha_delta} imply that 
\begin{align*}
A(\delta) + \alpha(\delta) 
	&= -\omega_0 + \sum_{k=1}^{n+1} (\alpha_k -1)\omega_k 
	= -\omega_1 +  (\alpha'_0 -1)\omega_0 + \sum_{k=2}^{n+1} (\alpha_k -1)\omega_k
\end{align*}
where $\alpha'_0 := \omega_1/\gcd(\omega_0, \omega_1)$. Note that $\alpha_k \geq 1$ for $k = 2, \ldots, n+1$. Moreover, the normality of $\vec \omega$ implies that if $\alpha'_0 > 1$, then $\omega_0 > \omega_1$. Since $A(\delta) + \alpha(\delta) \leq 0 $, it follows that $\alpha'_0 = 1$. It follows then from the properties of normal form that $n = 0$ and $\omega_1 = 1$. Consequently $\alpha_{n+1} = \omega_0$, so that identity \eqref{alpha_delta} implies that $\omega_0 = \alpha_\delta$. Therefore $\delta$ is the weighted degree in $(x,y)$ coordinates corresponding to weight $\alpha_\delta$ for $x$ and $1$ for $y$, as required. 
\end{proof}

\begin{cor}[{\cite[Theorem 1.2]{borisov}}] \label{borisov}
The following are equivalent:
\begin{enumerate}
\item $\alpha_\delta = 1$ and $A_\delta = -2$.
\item $\delta$ is the degree of polynomials in some system of coordinates on $X$. \qed
\end{enumerate}
\end{cor}

\subsection{Primitive compactifications with simple singularities} \label{singular-section}
Throughout this section $\bar X$ is a primitive normal compactification of $X := \cc^2$ and $\vec \omega := (\omega_0, \ldots, \omega_{n+1})$ is the key sequence corresponding to the semidegree $\delta$ on $\cc[X]$ with respect to a system of coordinates $(x,y)$ on $X$ such that $\delta(x) > 0$. In this section we characterize in terms of $\vec \omega$ when $\bar X$ has simple types of singularities. \\

Recall that the {\em geometric genus} of an isolated singular point $P$ on a complex surface $Y$ is $p_g(P) := \dim_\cc(R^1 \pi_*\sheaf_{\tilde Y})_{P}$, where $\pi:\tilde Y \to Y$ is a resolution of singularities. The singularity of $P$ is called {\em rational} (resp.\ {\em elliptic}) if $p_g(P) = 0$ (resp.\ $p_g(P) = 1$). 

\begin{lemma}[{\cite[Lemma 2.2]{furushima}}] \label{geometric-sum}
Assume $\bar X$ is algebraic. Then the sum of the geometric genera of singular points of $\bar X$ is equal to $\dim_\cc(H^0(\bar X, \sheaf_{\bar X}(K_{\bar X}))$.
\end{lemma}

Recall from \cref{general-structure-thm} that $\bar X$ has at most two singular points, and $P_\infty$ (defined in \cref{main-curve-prop}) is the only point on $\bar X$ which may have a non-rational singularity (since all quotient singularities are rational). The following result characterizes when the singularity at $P_\infty$ is rational or elliptic. It follows immediately via combining \cref{canonical-thm,geometric-sum,basis-lemma}. 

\begin{cor} \label{singularity-cor}
Set $k_{\bar X} := -\left(\omega_0 + \omega_{n+1} + 1 - \sum_{k=1}^n (\alpha_k -1)\omega_k\right)$. 
\begin{enumerate}
\item The singularity at $P_\infty$ is rational iff $k_{\bar X} < 0$.
\item \label{elliptic} Assume $\bar X$ is algebraic. Then the singularity at $P_\infty$ is elliptic iff $0 \leq k_{\bar X} < \omega_{\min}$, where $\omega_{\min} := \min\{\omega_0, \ldots, \omega_{n+1}\}$.
\item Assume $\bar X$ is algebraic. Then $p_g(P_\infty) = |\Sigma|$, where $\Sigma$ is the collection of all $(\beta_0,\cdots, \beta_{n+1}) \in \zz_{\geq 0}^{n+2}$ such that $\beta_j < \alpha_j$, $1 \leq j \leq n$, and $\sum_{j=0}^{n+1} \omega_j\beta_j \leq k_{\bar X}$. \qed
\end{enumerate}
\end{cor}

Recall that a normal surface is {\em Goerenstein} iff the canonical divisor is Cartier. Combining \cref{singularity-cor} with \cref{index-prop} immediately gives the following characterization of Goerenstein primitive algebraic compactifications. 

\begin{cor} \label{gorenstein-cor}
Assume $\bar X$ is algebraic. Then the {\em index} of $\bar X$ (i.e.\ the smallest positive integer $m$ such that $mK_{\bar X}$ is Cartier) is 
\begin{align}
\ind(\bar X) 
	&= \frac{\alpha_{n+1}\omega_{n+1}}{\gcd(\alpha_{n+1}\omega_{n+1}, k_{\bar X})}
	= \frac{\alpha_{n+1}}{\gcd(\alpha_{n+1}, \omega_{n+1}+1)}
	\cdot \frac{\omega_{n+1}}{\gcd(\omega_{n+1}, k_{\bar X})}
\end{align}
where $\alpha_{n+1} := \gcd(\omega_0, \ldots, \omega_n)$ and $k_{\bar X} := -\left(\omega_0 + \omega_{n+1} + 1 - \sum_{k=1}^n (\alpha_k -1)\omega_k\right)$. 
In particular, $\bar X$ is Gorenstein iff $k_{\bar X}$ is divisible by $\alpha_{n+1}\omega_{n+1}$. \qed
\end{cor}

\begin{cor} [cf.\ {\cite[Theorem 6]{brenton-graph-1}}] \label{rational-cor}
Let $\bar X$ be a Gorenstein primitive compactification of $X$ with rational singularities. Then one of the following is true: 
\begin{enumerate}
\item $\bar X \cong \pp^2$,
\item \label{case-wt-proj-1} $\bar X \cong \pp^2(1,1,2)$,
\item \label{case-wt-proj-2} $\bar X \cong \pp^2(1,2,3)$,
\item \label{case-non-wt-proj} $\bar X$ is the hypersurface in $\pp^2(1,2,3,r)$ (with weighted homogeneous coordinates $[w:x:y:z]$) for $5 \geq r \geq 1$ defined by the weighted homogeneous polynomial $F_r$ given by
\begin{align*}
F_r &:= \begin{cases}
		wz - (y^3 + x^2) &\text{if}\ r = 5,\\
		w^2z - (y^3 + x^2 + awxy) &\text{if}\ r = 4,\\
		w^3z - (y^3 + x^2 + awxy + bw^2y^2) &\text{if}\ r = 3,\\
		w^4z - (y^3 + x^2 + awxy + bw^2y^2 + cw^3x) &\text{if}\ r = 2,\\
		w^5z - (y^3 + x^2 + awxy + bw^2y^2 + cw^3x + dw^4y) &\text{if}\ r = 1,\\
		\end{cases}
\end{align*}
where $a,b,c,d \in \cc$.
\end{enumerate}
\end{cor}

\begin{proof}
W.l.o.g.\ assume $\bar X \not\cong \pp^2$. Choose coordinates $(x,y)$ on $X$ such that the corresponding key sequence $\vec \omega := (\omega_0, \ldots, \omega_{n+1})$ of the semidegree on $\kk[x,y]$ associated to $C_\infty$ is in the normal form. In particular,
\begin{enumerate}[label= (\alph{enumi})]
\item $\omega_0 > \omega_1$, and
\item either $n = 0$ or $\alpha_1 > 1$.
\end{enumerate}
 Let $k_{\bar X} := -\left(\omega_0 + \omega_{n+1} + 1 - \sum_{k=1}^n (\alpha_k -1)\omega_k\right)$. \cref{singularity-cor} implies that $|k_{\bar X}| = \omega_0 + \omega_{n+1} + 1 - \sum_{k=1}^n (\alpha_k -1)\omega_k \geq 1$. At first consider the case that $n = 0$. Then $|k_{\bar X}| = \omega_0 + \omega_1 + 1$, so that $\omega_0 < |k_{\bar X}| \leq 2\omega_0$. Consequently, $\alpha_{n+1} = \omega_0$ divides $k_{\bar X}$ iff $\omega_0 = \omega_1 + 1$, i.e.\ $|k_{\bar X}| = 2\omega_1+ 2$. But then $\omega_{n+1} = \omega_1$ divides $k_{\bar X}$ iff $\omega_1 = 2$ or $\omega_1 = 1$. Consequently we have two possibilities: $\omega_0 = 2$, $\omega_1 = 1$, which corresponds to case \eqref{case-wt-proj-1}, or $\omega_0 = 3$, $\omega_1 = 2$, which corresponds to case \eqref{case-wt-proj-2} of the corollary. Now assume $n \geq 1$. Then 
\begin{align*}
|k_{\bar X}| 
	&= \omega_{n+1} + \omega_0 + 1 - \sum_{k=1}^n (\alpha_k -1)\omega_k \\
	&\leq \omega_{n+1} + \omega_0 + 1 - (\alpha_1 -1)\omega_1 \\
	&= \omega_{n+1} + 1 + \omega_1 - \ (\alpha'_0 -1)\omega_0 \quad
	\text{(where $\alpha'_0 := \omega_1/\gcd(\omega_0, \omega_1)$)} 
\end{align*}
Property \eqref{generally-alpha} of normal forms ensures that $\alpha'_0 > 1$, so that $|k_{\bar X}| \leq \omega_{n+1}$. It follows that $\omega_{n+1}$ divides $k_{\bar X}$ iff $\omega_{n+1} = |k_{\bar X}|$ iff
\begin{enumerate}[label=(\roman{enumi})]
\item \label{alpha-k-1} $\alpha_k = 1$ for all $k \geq 2$, 
\item \label{alpha-1-2} $\alpha'_0 = 2$, and
\item \label{omega-1-0-1} $\omega_0 = \omega_1+ 1$.
\end{enumerate}
But then, properties \ref{alpha-k-1} and \ref{alpha-1-2} imply $\omega_1 = \alpha'_0 = 2$, so that $\omega_0 = 3$ due to \ref{omega-1-0-1}. Case \ref{case-non-wt-proj} of the corollary now follows from a straightforward examination of possibilities for equations \eqref{projective-definition}.
\end{proof}

Following the characterization of Gorenstein primitive compactifications with rational singularities,  we now characterize those with the next simplest type of singularities: we say that a primitive compactification $\bar X$ has a point with {\em minimally elliptic} singularity (in the sense of \cite{laufer-elliptic}) if it is Gorenstein and has a point with elliptic singularity. 


\begin{cor} \label{elliptic-cor}
\mbox{}
\begin{enumerate}
\item \label{elliptic-0} Let $\bar X$ be a primitive algebraic compactification of $\cc^2$. Then $\bar X$ has a point with minimally elliptic singularity iff $k_{\bar X} = 0$.
\item \label{ellipticc} Let $m \geq 0$, $\vec \omega := (\omega_0, \ldots, \omega_{m+1})$ be a primitive algebraic key sequence in the normal form and $d$ be a positive integer such that
\begin{enumerate}
\item \label{elliptic-1} $\alpha_{m+1}\omega_{m+1} \in \zz_{\geq 0}\langle \omega_0, \ldots, \omega_m \rangle$, where $\alpha_{m+1} := \gcd( \omega_0, \ldots, \omega_n)$.
\item $\vec \omega \neq (1,1)$.
\item \label{elliptic-3} Either $\vec \omega \neq (3,2)$ or $d \geq 2$.
\end{enumerate}
Define $\vec \omega'_{\vec \omega, d} := (d\omega_0, \ldots, d\omega_{m+1},  \sum_{k=1}^{m+1} (\alpha_k -1)d\omega_k -d\omega_0 - 1)$, where $\alpha_k$'s are as in Definition \ref{key-seqn}. Then 
\begin{enumerate}[label=(\roman{enumii})]
\item \label{min-elliptic-1} For every $\vec \theta := (\theta_1, \ldots, \theta_{m+1}) \in (\cc^*)^{m+1}$, $\xomegaprimeomegadtheta$ is a primitive algebraic compactification of $X$ having a point with minimally elliptic singularity.
\item Conversely, every primitive algebraic compactification of $X$ having a point with minimally elliptic singularity is isomorphic to $\xomegaprimeomegadtheta$ for some $\vec \theta \in (\cc^*)^{m+1}$, $d > 0$, and $\vec \omega$ satisfying properties \eqref{elliptic-1}--\eqref{elliptic-3}. 
\end{enumerate}
\end{enumerate}
\end{cor}

\begin{proof}
Assertion \eqref{elliptic-0} immediately follows from combining \cref{gorenstein-cor} with assertion \eqref{elliptic} of \cref{singularity-cor}. For assertion \eqref{ellipticc}, pick a primitive algebraic key sequence $\vec \omega' = (\omega'_0, \ldots, \omega'_{n+1})$ in the normal form and $\vec \theta \in \ntorus$ such that $\xomegaprimetheta$ has a point with minimally elliptic singularity. Assertion \ref{elliptic-0} implies that  
\begin{align*}
 \omega'_{n+1} =  \sum_{k=1}^n (\alpha_k -1)\omega'_k -\omega'_0 - 1
\end{align*}
In particular, this implies (due to primitiveness of $\vec \omega'$) that $n \geq 1$. Note that $\omega'_0, \ldots, \omega'_n$ completely determines $\omega'_{n+1}$; one has to only ensure that 
$$\alpha_{n+1}(\sum_{k=1}^n (\alpha_k -1) \omega_k - \omega_0) - 1  
= \sum_{k=1}^n (\alpha_k -1)\omega'_k -\omega'_0 - 1 > 0$$
where $\alpha_{n+1} = \gcd(\omega'_0, \ldots, \omega'_n)$ and $\omega_k := \omega'_k/\alpha_{n+1}$, $k = 0, \ldots, n$. Since $n \geq 1$, it follows that $\vec \omega := ( \omega_0, \ldots,  \omega_n)$ is a primitive algebraic key sequence which additionally satisfies $\alpha_n \omega_n \in \zz_{\geq 0} \langle \omega_0, \ldots, \omega_{n-1} \rangle$. Assertion \eqref{ellipticc} now follows in a straightforward manner with $m := n-1$ and $d := \alpha_{n+1}$.
\end{proof}

%

\subsection{Proof of \cref{canonical-thm}} \label{canonical-subsection}
Consider the notations of \cref{projective-embedding}. Let $W := \cc^{n+2}$ with coordinates $( w,  y_1, \ldots, y_{n+1})$. Then $\WP\setminus V(y_0)$ is the quotient of $ W$ by the action of the cyclic group of $\omega_0$ elements given by
\begin{align}
\zeta \cdot ( w,  y_1, \ldots, y_{n+1}) := (\zeta  w,\zeta^{\omega_1} y_1, \ldots,\zeta^{\omega_{n +1}} y_{n+1}) \label{zeta-w-action}
\end{align}
where $\zeta$ is a primitive $\omega_0$-th root of unity. Let $ \pi:  W \to \WP\setminus V(y_0)$ be the quotient map 
$$ ( w,  y_1, \ldots, y_{n +1}) \mapsto [ w:1:  y_1: \cdots : y_{n+1}]$$
Let $ \tilde Y :=  \pi^{-1}(Y)$ and $\tilde C := \pi^{-1}(C)$, where $Y, C$ are as in \cref{projective-embedding}. Then $\tilde Y$ is defined in $W$ by $G_1, \ldots, G_k$ from \eqref{projective-definition}. A computation shows that the matrix of partial derivatives $\partial  G_i/\partial  y_j$, $1 \leq i, j \leq n$ has non-zero determinant on $\tilde C =  \tilde Y \cap \{ w = 0\}$ and therefore $\tilde Y$ is non-singular at every point on $\tilde C$. Moreover, the only ramification points of $ \pi|_{ \tilde Y}$ on $\tilde C$ are the points with zero $ y_{n+1}$-coordinate and $ \pi$ maps every such point to the same point on $C$ (namely the point $P_0$ from \cref{projllary}). \\

It follows from the preceding paragraph that for every $P \in C \setminus \{P_0\}$ and every $Q \in  \pi^{-1}(P)$, $ \pi|_Y$ restricts to an isomorphism near $Q$. It follows that $( w,  y_{n+1})$ defines a system of coordinates near every $P \in C \setminus \{P_0\}$. Now note that $x = y_0/w^{\omega_0} = 1/ w^{\omega_0}$ and the last key form for $\delta$ is $g_{n+1}(x,y) =  y_{n+1}/ w^{\omega_{n+1}}$. It follows that
\begin{align}
d w &= -( w^{\omega_0+1}/\omega_0)dx \notag\\
d y_{n+1} &=  w^{\omega_{n+1}}dg_{n+1} + g_{n+1}d\left( w^{\omega_{n+1}}\right),\quad \text{and} \notag\\
d w d y_{n+1} &= -\frac{ w^{\omega_0+\omega_{n+1}+1}}{\omega_0}\frac{\partial g_{n+1}}{\partial y} dxdy,\quad \text{so that} \notag \\
A_\delta &= \pole_C(w^{\omega_0+\omega_{n+1}+1}\frac{\partial g_{n+1}}{\partial y}) + 1 
= -\omega_0-\omega_{n+1} + \delta\left(\partial g_{n+1}/\partial y  \right) \label{can-1}
\end{align}
where we wrote $\delta$ for $\deltaomegatheta$. We now compute $\delta\left(\partial g_{n+1}/\partial y \right)$ using a result of \cite{kuo-parusinski-pol-arc}. 

\begin{defn} \label{modinition}
Let $q \in \qq$, $f \in \cc[u,v]$, and $\phi(u)$ be a Puiseux series in $u$. We say that 
\begin{enumerate}
\item $\phi(u)$ is a {\em root mod $q$} of $f=0$ iff there exists a Puiseux series $\bar \phi(u)$ such that $f(u,\bar \phi(u)) \equiv 0$ and $\phi(u) - \bar \phi(u) = cu^q + \hot$ for some $c \in \cc$ (where $\hot$ stands for terms with higher order in;
\item $\phi(u)$ is a {\em root exactly mod $q$} of $f=0$ iff there exists a Puiseux series $\bar \phi(u)$ such that $f(u,\bar \phi(u)) \equiv 0$ and $\phi(u) - \bar \phi(u) = cu^q + \hot$ for some $c \in \cc$, $c \neq 0$.
\end{enumerate}
A mod $q$ root $\phi(u)$ of $f=0$ has {\em multiplicity $m$} iff there are exactly $m$ distinct Puiseux series $\bar \phi_1(u), \ldots, \bar \phi_m(u)$ such that $f(u,\bar \phi_k(u)) \equiv 0$ and $\phi(u) - \bar \phi_k(u) = c_ku^q + \hot$ for some $c_k \in \cc$, $1 \leq k \leq m$. Similarly, $\phi(u)$ is an {\em exactly mod $q$ root of multiplicity $m$} of $f=0$ iff there are exactly $m$ distinct Puiseux series $\bar \phi_1(u), \ldots, \bar \phi_m(u)$ which satisfy the conditions of the preceding sentence with $c_k \neq 0$, $1 \leq k \leq m$.
\end{defn}

\begin{rem}
Note that the notion of mod $q$ roots and exactly mod $q$ roots have (obvious) analogues in the case of descending Puiseux series: namely in Definition \ref{modinition} replace every occurrence of `Puiseux series' with `descending Puiseux series', and `$\hot$' with `$\ldt$', where as usual, $\ldt$ stands for terms with lower degree (in $u$). 
\end{rem}

\begin{thm}[{\cite[Theorem 1.1]{kuo-parusinski-pol-arc}}] \label{mod-thm}
Let $q \in \qq_{>0}$ and $\phi(u)$ be a Puiseux series which is a mod $q$ root of $f = 0$ of multiplicity $m \geq 1$. Then $\phi(u)$ is a mod $q$ root of $\partial f/\partial v = 0$ of multiplicity $m-1$.
\end{thm} 

\begin{rem}
It was assumed throughout \cite{kuo-parusinski-pol-arc} that $f$ is {\em mini-regular} in $v$, i.e.\ if $d := \ord(f)$ then there is a monomial term in $f$ of the form $cv^d$ with $c \in \cc^*$. However, the proof of Theorem 1.1 of \cite{kuo-parusinski-pol-arc} does not use this assumption.
\end{rem}

\begin{cor} \label{mod-dcor}
Let $g\in \cc[x, x^{-1},y]$ and $\phi(x)$ be a descending Puiseux series in $x$ which is a mod $q$ root of $g = 0$ of multiplicity $m \geq 1$. Then $\phi(x)$ is a mod $q$ root of $\partial g/\partial y = 0$ of multiplicity $m-1$.
\end{cor} 

\begin{proof}
Consider the (birational) change of coordinates $(u,v) = (1/x, y/x^d)$, where $d \gg 1$. Let $\tilde g := u^{d \deg_y(g)} g(1/u, v/u^d) \in \cc[u,v]$. Now note that $y = \phi(x)$ is a mod $q$ root of $g =0$ iff $v = u^d\phi(1/u)$ is a mod $d-q$ root of $\tilde g$. \cref{mod-thm} implies that $v = u^d\phi(1/u)$ is a mod $d-q$ root of $\partial \tilde g/\partial v$ of multiplicity $m-1$. Since $\partial \tilde g/\partial v = u^{d (\deg_y(g) - 1)} \partial g/\partial y$, it follows that $y = \phi(x)$ is a mod $q$ root of $\partial g/\partial y =0$ of multiplicity $m-1$, as required.
\end{proof}

\begin{cor} \label{mod-cor}
Let $f \in \cc[x,x^{-1},y]$ and $\phi(x)$ be a descending Puiseux series in $x$. Let the multiplicity of $\phi(x)$ as a mod $q$ and exactly mod $q$ root of $f = 0$ be respectively $m'$ and $m$ (so that $m' \geq m$). Assume $m' > m \geq 1$. Then the multiplicity of $\phi(x)$ as an {\em exactly mod $q$} root of $\partial f/\partial y = 0$ is also $m$.
\end{cor}

\begin{proof}
Let $n := m' - m \geq 1$. Then for all sufficiently small $\epsilon > 0$, $\phi(x)$ is a mod $(q-\epsilon)$ root of $f=0$ of multiplicity $n$, so that \cref{mod-dcor} implies that it is a mod $(q-\epsilon)$ root of $\partial f/\partial y =0$ of multiplicity $n-1$. On the other hand \cref{mod-dcor} also implies that $\phi(x)$ is a mod $q$ root of $\partial f/\partial y =0$ of multiplicity $m'-1$. It follows that $\phi(x)$ is an exactly mod $q$ root of $\partial f/\partial y =0$ of multiplicity $m'-1 - (n-1) = m$.
\end{proof}

\begin{cor} \label{partially-global}
Let $f \in \cc[x,x^{-1}, y]$ be monic in $y$ and have an analytically irreducible branch at infinity for which $|x| \to \infty$; in other words, assume that 
\begin{align*}
f = \prod_{\parbox{1.5cm}{\scriptsize{$\phi_i$ is a con\-ju\-ga\-te of $\phi$}}}\mkern-27mu \left(y - \phi_{i}(x)\right),
\end{align*}
where $\phi(x)$ is a descending Puiseux series in $x$. Let the Puiseux pairs of $\phi$ be $(\tilde q_1, \tilde p_1), \cdots, (\tilde q_k,\tilde p_k)$, $k \geq 1$. Set $\tilde p_{k+1} := 1$. Then 
\begin{gather}
\min\{\deg_x(\phi(x) - \psi(x)): \psi(x)\ \text{is a descending Puiseux root of}\ \partial f/\partial y = 0\} 
	= \frac{\tilde q_k}{\tilde p_1 \cdots \tilde p_k} \label{partially-global-1} \\
\deg_x\left((\partial f/\partial y)|_{y=\phi(x)}\right) 
	= \sum_{j=1}^k  (\tilde p_j-1)\tilde p_{j+1} \cdots \tilde p_{k+1} \frac{\tilde q_j}{\tilde p_1 \cdots \tilde p_j}. \label{partially-global-2}
\end{gather}	
\end{cor}

\begin{proof}
Let $\tilde p := \tilde p_1 \cdots \tilde p_k$. Then $f = 0$ has precisely $\tilde p$ descending Puiseux roots in $x$, and for each $j$, $1 \leq j \leq k$, the multiplicity of $\phi(x)$ as a mod $\tilde q_j/(\tilde p_1 \cdots \tilde p_j)$ root (resp.\ as an {\em exactly} mod $\tilde q_j/(\tilde p_1 \cdots \tilde p_j)$ root) of $f = 0$ is $\tilde p_j\tilde p_{j+1} \cdots \tilde p_{k+1}$ (resp.\ $(\tilde p_j - 1)\tilde p_{j+1} \cdots \tilde p_{k+1}$). \cref{mod-cor} implies that for each $j$, $1 \leq j \leq k$, the multiplicity of $\phi(x)$ as an exactly mod $\tilde q_j/(\tilde p_1 \cdots \tilde p_j)$ root of $\partial f/\partial y = 0$ is $(\tilde p_j-1)\tilde p_{j+1} \cdots \tilde p_{k+1}$. The corollary follows since $\sum_{j=1}^k  (\tilde p_j-1)\tilde p_{j+1} \cdots \tilde p_{k+1} = \tilde p - 1$, and since $\partial f/\partial y$ has precisely $\tilde p-1$ descending Puiseux roots in $x$.
\end{proof}

Now we go back to the proof of \cref{canonical-thm}. Let $\tilde \phi_\delta(x,\xi):= \phi_\delta(x) + \xi x^{r_\delta}$ be the generic descending Puiseux series of $\delta$ and $(q_1, p_1), \ldots, (q_{l+1}, p_{l+1})$ be the formal Puiseux pairs of $\tilde \phi_\delta$. At first consider the case that $l=0$, i.e.\ $\phi_\delta \in \cc((x))$. Then $g_{n+1} = y - \phi_\delta(x)$, so that $\delta\left(\partial g_{n+1}/\partial y \right) = 0$. On the other hand, $\alpha_k = 1$ for all $k \geq 1$, so that \eqref{log-discrepancy-formula} follows from \eqref{can-1}. Now assume $l \geq 1$. Recall from \eqref{phi-delta-defn} that
\begin{align*}
\delta\left(\partial g_{n+1}/\partial y \right) = p\deg_x\left(\left(\partial g_{n+1}/\partial y \right)|_{y = \tilde \phi_\delta(x,\xi)} \right),
\end{align*}
where $p := p_1 \cdots p_{l+1} = \delta(x)$. Let $\phi(x)$ be the descending Puiseux root of $g_{n+1}$ from assertion \eqref{last-factorization} of \cref{last-curve-prop}. \cref{last-curve-prop} implies that $g_{n+1}$ satisfies the assumption of \cref{partially-global}. Since $\phi_\delta$ is a mod $r_\delta$ root of $g_{n+1}$ and since $r_\delta < q_l/(p_1 \cdots p_l)$, identity \eqref{partially-global-1} implies that 
\begin{align*}
\deg_x\left(\left(\partial g_{n+1}/\partial y \right)|_{y = \tilde \phi_\delta(x,\xi)} \right)
	&= \deg_x\left(\left(\partial g_{n+1}/\partial y \right)|_{y =  \phi(x)} \right).
\end{align*}
It then follows from \eqref{partially-global-2} that
\begin{align*}
\delta\left(\partial g_{n+1}/\partial y \right)
	&= p_1 \cdots p_l \sum_{k=1}^l (p_k-1)p_{k+1} \cdots p_{l+1} \frac{q_k}{p_1 \cdots p_k}
	= \omega_{n+1} - q_{l+1} = \sum_{k=1}^n (\alpha_k-1)\omega_k,
\end{align*}
where the last two equalities follow from \eqref{omega-and-chi-0} and \eqref{omega-and-chi-1}. The theorem then follows from identity \eqref{can-1}. \qed

\appendix

\section{Some properties of key forms} \label{key-appendix}
Let $\delta$ be a divisorial semidegree such that $\delta(x) > 0$, $\tilde \phi_\delta(x,\xi):= \phi_\delta(x) + \xi x^{r_\delta}$ be the associated generic descending Puiseux series, $g_0 = x, g_1 = y, \ldots, g_{n+1}$ be the corresponding sequence of key forms, and $\vec \omega := (\omega_0, \ldots, \omega_{n+1})$ be the key sequence. In this section we collect most of the properties of key forms used in this article. The assertions from \cref{last-curve-prop,basis-lemma} are used multiple times throughout the article; \cref{delta-1-cor} is used only in the proof of \cref{rigid-prop}. \\

The following notations are used in this section: we denote the formal Puiseux pairs (resp.\ characteristic exponents) of $\tilde \phi_\delta$ by $(q_1, p_1), \ldots, (q_{l+1},p_{l+1})$ (resp.\ $\chi_1, \ldots, \chi_{l+1}$); recall that $\chi_j  = q_j/(p_1 \cdots p_{j})$, $1 \leq j \leq l+1$. Also define $\alpha_1, \ldots, \alpha_{n+1}$ as in \cref{key-seqn}.



\begin{prop} \label{last-curve-prop}
\mbox{}
\begin{enumerate}
\item $\omega_j = \delta(g_j)$, $0 \leq j \leq n+1$.
\item the essential subsequence of $\vec \omega$ consists of $l+2$ elements, i.e.\ it is of the form $\vec \omega_e := (\omega_{i_0}, \ldots, \omega_{i_{l+1}})$. Moreover,
\begin{enumerate}
\item \label{alpha-p} $\alpha_{i_k} = p_k$, $1 \leq k \leq l+1$.
\item for each $k$, $1 \leq k \leq l+1$,
\begin{align}
\omega_{i_{k+1}}
		&= \sum_{j=1}^k (p_j -1)\omega_{i_j} + \omega_0 \chi_{k+1} \label{omega-and-chi-0} \\
		&=p_1 \cdots p_{l+1} \left(
		 \sum_{j=1}^k (p_j-1)p_{j+1} \cdots p_k\frac{q_j}{p_1 \cdots p_j} + \frac{q_{k+1}}{p_1 \cdots p_{k+1}}
		 \right)  \label{omega-and-chi-1}
\end{align}		
\end{enumerate}
\item \label{deg-y-g-n+1} For each $j$, $0 \leq j \leq n$, $g_{j+1}$ is monic in $y$ and $\deg_y(g_{j+1}) = \alpha_1 \cdots \alpha_j$; in particular, $\deg_y(g_{n+1}) = p_1 \cdots p_l$, which is also the polydromy order of $\phi_\delta$.	
\item \label{last-factorization} $g_{n+1}$ has a descending Puiseux factorization of the form  
\begin{align}
\begin{aligned}
&g_{n+1} = x^{m_{n+1}}\prod_{\parbox{1.5cm}{\scriptsize{$\phi_i$ is a con\-ju\-ga\-te of $\phi$}}}\mkern-27mu \left(y - \phi_{i}(x)\right),\quad \text{where $\phi$ satisfies} \\
&\phi(x) = \phi_\delta(x) + \text{terms of degree less than or equal to}\ r_\delta. 
\end{aligned} \label{minimal-expansion}
\end{align}
\end{enumerate} 
\end{prop}

\begin{proof}
All the assertions follow from \cite[Propositions 3.21 and 5.3]{algebraicity}. 
\end{proof}

\begin{lemma}[cf.\ {\cite[Fundamental Theorem, Section 8.5]{abhyankar-expansion}}] \label{basis-lemma}
Let $R := \cc[x,x^{-1},y]$. With the above notations, define
\begin{align*}
\scrI &:= \{(\beta_0, \ldots, \beta_{n+1}) \in \zz^{n+2}: 0 \leq \beta_j < \alpha_j,\ 1 \leq j \leq n,\ \text{and}\ \beta_{n+1} \geq 0\},\ \text{and}\\
\scrI' &:= \{(\beta_0, \ldots, \beta_{n+1})  \in \scrI: \beta_0 \geq 0\}. 
\end{align*}
For $\beta := (\beta_0, \ldots, \beta_{n+1}) \in \scrI$, we write $g_\beta$ for the corresponding `monomial' $g_0^{\beta_0} g_1^{\beta_1} \cdots g_{n+1}^{\beta_{n+1}}$ in $g_0, \ldots, g_{n+1}$. Define
\begin{align*}
\scrB &:= \{g_\beta: \beta \in \scrI\},\ \text{and} \\
\scrB' &:= \{g_\beta: \beta \in \scrI' \}.
\end{align*}
Then 
\begin{enumerate}
\item \label{basis} $\scrB$ is a $\cc$-vector space basis of $R$.
\item \label{delta-basis} Let $g \in R$. Write $g$ as $g = \sum_{\beta \in \scrI} a_\beta g_\beta$, where each $a_\beta \in \cc$. Then 
\begin{align}
\delta(g) = \max\{\delta(g_\beta): a_\beta \neq 0\}.
\end{align}
\item \label{basis-plus} Let $R'$ be the $\cc$-vector space spanned by $\scrB'$. Then $R' \supseteq \cc[x,y]$. In particular, for every $f \in \cc[x,y]\setminus \{0\}$, $\delta(f)$ is in the semigroup generated by $\delta(g_0), \ldots, \delta(g_{n+1})$.
\end{enumerate}
\end{lemma}

\begin{proof}
We claim that $\scrB$ spans $R$ as a vector space over $\cc$. Indeed, it suffices to show that for each monomial of the form $x^dy^e$ where $e \geq 0$, there is $g \in \scrB$ such that $\deg_y(x^dy^e-g) < e$. It follows from the definition of $\alpha_j$'s that $e$ can be expressed as $\sum_{j=1}^{n+1} \beta_j\alpha_1 \cdots \alpha_{j-1}$ with $0 \leq \beta_j \leq \alpha_j$ for $j = 1, \ldots, n$. The claim is proved by taking $\beta := (d, \beta_1, \ldots, \beta_{n+1})$ and $g := g_\beta$. \\

Now pick $d \in \zz$ and pairwise distinct elements $\beta^1, \ldots, \beta^m \in \scrI$ such that $\delta(g_{\beta^k}) = d$ for each $k$. Let $(a_1, \ldots, a_m) \in \cc^m\setminus \{0\}$ and $g := \sum_k a_k g_{\beta^k}$.

\begin{proclaim} \label{delta-g-claim}
$\delta(g) = d$; in particular, $g \neq 0$.
\end{proclaim}

\begin{proof}
\cref{phi-omega-algorithm} implies that
\begin{align}
g_j|_{y = \tilde \phi_\delta(x,\xi)} 
	= \begin{cases}
		c_j x^{\omega_j/\omega_0} + \ldt				& \text{for}\ 0 \leq j \leq n,\\
		h(\xi)x^{\omega_{n+1}/\omega_0} + \ldt	& \text{for}\ j = n + 1,
	  \end{cases} \label{g-j-substitution}
\end{align}
for some $c_0, \ldots, c_{n+1} \in \cc^*$ and a non-constant polynomial $h \in \cc[\xi]$, where $\ldt$ denotes terms with lower degree in $x$. Consequently, 
\begin{align*}
& g|_{y = \tilde \phi_\delta(x,\xi)} =  c(\xi) x^{d/\omega_0} + \lot,\ \text{where}\\
& c(\xi) := \sum_k a_k (h(\xi))^{\beta^k_{n+1}} \prod_{j=0}^n c_j^{\beta^k_j}
\end{align*}
If $c(\xi) \not\equiv 0$, then $\delta(g) = d$ and we are done. So assume $c(\xi) \equiv 0$. Then there must exist $k \neq k'$ such that $\beta^k_{n+1} = \beta^{k'}_{n+1}$ and $\sum_{j=0}^n \beta^k_j\omega_j = \sum_{j=0}^n \beta^k_j\omega_j = d - \beta^k_{n+1}\omega_{n+1}$. Let $j'$ be the maximal integer $\leq n$ such that $\beta^k_{j'} \neq \beta'_{j'}$. Then it follows that $(\beta^k_{j'} -\beta^{k'}_{j'})\omega_{j'}$ is in the group generated by $\omega_0, \ldots, \omega_{j'-1}$. Since $|\beta^k_{j'} -\beta^{k'}_{j'}| < \alpha_{j'}$, this is impossible by definition of $\alpha_{j'}$. Consequently $c(\xi) \not\equiv 0$, which proves the claim. 
\end{proof}

Since $\scrB$ spans $R$, it is straightforward to see that assertions \ref{basis} and \ref{delta-basis} follow from \cref{delta-g-claim}. Assertion \ref{basis-plus} follows from combining assertion \eqref{deg-y-g-n+1} of \cref{last-curve-prop} and \cite[Theorem 2.13]{abhyankar-expansion}.
\end{proof}

\begin{cor} \label{delta-1-cor}
Assume that $\delta(g_j) > 0$ for all $j$, $0 \leq j \leq n+1$, and that there exists $f \in \cc[x,y]$ such that $\delta(f) = 1$. Then  
\begin{enumerate}
\item \label{delta-j} there exists $j_* \leq n + 1$ such that $\delta(g_{j_*}) = 1$,
\item \label{j-form}  
\begin{enumerate}
\item either $j_* = n + 1$, or 
\item \label{j* < n+1} $\alpha_{j_*} > 1$ and for all $k$, $j_* < k \leq n+1$, $g_k$ is of the form 
\begin{align*}
g_k = g_{j_*+1} - h_k(g_{j_*}),
\end{align*} 
where $h_k$ is a polynomial in one variable with $\deg(h_k) < \alpha_{j_*}$. 
\end{enumerate}

Moreover, $\delta(g_i) > 1$ for all $i < j_*$.
\item \label{f-form}
\begin{enumerate}
\item \label{f-form-1} either $f = ag_{j_*} + b$, for some $a \in \cc^*$, $b \in \cc$, in which case $g_{j_*}$ is a polynomial, or 
\item \label{f-form-2} $j_* < n+ 1$, $\delta(g_{n+1}) = 1$, and $f = ag_{j_*} + bg_{n+1} + c$ for some $b \in \cc^*$, $a,c \in \cc$. In this case both $g_{n+1}$ and $g_{j_*}$ are polynomials.
\end{enumerate}

In every case the curve $f = 0$ has only one place at infinity. 
\end{enumerate} 
\end{cor}

\begin{proof}
Assertion \ref{delta-j} follows from assertion \eqref{basis-plus} of \cref{basis-lemma}. But then an inspection of \cref{phi-omega-algorithm} immediately yields assertion \eqref{j-form}. We now prove assertion \eqref{f-form}. Assume that \eqref{f-form-1} does not hold. Then assertions \eqref{delta-basis} and \eqref{basis-plus} of \cref{basis-lemma} and the assertion \eqref{j-form} imply that $j_* < n+1$ and
\begin{align}
f = ag_{j_*} + bg_{n+1} + c \label{f-form-2-again}
\end{align} 
for some $b \in \cc^*$, $a,c \in \cc$. Then \eqref{g-j-substitution} implies that
\begin{align}
f|_{y = \tilde \phi_\delta(x,\xi)} =  (ac_{j_*} + bh(\xi)) x^{1/\omega_0} + \ldt, \label{f-sub}
\end{align}
for some $c_{j_*}\in \cc^*$. Since the coefficient of $x^{1/\omega_0}$ in the right hand side of \eqref{f-sub} is a non-constant polynomial in $\xi$, it follows that there is a descending Puiseux root $\phi$ of $f$ such that $\deg_x(\phi - \phi_\delta) \leq r_\delta$. It follows that 
\begin{align*}
\deg_y(f) &\geq \text{polydromy order of $\phi$} 
	 \geq \text{polydromy order of $\phi_\delta$} 
	 = \deg_y(g_{n+1})
\end{align*}
(the last identity uses assertion \eqref{deg-y-g-n+1} of \cref{last-curve-prop}). On the other hand, assertion \eqref{j* < n+1}, identity  \eqref{f-form-2-again}, and assertion \eqref{deg-y-g-n+1} of \cref{last-curve-prop} together imply that $\deg_y(f) = \deg_y(g_{n+1})$. It follows that 
\begin{align}
\deg_y(f) = \deg_y(g_{n+1}) > \deg(g_{j_*}) \label{degree}
\end{align}
 and $\phi$ is in fact the only descending Puiseux root of $f$, i.e.\ the descending Puiseux expansion of $f$ is of the form 
\begin{align*}
f = dx^m  \prod_{\parbox{1.5cm}{\scriptsize{$\phi_{i}$ is a con\-ju\-ga\-te of $\phi$}}}\mkern-27mu \left(y - \phi_{i}(x)\right) 
\end{align*}
for some $d \in \cc^*$. Let $p:= \deg_y(f) = \deg_y(g_{n+1})$. Then the coefficient of $y^p$ in $f$ is $cx^m$. Since $g_{n+1}$ is monic in $y$ (assertion \eqref{deg-y-g-n+1} of \cref{last-curve-prop}), identities \eqref{f-form-2-again} and \eqref{degree} then imply that $m = 0$. This implies that $f = 0$ has only one place at infinity. It then follows from \cite[Theorem 4.3]{algebraicity} that $g_j$ is a polynomial for each $j$, $0 \leq j \leq n+1$. This proves that assertion \eqref{f-form-2} holds. It remains to show that $f = 0$ has only one place at infinity in the case of assertion \eqref{f-form-1}. But since in that case $g_{j_*}$ is a polynomial, this again follows from \cite[Theorem 4.3]{algebraicity}.
\end{proof}

\section{Dependence of $\vec\omega, \vec \theta$ on $\phiomegatheta$} \label{theta-appendix}
Let $\tilde \phi(x,\xi) := \sum_\beta a_\beta x^{\beta} + \xi x^{r}$ be a generic descending Puiseux series. Let $\vec \omega = (\omega_0, \ldots, \omega_{n+1})$ and $\vec \theta = (\theta_1, \ldots, \theta_n) \in \ntorus$ be the output when \cref{phi-omega-algorithm} is run with input $\tilde \phi$. In this section we establish some relations between $\tilde \phi$ and $\vec\omega, \vec \theta$ that we use later to prove the basic properties of normal forms of key sequences. At first we set up some notations to be used throughout this section:

\begin{notation} \label{puiseux-notation}
\mbox{}

\begin{itemize}
\item $(q_1, p_1), \ldots, (q_{l+1},p_{l+1})$, $l \geq 0$, are the formal Puiseux pairs and $\chi_k =  q_k/(p_1 \cdots p_k)$, $k = 1, \ldots, l+1$, are the formal characteristic exponents of $\tilde \phi$ (\cref{generic-descending-defn}). 
\item $p := p_1 \cdots p_{l+1} = \omega_0$. 
\item $\Etildephi := \{\beta: a_\beta \neq 0\}$. 
\item  $\vec \omega_e := (\omega_{i_0}, \ldots, \omega_{i_{l+1}})$ is the essential subsequence of $\vec \omega$.
\item $\alpha_1, \ldots,  \alpha_{n+1}$ are as in Definition \ref{key-seqn}. Recall that $\alpha_{i_k} = p_k$, $1 \leq i \leq l+1$ (assertion \eqref{alpha-p} of \cref{last-curve-prop}).
\end{itemize} 
\end{notation}

\subsection{Results to be used in the proof of properties of normal forms of key sequences}

\begin{thm} \label{theta-thm}
Fix $k,i$ such that $0 \leq k \leq l$ and $\max\{1,i_k\} \leq i < i_{k+1}$. Let $\beta_{i,0}, \ldots, \beta_{i,i-1}$ be as in Remark \ref{unique-remark}. Set $\beta_{i,i} := 0$.
\begin{enumerate}
\item $\theta_i \in \cc[a_\beta: \beta \in \Etildephi][a_{\chi_1}^{-1}, \ldots, a_{\chi_l}^{-1}]$.
\item \label{theta-k2} $\theta_i$ is homogeneous in $a_\beta$'s of degree $\mu_i := p_1 \cdots p_k - \sum_{j=1}^k p_1 \cdots p_{j-1} \beta_{i,i_{j}}$.
\item \label{theta-k3} Let $\eta$ be the weighted degree on $\cc(a_\beta: \beta \in \Etildephi)$ so that the weight of each $a_\beta$ is $\beta$. Then $\theta_i$ is weighted homogeneous with respect to $\eta$ with $\eta(\theta_i) = \beta_{i,0}$.
\end{enumerate}
\end{thm}

\begin{defn} \label{compatible}
We say that a rational number $\beta$ is {\em \reachable{$\tilde \phi$}} if 
\begin{defnlist}
\item \label{trivially-reachable}  either $\beta \leq r = \chi_{l+1}$, or
\item \label{non-trivially-reachable} $\beta > r$ and $\beta\in \zz + \zz\langle \beta' \in \Etildephi: \beta' > \beta \rangle = \zz + \zz\langle \chi_k: 1 \leq k \leq \hat k (\beta) \rangle$, where $\hat k(\beta)$ is defined in \eqref{hat-k-beta}. 
\end{defnlist}
In the case that \ref{non-trivially-reachable} holds, we say that $\beta$ is {\em non-trivially}  \reachable{$\tilde \phi$}. 
\end{defn}

\begin{rem}[Motivation for the term ``key compatible'']
If $\beta'$ is \reachable{$\tilde \phi$}, then for all $c \in \cc$, the formal characteristic exponents of $\tilde \phi + c x^{\beta'}$ and $\tilde \phi$ are the same, so that the essential key sequences of the corresponding semidegrees are identical. If condition \ref{trivially-reachable} of key compatibility holds, then changing the coefficient of $x^\beta$ in $\tilde \phi$ does not have any effect on the semidegree; hence this is the `trivial' case.
\end{rem}

Let $\beta'_1 > \cdots > \beta'_s$ be rational numbers which are \reachable{$\tilde \phi$}. For $c_1, \ldots, c_s \in \cc$, consider the generic descending Puiseux series
\begin{align*}
\tilde \tau_{(c_1, \ldots, c_s)} :=\tilde \phi + \sum_{j=1}^s (c_j - a_{\beta'_j}) x^{\beta'_j} 
\end{align*}
Let $\sigma$ be the semidegree and $h_0, h_1, \ldots$ be the sequence of key forms corresponding to $\tilde \tau_{(c_1, \ldots, c_s)}$. Pick the largest integer $\hat i$ such that $h_i = g_i$ for all $i = 0,1, \ldots, \hat i$ for all choices of $c_1, \ldots, c_s$.

\begin{thm} \label{c-thm}
\mbox{}
\begin{enumerate}
\item \label{c-1} Assume $\beta'_1$ is non-trivially \reachable{$\tilde \phi$}. Then there is a unique element $c_{\tilde \phi}(\beta'_1) \in \cc$ depending only on $\beta'_1$ and $(a_\beta: \beta \in \Etildephi,\ \beta > \beta'_1)$ such that 
\begin{enumerate}
\item if $c_1 \neq c_{\tilde \phi}(\beta'_1) $, then $\sigma(h_{\hat i}) = \hat \omega_{\beta'_1}$.
\item if $c_1= c_{\tilde \phi}(\beta'_1)$, then $\sigma(h_{\hat i}) < \hat \omega_{\beta'_1}$.
\end{enumerate} 

\item \label{c-2} Let $\scrE' \subseteq \Etildephi$, $\tilde \psi(x,\xi) := \sum_\beta a'_\beta x^\beta + \xi x^{r'}$ be a generic descending Puiseux series, and $\beta' \in \qq$ be non-trivially \reachable{$\tilde \phi$}. Assume
\begin{enumerate}
\item there are $a,b \in \cc^*$ and $m$, $1 \leq m \leq N$, such that $a'_\beta = ba^{-\omega_0\beta}$ for all $\beta \in \scrE'$.
\item $\{\beta \in \Etildephi: \beta > \beta'\} \subseteq \scrE'$. 
 \end{enumerate}
 Then $c_{\tilde \psi}(\beta') = ba^{-\omega_0\beta'}c_{\tilde \phi}(\beta')$
 \end{enumerate}
\end{thm}	

\begin{thm}\label{gcd-difference}
For each $i,k$ such that $1 \leq k \leq l$ and $i_k < i \leq i_{k+1}$, define
\begin{align*}
\omega^*_i := \omega_i - \sum_{j=2}^{k}(\alpha_{i_j}-1)\omega_{i_j} - \alpha_{i_1}\omega_{i_1}
\end{align*}
Then $\gcd\{\omega^*_i: i_1 < i \leq n\} = \gcd\{\omega_0\beta - \omega_{i_1}: \beta \in \Etildephi,\ \beta < \chi_1\} $. 
\end{thm}

\subsection{Proof of \cref{theta-thm,c-thm,gcd-difference}}
Let $g_0, \ldots, g_{n+1}$ be the key forms associated to $\tilde \phi$ and $\tilde \omega_i := \omega_i/\omega_0$, $0 \leq i \leq n+1$. It follows from \cref{phi-omega-algorithm} that for each $i$, $1 \leq i \leq n$, 
\begin{align}
g_i|_{y = \tilde \phi(x,\xi)} =  \sum_{r_i < \epsilon \leq \tilde \omega_i} a_{i,\epsilon}x^\epsilon + \rho_i(\xi) x^{r_i} + \ldt
\end{align}
where 
\begin{prooflist}
\item $a_{i,\tilde \omega_i}$ is a non-zero element of $\cc(a_\beta: \beta \in \Etildephi)$ (note that $a_{i,\tilde \omega_i} = \tilde a_i$ in the notation of \cref{phi-omega-algorithm});
\item $r_i < \tilde \omega_i$;
\item $\rho_i$ is a non-constant polynomial in $\xi$ with coefficients in $\cc(a_\beta: \beta \in \Etildephi)$. 
\end{prooflist}
\Cref{a-lemma} extracts some information about $a_{i,\epsilon}$'s and $\rho_i$'s. To state it we need the following definitions: for each $\epsilon \in \qq$ and $k = 0, \ldots, l$, 
\begin{align}
\tilde \omega_{k,\epsilon} &:= \sum_{j=1}^k (p_j-1)\tilde \omega_{i_j} + \epsilon,
\label{tilde-omega-k-beta} \\
\tilde \beta_{k,\epsilon} &:= \epsilon - \sum_{j=1}^k (p_j-1)\tilde \omega_{i_j}.
\label{tilde-beta-k-epsilon} 
\end{align}
Note that
\begin{prooflist}[resume]
\item \label{inverse} $\tilde \omega_{k,\cdot}$ and $\tilde \beta_{k,\cdot} $ are inverse operations. 
\item $\tilde \omega_{\hat k(\beta),\beta} = \hat \omega_\beta/\omega_0$ in the notation of \eqref{hat-k-beta} and \eqref{hat-omega-beta}. In particular, 
\begin{align}
\tilde \omega_{k, \chi_{k+1}} = \omega_{i_{k+1}}/\omega_0  = \tilde \omega_{i_{k+1}}. \label{tilde-omega-chi}
\end{align}
\end{prooflist}
We list some inequalities satisfied by $\tilde \omega_{k,\cdot}$ in \cref{inequality-lemma} below; these are used in the proof of \cref{a-lemma}.

\begin{lemma} \label{inequality-lemma}
\mbox{}
\begin{enumerate}
\item \label{assertion-1} Fix $k$, $1 \leq k \leq l$. Pick $0 = j_0 <  j_1 < \cdots < j_m \leq l$ and $\beta_{j_0}, \ldots, \beta_{j_m} \in \qq$ such that $\sum_{s=0}^m \beta_{j_s} \omega_{i_{j_s}} \leq p_k\omega_{i_k}$. Assume
\begin{defnlist}
\item $m \geq 1$
\item $\sum_{s=1}^m(\chi_{j_s} - \epsilon) 	> \chi_k - \epsilon$. 
\end{defnlist}
Then $\tilde \omega_{k,\epsilon} > \beta_{j_0} + \sum_{s=1}^m( (\beta_{j_s}-1)\tilde \omega_{i_{j_s}} + \tilde \omega_{j_s-1, \epsilon})$. 
\item \label{assertion-2} Let $\beta \in \qq$ be such that $\tilde \omega_{k,\beta} \leq \tilde \omega_{i_{k+1}}$. Then $\tilde \omega_{j,\beta} < \tilde \omega_{i_{j+1}}$ for all $j = 0, \ldots, k-1$. 
\end{enumerate}
\end{lemma}

\begin{proof}
For the first assertion note that 
\begin{align*}
 \tilde \omega_{k,\epsilon} - \beta_{j_0} &- \sum_{s=1}^m (\beta_{j_s}-1)\tilde \omega_{i_{j_s}} - \sum_{s=1}^m \tilde \omega_{j_s-1, \epsilon} \\
	&= \sum_{j=1}^k (p_j-1)\tilde \omega_{i_j} + \epsilon - \sum_{s=0}^m \beta_{j_s} \omega_{i_{j_s}}  + \sum_{s=1}^m(\tilde \omega_{i_{j_s}} - \tilde \omega_{j_s-1,\epsilon}) \\
	&= (p_k\tilde \omega_{i_k} - \sum_{s=0}^m \beta_{j_s} \omega_{i_{j_s}} ) - (\tilde \omega_{i_k} - \tilde \omega_{k-1,\epsilon}) + \sum_{s=1}^m(\tilde \omega_{i_{j_s}} - \tilde \omega_{j_s-1,\epsilon}) \\
	&= (p_k\tilde \omega_{i_k} - \sum_{s=0}^m \beta_{j_s} \omega_{i_{j_s}} ) - (\chi_k - \epsilon) + \sum_{s=1}^m(\chi_{j_s} - \epsilon) 	\\
	& > 0.
\end{align*}
For the second assertion, note that 
\begin{align*}
\tilde \omega_{k-1,\beta} 
	&= \tilde \omega_{k,\beta} - (p_k -1)\tilde \omega_{i_k} 
	\leq \tilde \omega_{i_k} - (p_k\tilde \omega_{i_k} - \tilde \omega_{i_{k+1}})
	< \tilde \omega_{i_k} 
\end{align*}
since $p_k\tilde \omega_{i_k}> \tilde \omega_{i_{k+1}}$. Now continue in the same way with $j = k-2$ and so on.
\end{proof}

\begin{lemma} \label{a-lemma}
Let $\eta$ be the weighted degree on $\cc(a_\beta: \beta \in \Etildephi)$ which assigns weight $\beta$ to $a_\beta$ for all $\beta \in \Etildephi$. Fix $k,i,\epsilon$ such that $0 \leq k \leq l$, $i_k < i \leq i_{k+1}$ and $r_i < \epsilon \leq\tilde \omega_i$. Then
\begin{enumerate}
\item \label{max-assertion} Let $\beta \in \Etildephi$ such that $\tilde \omega_{k,\beta} \leq \tilde \omega_i$ (in particular, $\beta > \chi_k$). Then 
\begin{align}
\tilde \omega_{k,\beta} = \max\{\epsilon': a_{i,\epsilon'}\ \text{depends non-trivially on}\ a_\beta\} 
\label{tilde-omega-k-beta-max}
\end{align}
\item \label{a-i-epsilon} If $k = 0$, then $a_{i,\epsilon} = a_\epsilon$. For $k \geq 1$, $a_{i,\epsilon} \in \cc[a_\beta: \beta \in \Etildephi, \chi_1 \geq \beta \geq \tilde \beta_{k,\epsilon}][a_{\chi_1}^{-1}, \ldots, a_{\chi_k}^{-1}]$.
\item \label{tilde-beta-k-epsilon-assertion} Assume $\epsilon = \tilde \omega_{k,\beta}$ for some $\beta \in \Etildephi$ (in other words, $\beta := \tilde \beta_{k,\epsilon}$ is in $\Etildephi$). Then $a_{i,\epsilon} = e_ka_\beta + a'_{i,\epsilon}$, where
\begin{align}
e_k := \left(\prod_{s=1}^k p_s^{p_{s+1}\cdots p_k}\right) 
\left(\prod_{s=1}^k a_{\chi_s}^{(p_s-1)p_{s+1}\cdots p_k}\right) 
 \label{e-k}
\end{align}
and $a'_{i,\epsilon}$ does not depend on $a_\beta$. 
\item \label{etageneous}  $a_{i,\epsilon}$ is weighted homogeneous with respect to $\eta$ with weighted degree $\epsilon$.
\item  \label{homogeneous} $a_{i,\epsilon}$ is homogeneous in $a_\beta$'s of degree $p_1\cdots p_k$.
\item \label{a-i-chi}  Fix $j$, $k + 1 \leq j \leq l$. Then  
\begin{align*}
\tilde \omega_{k, \chi_j} 
	&= \max\{\epsilon' \in \qq: a_{i, \epsilon'} \neq 0,\ \epsilon' \not\in \frac{1}{p_1 \cdots p_{j-1}}\zz\} \\ 
a_{i, \tilde \omega_{k,\chi_j}} 
	&=  a_{i_{k+1}, \tilde \omega_{k,\chi_j}}  = e_k a_{\chi_j} 
\end{align*}
where $e_k$ is from \eqref{e-k}. 
\item \label{r_i} $r_i = \tilde \omega_{k,r}$.  
\item \label{L_inear}  $\rho_i = e_k\xi+ \rho'_i$, where $e_k$ is from \eqref{e-k} and $\rho'_i \in \cc(a_\beta: \beta \in \Etildephi)$.
\end{enumerate}
\end{lemma}

\begin{proof}
We prove the lemma by induction on $k$. It follows from \cref{phi-omega-algorithm} and the definition of essential subsequence of key forms that there are precisely $i_1 - 1$ elements in $\Etildephi$ which are greater than $\chi_1$, and if we denote them as $\beta_1 > \cdots > \beta_{i_1-1}$, then each $\beta_i$ is an integer and $g_i = y -  \sum_{j=1}^{i-1} a_j x^{\beta_j}$ for each $i = 1, \ldots, i_1$. This implies that
\begin{align*}
g_i|_{y = \tilde \phi(x,\xi)} = \sum_{\beta \geq \beta_i} a_\beta x^{\beta} + \xi x^r,
\quad i = 1, \ldots, i_1.
\end{align*}
It follows that for each $i$, $1 \leq i \leq i_1$, $\rho_i = \xi$, $r_i = r$ and $a_{i,\epsilon} = a_\epsilon$ for each $\epsilon$ such that $r < \epsilon \leq \tilde \omega_i = \beta_i$. In particular \cref{a-lemma} holds for $k = 0$. \\

Now assume it holds for $k$, $0 \leq k < l$. Pick $i$, $i_{k+1} \leq i < i_{k+2}$. We prove by induction on $i$ that it holds for $i+1$. Note that 
\begin{align*}
g_{i}^{\alpha_i} |_{y = \tilde \phi(x,\xi)} 
	=
	a_{i, \tilde \omega_{i}}^{\alpha_i}x^{\alpha_i\tilde \omega_{i}} + \ldt
\end{align*}
Consider $\beta_{i, j}$'s from \cref{unique-remark}. Then $\alpha_i\tilde \omega_{i} = \beta_{i,0} + \beta_{i,i_1}\tilde \omega_{i_1} + \cdots + \beta_{i,i_{k+1}}\tilde \omega_{i_{k+1}}$ (with $\beta_{i,i_{k+1}} = 0$ for $i = i_{k+1}$). Since $g_{i_j}|_{y = \tilde \phi(x,\xi)}  =  a_{i_j, \tilde \omega_{i_j}}x^{\tilde \omega_{i_j}} + \ldt$ for each $j$, $1 \leq j \leq k+1$, \cref{phi-omega-algorithm} implies that  
\begin{align}
\theta_{i} 
	&=  \frac{a_{i, \tilde \omega_{i}}^{\alpha_i}} 
						{a_{i_1,\tilde \omega_{i_1}}^{\beta_{i,i_1}} \cdots a_{i_{k+1},\tilde \omega_{i_{k+1}}}^{\beta_{i,i_{k+1}}}},\ \label{theta-i} \text{and} \\
g_{i+1}|_{y=\tilde \phi_\delta(x,\xi)} 
	&= (g_{i}^{\alpha_i} - \theta_{i} x^{\beta_{i,0}}g_{i_1}^{\beta_{i,i_1}} \cdots g_{i_{k+1}}^{\beta_{i,i_{k+1}}})|_{y=\tilde \phi_\delta(x,\xi)} \label{g-i+1}
\end{align}
Let $\beta \in \Etildephi$. For each $j$, $1 \leq j \leq i$, denote by $\epsilon_{j,\beta}$ the right hand side of \eqref{tilde-omega-k-beta-max}, i.e.\ $\epsilon_{j,\beta}$ is the largest rational number such that $a_{j, \epsilon_{j,\beta}}$ depends non-trivially on $a_\beta$. Assume 
\begin{align}
\tilde \omega_{k+1,\beta} \leq \tilde \omega_{i+1} \label{assumption-0}
\end{align}
Since $\tilde \omega_{i+1} < p_{k+1}\tilde \omega_{i_{k+1}}$, it follows that
\begin{align*}
\tilde \omega_{k,\beta}  
	= \tilde \omega_{k+1,\beta} - (p_{k+1} - 1) \tilde \omega_{i_{k+1}} 
	< \tilde \omega_{i_{k+1}}  
\end{align*}
Assertion \eqref{assertion-2} of \cref{inequality-lemma} then implies that 
\begin{align}
\tilde \omega_{j,\beta} < \tilde \omega_{i_{j+1}},\quad j = 0, \ldots, k.\label{tilde-w-k-beta-ineq}
\end{align}
Therefore assertion \eqref{max-assertion} of \cref{a-lemma} implies by induction that
\begin{align}
\epsilon_{i_{j+1},\beta} &= \tilde \omega_{j,\beta},\ j=0, \ldots, k. \label{tilde-w-j-beta-eq}
\end{align}
From identity \eqref{g-i+1} we see that
\begin{align}
g_{i+1}|_{y=\tilde \phi_\delta(x,\xi)} 	
	&=  g^1_{i+1}(\xi, x) -  g^2_{i+1}(\xi,x),\ \text{where}\notag \\
g^1_{i+1}
	&:= 
		\left( a_{i,\tilde \omega_i}x^{\tilde \omega_i} + \cdots + a_{i,\epsilon_{i,\beta}} x^{\epsilon_{i,\beta}} + \ldt\right)^{\alpha_i}, \label{g-i+1-1}  \\
g^2_{i+1}
	&:=
		 \theta_i
		   x^{\beta_{i,0}} \prod_{j=1}^{k+1}\left(a_{i_j,\tilde \omega_{i_j}}x^{\tilde \omega_{i_j}} + \cdots +
		   a_{i_j,\tilde \omega_{j-1,\beta}} x^{\tilde \omega_{j-1,\beta}} +  \ldt\right)^{\beta_{i,i_j}} \label{g-i+1-2} 
\end{align}
For each $j =1, 2$ and each $\epsilon'$, denote $g^j_{i+1,\epsilon'}$ the coefficient of $x^{\epsilon'}$ in the expansion of $g_{i+1,j}$. Note that the terms in $g^j_{i+1}$ with degree $\alpha_i \tilde \omega_i$ in $x$ cancel each other out. Denote by $\epsilon^j_{i+1,\beta}$ the largest $\epsilon'$ such that $\epsilon' < \alpha_i \tilde \omega_i$ and $g^j_{i+1,\epsilon'}$ depends non-trivially on $a_\beta$. If $i = i_{k+1}$, then $\alpha_i = p_{k+1}$ (assertion \eqref{alpha-p} of \cref{last-curve-prop}). Morever, \eqref{tilde-w-k-beta-ineq} and \eqref{tilde-w-j-beta-eq} imply that $\epsilon_{i_{k+1},\beta} < \tilde \omega_{i_{k+1}}$, so that \eqref{g-i+1-1} implies that 
\begin{align*}
g^1_{i+1,\epsilon^1_{i+1,\beta}}x^{\epsilon^1_{i+1,\beta}}
	&= p_{k+1}a_{i_{k+1},\tilde \omega_{i_{k+1}}}^{p_{k+1}-1 } a_{i_{k+1},\epsilon_{i_{k+1},\beta}}
		x^{(p_{k+1}-1)\tilde \omega_{i_{k+1}} + \epsilon_{i_{k+1},\beta}} 
\end{align*}
In particular, 
\begin{align*}
\epsilon^1_{i+1,\beta}
	&= (p_{k+1} -1)\tilde \omega_{i_{k+1}} + \epsilon_{i_{k+1},\beta} \\
	&= (p_{k+1} -1)\tilde \omega_{i_{k+1}}  + \tilde \omega_{k,\beta}\ (\text{due to \eqref{tilde-w-j-beta-eq}})\\
	&= \tilde \omega_{k+1,\beta}
\end{align*}
On the other hand, if $i_{k+1} < i < i_{k+2}$, then assertion \eqref{max-assertion} of \cref{a-lemma} implies by induction that $\epsilon_{i,\beta} = \tilde \omega_{k+1,\beta}$. Assumption \eqref{assumption-0} implies then that $\epsilon_{i,\beta} < \tilde \omega_i$. Since $\alpha_i = 1$, it follows that
\begin{align*}
g^1_{i+1,\epsilon^1_{i+1,\beta}}x^{\epsilon^1_{i+1,\beta}}
	&= a_{i,\epsilon_{i,\beta}} x^{\epsilon_{i,\beta}}
	=  a_{i, \tilde \omega_{k+1,\beta}} x^{ \tilde \omega_{k+1,\beta}}
\end{align*}
In particular, $\epsilon^1_{i+1,\beta}= \tilde \omega_{k+1,\beta}$ in this case as well. \\

Now we compute $\epsilon^2_{i+1,\beta}$. Let $0 = j_0 < \cdots < j_m \leq k+1$ be the unique sequence of integers such that for all $j = 1, \ldots, k+1$, $\beta_{i,i_j} > 0$ iff $j \in \{j_1, \ldots, j_m\}$. Identity \eqref{tilde-w-k-beta-ineq} and the definition of $g^2_{i+1}$ from \eqref{g-i+1-2} imply that 
\begin{align*}
g^2_{i+1,\epsilon^2_{i+1,\beta}}x^{\epsilon^2_{i+1,\beta}}
	&= 
		\begin{cases}
		0 &\text{if}\ m = 0, \\
		\theta_{i}x^{\beta_{i,0}} \prod_{s=1}^m  
		\left( \beta_{i,i_{j_s}}(a_{i_{j_s},\tilde \omega_{i_{j_s}}} 
		x^{\tilde \omega_{i_{j_s}}})^{\beta_{i,i_{j_s}}-1} 
		a_{i_{j_s},\tilde \omega_{j_s-1,\beta}} x^{\tilde \omega_{j_s-1,\beta}} \right)
			&\text{otherwise.}
		\end{cases}
\end{align*}
Therefore, 
\begin{align*}
\epsilon^2_{i+1,\beta}
	&= \beta_{i,0} + \sum_{s=1}^m ((\beta_{i,i_{j_s}}-1)\tilde \omega_{i_{j_s}} + \tilde \omega_{j_s-1,\beta})
\end{align*}
Now \eqref{tilde-w-k-beta-ineq} implies that 
 \begin{align*}
\beta = \tilde \beta_{k,\tilde \omega_{k,\beta}} < \tilde \beta_{k,\tilde \omega_{i_{k+1}}} = \chi_{k+1}
\end{align*}
Consequently, assertion \eqref{assertion-1} of \cref{inequality-lemma} implies that 
\begin{align*}
\epsilon^2_{i+1,\beta} < \tilde \omega_{k+1,\beta} = \epsilon^1_{i+1,\beta} 
\end{align*}
It follows that 
\begin{align}
\epsilon_{i+1,\beta} 
	&= \epsilon^1_{i+1,\beta} = \tilde \omega_{k+1,\beta} \label{epsilon-1}\\
a_{i+1, \tilde \omega_{k+1,\beta}}
	&= g^1_{i+1,\epsilon^1_{i+1,\beta}}
	=	p_{k+1}a_{i_{k+1},\tilde \omega_{i_{k+1}}}^{p_{k+1}-1 } a_{i_{k+1},\tilde \omega_{k,\beta}} \label{a-i+1+epsilon}
\end{align}
Identity \eqref{epsilon-1} proves assertion \eqref{max-assertion} of \cref{a-lemma}. Applying the inductive hypothesis to assertion \eqref{a-i-chi} shows that the denominator of $\theta_i$ in \eqref{theta-i} is a monomial in $a_{\chi_1}, \ldots, a_{\chi_{k+1}}$. Assertion \eqref{max-assertion} coupled with this observation proves assertion \eqref{a-i-epsilon}. Identity \eqref{a-i+1+epsilon} and the inductive hypothesis implies that
\begin{align*}
a_{i+1, \tilde \omega_{k+1,\beta}}
	&= 	p_{k+1}(e_ka_{\chi_{k+1}})^{p_{k+1}-1 } (e_ka_\beta + a'_{i_{k+1}, \tilde \omega_{k,\beta}}) 
	= e_{k+1}a_\beta + a'_{i+1, \tilde \omega_{k+1,\beta}},\ \text{where}\\
a'_{i+1, \tilde \omega_{k+1,\beta}}
	&:= p_{k+1}(e_ka_{\chi_{k+1}})^{p_{k+1}-1} a'_{i_{k+1}, \tilde \omega_{k,\beta}},
\end{align*}
which proves assertion \eqref{tilde-beta-k-epsilon-assertion}. The inductive hypothesis applied to \eqref{theta-i} also gives 
\begin{align}
\eta(\theta_i) 
	&= \alpha_i \eta(a_{i, \tilde \omega_{i}}) - \sum_{j=1}^{k+1} \beta_{i,i_j} \eta(a_{i_j,\tilde \omega_{i_j}})
	= \alpha_i \tilde \omega_{i} - \sum_{j=1}^{k+1} \beta_{i,i_j}\tilde \omega_{i_j}
	= \beta_{i,0}, \label{eta-theta-i} \\
\deg(\theta_{i}) 
	&= p_{k+1} \deg(a_{i, \tilde \omega_{i}}) - \sum_{j=1}^{k+1} \beta_{i,i_j} \deg(a_{i_j,\tilde \omega_{i_j}})
	= p_1\cdots p_{k+1} - \sum_{j=1}^{k+1} \beta_{i,i_j}p_1 \cdots p_{j-1} \label{omega-theta-i}
\end{align}
Identities \eqref{g-i+1} and \eqref{eta-theta-i} immediately imply assertion \eqref{etageneous}. Note from \eqref{g-i+1-2} that the coefficient of each term in the expansion of $g^2_{i+1}$ is a sum of terms of the form 
$\theta_i\prod_{j=1}^{k+1}\prod_{s=1}^{\beta_{i,i_j}} a_{i_j,\epsilon'_{j,s}}$ with degree 
\begin{align*}
\deg(\theta_i) + \sum_{j=1}^{k+1} \sum_{s=1}^{\beta_{i,i_j}} \deg(a_{i_j,\epsilon'_{j,s}}) 
	= p_1\cdots p_{k+1} - \sum_{j=1}^{k+1} \beta_{i,i_{j}} p_1 \cdots p_{j-1} 
	+ \sum_{j=1}^{k+1}\beta_{i,i_j}p_1 \cdots p_{j-1} 
	= p_1 \cdots p_{k+1}
\end{align*}
which proves assertion \eqref{homogeneous}. Assertion \eqref{a-i-chi} follows from assertions \eqref{max-assertion} and \eqref{tilde-beta-k-epsilon-assertion} by setting $\beta := \chi_j$, $k+2 \leq j \leq l$. Assertions \eqref{r_i} and \eqref{L_inear} also follow from assertions \eqref{max-assertion} and \eqref{tilde-beta-k-epsilon-assertion} by setting $\beta := r$. This completes the proof of \cref{a-lemma}.
\end{proof}

\begin{proof}[Proof of \cref{theta-thm,c-thm}]
\Cref{theta-thm} immediately follows from \cref{a-lemma} and identities \eqref{theta-i}, \eqref{eta-theta-i}, \eqref{omega-theta-i}. Assertion \eqref{tilde-beta-k-epsilon-assertion} of \cref{a-lemma} implies that assertion \eqref{c-1} of \cref{c-thm} holds, and if $\beta' \in \qq$ is non-trivially \reachable{$\tilde \phi$}, then 
\begin{align}
c_{\tilde \phi}(\beta') = - a'_{\hat i, \tilde \omega_{\hat k(\beta'), \beta'}}/e_{\hat k(\beta')} \label{c-tilde-phi}
\end{align}
For assertion \eqref{c-2} of \cref{c-thm} apply assertions \eqref{etageneous} and \eqref{homogeneous} of \cref{a-lemma} to get that
\begin{align*}
a'_{\hat i, \tilde \omega_{\hat k(\beta'), \beta'}} (\tilde \psi) 
	&= b^{p_1 \ldots p_{\hat k(\beta')}} a^{-\omega_0\tilde \omega_{\hat k(\beta'), \beta'}} \\
e_{\tilde k(\beta')}(\tilde \psi)
	&= b^{p_1 \ldots p_{\hat k(\beta')}-1}a^{-\omega_0(\tilde \omega_{\hat k(\beta'), \beta'}-\beta')} 
\end{align*}
Assertion \eqref{c-2} of \cref{c-thm} now follows from \eqref{c-tilde-phi}. 
\end{proof}

\begin{proof}[Proof of \cref{gcd-difference}]
For each $i$, $0 < i \leq n$, define $\hat k_i := \max\{k: 0 \leq k \leq n,\ i_k < i\}$. Then it is straightforward to see that 
\begin{align*}
\omega^*_i = \omega_0(\tilde \beta_{\hat k_i, \tilde \omega_i} - \chi_1), \ i_1 < i \leq n.
\end{align*}
with $\tilde \beta_{\cdot,\cdot}$ defined as in \eqref{tilde-beta-k-epsilon}. Therefore to prove \cref{gcd-difference} it suffices to show that
\begin{align}
\zz \langle \tilde \beta_{\hat k_i, \tilde \omega_i} - \chi_1: i_1 < i \leq n \rangle 
	= \zz \langle \beta - \chi_1: \beta \in \Etildephi,\ \beta <\chi_1 \rangle 
\end{align}

Recall the definition of $\hat k(\beta)$ from \eqref{hat-k-beta}. The following claim is immediate from observation \ref{inverse} following identity \eqref{tilde-beta-k-epsilon}.


\begin{claim} \label{i-beta-claim}
The following are equivalent:
\begin{enumerate}
\item \label{i-to-beta} there are $\beta \in \Etildephi$ and $i \in \{1, \ldots, n\}$ such that $\beta = \tilde \beta_{\hat k_i, \tilde \omega_i}$. 
\item  there are $\beta \in \Etildephi$ and $i \in \{1, \ldots, n\}$  such that $\tilde \omega_i = \tilde \omega_{\hat k(\beta), \beta}$.
\end{enumerate}
If any of these conditions holds, then $\beta > \chi_1$ iff $i > i_1$. \qed
\end{claim}

Fix $i$, $i_1 < i \leq n$. We now study what happens when condition \eqref{i-to-beta} from \cref{i-beta-claim} does not hold. For $i = i_k$, $2 \leq k \leq l$, identity \eqref{omega-and-chi-0} implies that 
\begin{align}
\tilde \beta_{\hat k_i, \tilde \omega_i} = \chi_k \in \Etildephi \label{tilde-beta-chi}
\end{align}
So assume $i_k < i < i_{k+1}$ for some $k$, $1 \leq k \leq l$, and that $ \tilde \beta_{k, \tilde \omega_i} \not\in \Etildephi$. Pick a monomial $a_{\beta_1}^{\gamma_1} \cdots a_{\beta_s}^{\gamma_s}$ that appears in $a_{i,\tilde \omega_i}$. Assertions \eqref{a-i-epsilon}, \eqref{etageneous} and \eqref{homogeneous} of \cref{a-lemma} imply that
\begin{enumerate}[label= (\alph{enumi})]
\item \label{g-1-} $\beta'_j \leq \chi_1$, $1 \leq j \leq s$.
\item \label{g-2} $\sum_{j=1}^s \gamma_{j} = p_1 \cdots p_k$ ,
\item \label{g-3} $\sum_{j=1}^{s}\gamma_j \beta_j= \tilde \omega_i$.
\end{enumerate}
It follows then from definition of $\tilde \beta_{\cdot,\cdot}$ in \eqref{tilde-beta-k-epsilon} that
\begin{align*}
\tilde \beta_{k, \tilde \omega_i}  - \chi_1
	&= \tilde \omega_i - \sum_{j=1}^k(p_j - 1)\tilde \omega_{i_j} - \chi_1\\
	&= \tilde \omega_i - (\tilde \omega_{i_{k+1}} - \chi_{k+1})- \chi_1\ (\text{due to \eqref{tilde-omega-chi}}) \\
	&= \tilde \omega_i - \sum_{j=1}^k (p_j-1)p_{j+1} \cdots p_k\chi_j- \chi_1\ (\text{due to \eqref{omega-and-chi-1}}) \\
	&= \sum_{j=1}^{s}\gamma_j (\beta_j - \chi_1) - \sum_{j=1}^k (p_j-1)p_{j+1} \cdots p_k(\chi_j - \chi_1) \ (\text{due to observations \ref{g-2} and \ref{g-3}}) 
\end{align*}
This implies the following claim:

\begin{claim} \label{i-beta-no-claim}
Fix $i$, $i_1 < i \leq n$. If $\tilde \beta_{\hat k_i, \tilde \omega_i} \not\in \Etildephi$, then $\tilde \beta_{\hat k_i, \tilde \omega_i} - \chi_1 \in \zz \langle \beta - \chi_1: \beta \in \Etildephi,\ \chi_1 > \beta > \tilde \beta_{\hat k_i, \tilde \omega_i} \rangle$. \qed
\end{claim}

\Cref{gcd-difference} follows from combining \cref{i-beta-claim,i-beta-no-claim} and identity \eqref{tilde-beta-chi}. 
\end{proof}

\section{Effect of changes of coordinates on the generic descending Puiseux series of a semidegree} \label{change-appendix}
\begin{thm}[{\cite{jung}}] \label{jung}
Every polynomial automorphism $F$ of $\cc[x,y]$ has a factorization of the form 
\begin{align}
F = F_k \circ \cdots \circ F_1 \label{jung-factorization}
\end{align}
where each $F_j$ is either an {\em affine} map of the form 
\begin{align} \label{affine}
(u,v) \mapsto (au + bv + c, a'u + b'v + c'), \quad a,b,c,a',b',c' \in \cc \tag{Type I}
\end{align}
or a map of the form 
\begin{align} \label{triangular}
(u,v) \mapsto (a u + f(v), b v + c), \quad a, b, c \in \cc,\ f(v) \in \cc[v] \tag{Type II} 
\end{align} 
\end{thm}

Let $\delta$ be a semidegree on $\cc[x,y]$ such that $\delta(x) > 0$. Let $\tilde \phi_\delta(x,\xi)$ be the generic descending Puiseux series of $\delta$ in $(x,y)$ coordinates. For an automorphism $F$ of $\cc[x,y]$, we write $F_*(\tilde \phi_\delta)$ for the generic descending Puiseux series of $\delta$ after change of coordinates by $F$; the precise definition is as follows. 

\begin{defn}
Let $F:\cc[x,y] \to \cc[x,y]$ is an automorphism. Set $(u,v) := (F(x),F(y))$. If $\delta(u) > 0$, then construct the generic descending Puiseux series $\psi(u,\xi)$ of $\delta$ in $(u,v)$ coordinates. The {\em push forward} of $\tilde \phi_\delta$ by $F$ is $F_*(\tilde \phi_\delta) := \psi(u,\xi)|_{u = x}$. 
\end{defn}

\begin{rem} \label{F_*-remark}
 $F_*(\tilde \phi_\delta)$ is defined only if $\delta(F(x)) > 0$. 
Moreover, computing $F_*(\tilde \phi_\delta)$ is tantamount to the following procedure:
\begin{enumerate}[label=(Step \arabic{enumi})]
\item \label{step-1} convert the relation ``$y = \tilde \phi_\delta(x,\xi)|_{\xi = \lambda}$'', $\lambda \in \cc$, to a relation of the form ``$F(y) = \tilde \psi_\lambda(F(x))$'', where $\tilde \psi_\lambda = \sum_\beta a'_\beta x^\beta $ is a descending Puiseux series in $x$.
\item \label{step-2} Take the highest exponent $r'$ of $x$ such that $a'_{r'}$ is a non-constant function of $\lambda$. Then $F_*(\tilde \phi_\delta) = \sum_{\beta > r'} a'_\beta x^\beta + \xi x^{r'}$. 
\end{enumerate}
\end{rem}

%
%

\begin{lemma} \label{change-of-coordinates-0}
\mbox{}
\begin{enumerate}
\item  Assume $\delta(y) > 0$ and $F$ is the \ref{affine} automorphism $(x,y) \mapsto (y,x)$. Then
\begin{enumerate}
\item \label{typeIchange1}
$\deg_x(F_*(\tilde \phi_\delta)) = \omega_0/\omega_1$. 
\item \label{typeIchange2}  
If $\omega_0/\omega_1$ is a positive integer $\geq 2$, then the polydromy order of $F_*(\tilde \phi_\delta)$ is a proper divisor of $\omega_0$. 
\end{enumerate}
\item \label{typeIIIchange} Assume $F$ is an automorphism of the form $(x,y) \mapsto (x,by-f(x))$ for some $b \in \cc^*$ and $f \in \cc[x]$. Then $F_*(\tilde \phi_\delta) =  b \tilde \phi_\delta(x,\xi) - f(x)$. 
\item \label{typeIIchange0}  Assume $F$ is \ref{triangular} with $\deg(f)\delta(y) > \delta(x) > 0$. Then $\deg_x(F_*(\tilde \phi_\delta)) = 1/deg(f)$. 
\end{enumerate}
\end{lemma}

\begin{proof}
If $\delta(y) > 0$ and $F$ is the automorphism $(x,y) \mapsto (y,x)$, then $\tilde \psi_\lambda$ from \ref{step-1} of the construction of  $F_*(\tilde \phi_\delta)$ is the unique descending Puiseux series in $x$ such that  $\tilde \psi_\lambda(\tilde \phi_\delta(x, \xi)|_{\xi=\lambda}) = x$. This immediately implies that $\deg_x(F_*(\tilde \phi_\delta)) = 1/\deg_x(\tilde \phi_\delta)$, which proves assertion \eqref{typeIchange1}. Now note that for each $\lambda \in \cc$, the relation ``$y = \tilde \phi_\delta(x,\xi)|_{\xi = \lambda}$'' is equivalent to the relation ``$(x,y)  = (t^{\omega_0},\tilde \phi_\delta(t^{\omega_0},\lambda))$''. In the scenario of assertion \eqref{typeIchange2}, there are positive integers $m, p$ with $m \geq 2$ such that $\omega_0 = md$ and $\phi_\delta(t^{\omega_0},\lambda)$ is an element in $\cc[t,t^{-1}]$ with degree $d$. Consequently, $\tilde \phi_\delta(t^{\omega_0},\lambda) = at^d \tau(t)$, where $\tau(t)$ is of the form $1 + \sum_{i \geq 1} a_it^{-i}$. It follows that the $d$-th root $\tilde \tau(t)$ of $\tau(t)$ is a power series in $t^{-1}$. Therefore setting $s := t\tilde \tau(t)$ yields that the relation ``$(x,y)  = (t^{\omega_0},\tilde \phi_\delta(t^{\omega_0},\lambda))$'' is equivalent to ``$(x,y)  = (\tilde \psi(s), s^d)$'' for some Laurent series $\tilde \psi(s)$ in $s$. It follows that the polydromy order of $F_*(\tilde \phi_\delta)$ is a divisor of $d$, which proves assertion \eqref{typeIchange2}. \\

Assertions \eqref{typeIIIchange} and \eqref{typeIIchange0} follow in a straightforward way from \ref{step-1} of the construction of  $F_*(\tilde \phi_\delta)$. 
\end{proof}

For the next result we consider a change of coordinate of the form $F:(x,y) \mapsto (\bar ax + f(y), by + c)$ where $\bar a, b \in \cc^*$, $c \in \cc$ and $f = \sum_k c_ky^k\in \cc[y]$ under the assumptions that
\begin{prooflist}
\item \label{assumption-1} $\delta(x) > \delta(y) \neq 0$,
\item \label{assumption-2} $\delta(x) > \deg(f)\delta(y)$. 
\end{prooflist}
Write $\tilde \phi_\delta = \sum_\beta a_\beta x^\beta + \xi x^r$ and $F_*(\tilde \phi_\delta) = \sum_\beta a'_\beta x^\beta + \xi x^{r'}$. Let $\bar \omega_0$ be the polydromy order of $\sum_\beta a_\beta x^\beta$. Fix an $\bar \omega_0$-th root $a$ of $\bar a$. 

\begin{thm} \label{change-of-coordinates-1}
\mbox{}
\begin{enumerate}
\item \label{beta'-a'-beta'-1} Let $\beta_1 := \deg_x(\tilde \phi_\delta)$ and $\beta'_1 := \deg_x(F_*(\tilde \phi_\delta))$. Then
\begin{align}
\beta'_1 
	&=	
		\begin{cases}
			0	& \text{if}\ \delta(y) < 0\ \text{and}\  c \neq 0, \\
			\beta_1 & \text{otherwise}.
		\end{cases} \label{beta'-1} \\
a'_{\beta'_1} 
	&=	
		\begin{cases}
			c	& \text{if}\ \delta(y) < 0\ \text{and}\  c \neq 0, \\
			ba^{-\bar \omega_0\beta_1}a_{\beta_1} & \text{otherwise}.
		\end{cases} \label{a'-beta'-1}
\end{align}
\item \label{same-charac} $F_*(\tilde \phi_\delta)$ has the same formal characteristic exponents and Puiseux pairs as those of $\tilde \phi_\delta$. 
\item \label{c-dependency} For each $\beta > 0$, $a'_\beta$ does not depend on $c$. 
\item \label{f-dependency}  Let $\beta^*_1 := (d+1)\beta_1 - 1$, where
\begin{align*}
d &:= 
		\begin{cases}
			\deg(f) & \text{if}\ \delta(y)  > 0, \\
			\ord(f) & \text{if}\ \delta(y)  < 0.
		\end{cases}
\end{align*}
For each $\beta > \beta^*_1$, $a'_\beta$ does not depend on any of the coefficients of $f$. 
\item \label{>beta'} Let
\begin{align*}
\beta^*  &:= 
		\begin{cases}
			\beta^*_1 & \text{if}\ c = 0, \\
			\max\{\beta^*_1, 0\} & \text{if}\ c \neq 0.
		\end{cases}
\end{align*}
Then for each $\beta > \beta^*$, $a'_\beta = ba^{-\bar \omega_0\beta}a_\beta$.  
\item \label{beta'_1} If $\beta^*_1 > r$, then 
\begin{align}
a'_{\beta^*_1} &=
	\begin{cases}
	b a^{-\bar \omega_0\beta^*_1}a_{\beta^*_1} - \beta_1ba^{-\bar \omega_0\beta_1(d+1)}a_{\beta_1}^{d+1}c_d
				& \text{if}\ \beta^*_1 \neq 0, \\
	c + ba_{\beta^*_1} - \beta_1ba^{-\bar \omega_0}a_{\beta_1}^{d+1}c_d
				& \text{if}\ \beta^*_1 =  0. 
	\end{cases}
\end{align} 
\item \label{0}  If $\max\{r, \beta^*_1\} < 0$ then $a'_0 =	c + ba_0$.  
\end{enumerate} 
\end{thm}

\begin{proof}
Assumptions \ref{assumption-1} and \ref{assumption-2} imply that 
\begin{align*}
\delta(\bar a x + f(y)) &= \delta(x) \\
\delta(by + c) 
	&=	
		\begin{cases}
			0	& \text{if}\ \delta(y) < 0\ \text{and}\  c \neq 0, \\
			\delta(y) & \text{otherwise}.
		\end{cases} 
\end{align*}
which implies \eqref{beta'-1}. Let $\lambda \in \cc$. \Cref{F_*-remark} implies that there is an identity of the form  
\begin{align}
\sum_{\beta'_1 \geq \beta > r'} a'_\beta x^\beta + \sum_{\beta \leq r'} a''_\beta(\lambda) x^{\beta}  
	&= c + b\sum_{\beta_1 \geq \beta > r} a_\beta h_\beta(x) + b\lambda h_r(x) \label{expansion-0}
\end{align}
where $a''_{r'}$ is a non-constant function of $\lambda$, and for each $\beta$, 
\begin{align*}
h_\beta(x)
	&:= 
	\left(\frac{x - f\left(\left(\sum_{\beta'_1 \geq \alpha > r'} a'_\alpha x^\beta + \sum_{\alpha \leq r'} a''_\alpha(\lambda) x^{\alpha}  -c\right)/b\right)}{a^{\bar \omega_0}}\right)^\beta  
	= a^{-\bar \omega_0\beta} x^\beta
		 \left( 1 -  \tilde f(x, \lambda) \right)^\beta,\ \text{where}\\
\tilde f(x, \lambda) 	
	&:= x^{-1}  \left(\sum_k c_kb^{-k}\left(\sum_{\beta'_1 \geq \alpha > r'} a'_\alpha x^\beta + \sum_{\alpha \leq r'} a''_\alpha(\lambda) x^{\alpha}  -c\right)^k \right) %
\end{align*}
Identity \eqref{beta'-1} and assumption \ref{assumption-2} imply that
\begin{align}
\deg_x(\tilde f(x,\lambda)) = -1 + d\beta'_1  < 0.
\end{align}
Therefore we can expand $(1 - \tilde f(x,\lambda))^\beta$ in a descending Puiseux series to get 
\begin{align}
h_\beta(x)
	&= a^{-\bar \omega_0\beta} x^\beta
		\left(1 - 
				\beta x^{-1} \tilde f(x,\lambda) + \beta(\beta-1)x^{-2}(\tilde f(x,\lambda))^2/2 + \cdots 
		\right) \notag \\
	&=  a^{-\bar \omega_0\beta} x^\beta
	\left(1 - 
			\beta x^{-1}  \left(\sum_k c_kb^{-k}\left(\sum_{\beta'_1 \geq \alpha > r'} a'_\alpha x^\beta + \sum_{\alpha \leq r'} a''_\alpha(\lambda) x^{\alpha}  -c\right)^k \right) - \cdots 
	\right) \label{h-beta}
\end{align}
Combining \eqref{h-beta} and \eqref{expansion-0} gives 
\begin{align}
\begin{split}
\sum_{\beta'_1 \geq \beta > r'} &a'_\beta x^\beta + \sum_{\beta \leq r'} a''_\beta(\lambda) x^{\beta}  - c \\
	&= b\sum_{\beta_1 \geq \beta > r}a^{-\bar \omega_0\beta} a_\beta x^\beta
		\left(1 - 
				\beta x^{-1}  \left(\sum_k c_kb^{-k}\left(\sum_{\beta'_1 \geq \alpha > r'} a'_\alpha x^\beta + \sum_{\alpha \leq r'} a''_\alpha(\lambda) x^{\alpha}  -c\right)^k \right) - \cdots 
		\right)  \\
	&\qquad + b\lambda h_r (x)
\end{split}\label{expansion-1} 
\end{align}
The term with highest degree in $x$ in the expansion of the right hand side of \eqref{expansion-1} is 
\begin{itemize}
\item $c$ if $\beta_1 < 0$ and $c \neq 0$,
\item $ba^{-\bar \omega_0\beta_1} x^{\beta_1}$ otherwise.
\end{itemize}
This implies \eqref{a'-beta'-1} and completes the proof of assertion \eqref{beta'-a'-beta'-1}. 
Now pick $\beta \leq \beta'_1$. Let $T_\beta$ be the term in the expansion of the right hand side of \eqref{expansion-1} such that 
\begin{defnlist}
\item \label{a'-beta-appears} $a'_\beta$ appears in $T_\beta$,
\item $\deg_x(T_\beta)$ is the highest among all terms satisfying \ref{a'-beta-appears}. 
\end{defnlist}
Note that 
\begin{defnlist}[resume]
\item if $f$ is a constant polynomial, then $T_\beta  = 0$.
\end{defnlist}
On the other hand, if $f$ is a non-constant polynomial, then a straightforward computation yields the following observations:
\begin{defnlist}[resume]
\item \label{T-beta-1} If either $\beta_1 > 0$, or if $\beta_1 < 0$, $c = 0$ and $d > 0$, then 
\begin{align*}
T_\beta 
	&= ba^{-\bar \omega_0\beta_1} a_{\beta_1} x^{\beta_1} \times (-\beta_1 x^{-1}c_db^{-d}d_\beta(a'_{\beta'_1}x^{\beta'_1})^{d-1}a'_\beta x^{\beta} ) \\
	&= - \beta_1 a^{-\bar \omega_0\beta_1} b^{1-d} c_d d_\beta a_{\beta_1}(a'_{\beta'_1})^{d-1} a'_\beta x^{\beta + \beta_1 + (d-1)\beta'_1 -1},\ \text{where} \\
d_\beta 
	&=
	\begin{cases}
	1 &\text{if}\ \beta = \beta'_1,\\
	d & \text{if}\ \beta < \beta'_1. 
	\end{cases}
\end{align*}
\item If $\beta_1 < 0$, $c = 0$ and $d = 0$, then 
\begin{align*}
T_\beta 
	&=
	 - \beta_1 a^{-\bar \omega_0\beta_1} b^{1-d'} c_{d'} d_\beta a_{\beta_1} (a'_{\beta'_1})^{d'-1}a'_\beta x^{\beta + \beta_1 + (d'-1)\beta'_1 -1}
\end{align*}
where $d' := \ord(f - c_0)$ and $d_\beta$ is as in observation \ref{T-beta-1}. 
\item \label{T-beta-1-3} If $\beta_1 < 0$ and $c \neq 0$, then 
\begin{align*}
T_\beta 
	&=
	\begin{cases}
	0 
	&\text{if}\ \beta = \beta'_1 = 0, \\
	- \beta_1 a^{-\bar \omega_0\beta_1} b^{1-d'} c_{d'} d'_\beta a_{\beta_1} (a'_{\beta'_2})^{d'-1}a'_\beta x^{\beta + \beta_1 + (d'-1)\beta'_2 -1} 
	& \text{if} \ \beta < 0. 
	\end{cases}
\end{align*}
where $d' := \ord(f - c_0)$, $\beta'_2 := \deg_x(\sum_{\beta < 0}a'_{\beta'} + \xi x^{r'})$ is the second largest exponent appearing in $F_*(\tilde \phi_\delta)$, and 
\begin{align*}
d'_\beta 
	&=
	\begin{cases}
	1 &\text{if}\ \beta = \beta'_2,\\
	d' & \text{if}\ \beta < \beta'_2. 
	\end{cases}
\end{align*}
\end{defnlist}  
Assumption \ref{assumption-2} and identity \eqref{beta'-1} imply that $\deg_x(T_\beta) < \beta$ in every case. Equating coefficients of both sides of \eqref{expansion-1} then implies that for each $\beta $ there is an identity of the form
\begin{align}
a'_\beta =
	 \begin{cases}
	 ba^{-\bar \omega_0\beta_1} a_\beta + S_\beta
		 &\text{if}\ \beta \neq 0,\\
	ba^{-\bar \omega_0\beta_1} a_\beta + S_\beta + c
		&\text{if}\ \beta = 0.
	 \end{cases} \label{expansion-2}
\end{align}
where $S_\beta$ is a sum of monomials in $a_{\alpha}$ and $a'_{\alpha'}$'s with $\alpha, \alpha' > 0$. Define
\begin{align*}
\scrE  &:= \{\beta: a_\beta \neq 0\} \\
\scrE'  &:= \{\beta: a'_\beta \neq 0\} 
\end{align*}
Identity \eqref{expansion-2} implies that for each $\beta \in \qq$,
\begin{defnlist}[resume]
\item \label{characteristic-1} if $\beta \in \scrE$ and $\beta \not\in \zz\langle \beta' \in \scrE: \beta' > \beta \rangle$, then $\beta \in \scrE'$.
\item \label{characteristic-2} if $\beta \in \scrE'$, then either $\beta \in \scrE$, or $\beta \in \zz\langle \beta' \in \scrE \cup \scrE': \beta' > \beta \rangle$
\end{defnlist}
Observations \ref{characteristic-1} and \ref{characteristic-2} immediately imply assertion \eqref{same-charac}. If $\delta(y) < 0$ then identity \eqref{beta'-1} shows that assertion \eqref{c-dependency} is vacuously true. On the other hand, if $\delta(y) > 0$, then an application of observation \ref{T-beta-1} with $\beta = 0$ implies that the degree in $x$ of each term on the right hand side of \eqref{expansion-1} in which $c$ appears is negative. This implies assertion \eqref{c-dependency}. For assertion \eqref{f-dependency}, let $T_f$ be the term on the right hand side of the expansion of \eqref{expansion-1} which has the highest degree in $x$ among all terms which depend non-trivially on coefficients of $f$. It is straightforward to see that
\begin{align}
T_f &=
			\begin{cases}
			T_{\beta'_1} 
				&\text{if either $\beta_1 > 0$, or if $\beta_1 < 0$, $d > 0$ and $c = 0$,} \\
			T_{\beta'_2} 
				&\text{if $\beta_1 < 0$, $d > 0$ and $c \neq 0$,} \\
			- \beta_1 a^{-\bar \omega_0\beta_1} b c_0a_{\beta_1}  x^{\beta_1 -1} 
				&\text{if $\beta_1 < 0$ and $d = 0$,}
			\end{cases} \label{T-f}
\end{align}
where $\beta'_2$ is as in observation \ref{T-beta-1-3}. If either $\beta_1 > 0$, or $d > 0$ and $c = 0$, then observation \ref{T-beta-1} implies that $\deg_x(T_f) =(d+1)\beta_1 -1  =: \beta^*_1$. If $\beta_1 < 0$, $d > 0$ and $c \neq 0$, then since $a'_{\beta'_1} = c$, identity \eqref{expansion-1} implies that $\beta'_2 = \beta_1$, so that observation \ref{T-beta-1-3} implies that $\deg_x(T_f) =\beta^*_1$. It follows that $\deg_x(T_f) = \beta^*_1$ in every case, which implies assertion \eqref{f-dependency}. Identity \eqref{expansion-1} and a combination of assertions \eqref{c-dependency} and \eqref{f-dependency} yields assertion \eqref{>beta'}. Moreover, identities \eqref{expansion-2} and \eqref{T-f} imply that 
\begin{align*}
S_{\beta^*_1} = T_f
\end{align*}
which implies assertion \eqref{beta'_1}. Assertion \eqref{0} follows from identity \eqref{expansion-1} and assertion \eqref{f-dependency}. 
\end{proof}

\section{Existence and uniqueness of normal forms} \label{normal-proof-1}
In this section we prove \cref{normal-thm-0,normal-morphism}. 

\subsection{Proof of assertion \eqref{existence-assertion} of \cref{normal-thm-0} (existence of normal forms)}
Let $\delta$ be a semidegree on $\cc[x,y]$ with key sequence $\vec \omega := (\omega_0, \ldots, \omega_{n+1})$ in $(x,y)$-coordinates. It is straightforward to see from assertions \eqref{typeIchange1} and \eqref{typeIchange2} of \cref{change-of-coordinates-0} that after a sequence of changes of coordinates of the form $(x,y) \mapsto (y,x)$ and $(x,y) \mapsto (x,y-f(x))$, we can ensure that
\begin{itemize}
\item either $n = 0$ and $\omega_0 \geq \omega_1$,
\item or $n \geq 1$, $\omega_0 > \omega_1$, and $\frac{\omega_1}{\omega_0} \not\in \{\frac{1}{k}: k \in \zz,\ k \geq 1\} \cup \{0\}$. 
\end{itemize}
Therefore, w.l.o.g.\ we may assume that $\delta$ satisfies properties \eqref{generaly->=1}--\eqref{generally-alpha}. We now find a change of coordinate which ensures that $\delta$ satisfies \eqref{generally-omega} as well (and continues to satisfy the other properties). It is straightforward to see that the set $\scrE_{\vec \omega} $ from \eqref{exponents-omega} can be expressed as follows: 
\begin{align}
\scrE_{\vec \omega} 
	&= 
	\begin{cases}
	\{\beta: \beta = (d+1)\frac{\omega_1}{\omega_0}-1,\ d \in \zz,\ 0 \leq d < \frac{\omega_0}{\omega_1},\ \chi_{l+1} < \beta < \frac{\omega_1}{\omega_0}\} \cup \{0\}
		& \text{if}\ \omega_1 > 0 \\
	\{\beta: \beta = (d+1)\frac{\omega_1}{\omega_0}-1,\ d \in \zz,\ 0 \leq d, \ \chi_{l+1} < \beta <  \frac{\omega_1}{\omega_0}\}
		& \text{if}\ \omega_1< 0
	\end{cases} \label{e-omega}
\end{align}
Let $\beta'_1 > \cdots > \beta'_M$ be the elements of $\scrE_{\vec \omega} \setminus \{0\}$. Let $d_m := (\beta_m+1)\omega_0/\omega_1 -1$, $1 \leq m \leq M$. %
%
%
Let $(a_1, \ldots, a_M)$ be an arbitrary element in $\cc^M$. \Cref{change-of-coordinates-1} implies that there is $(c_1, \ldots, c_M ) \in \cc^M$ such that the following holds: if $F_m$ is the change of coordinates $(x,y) \mapsto (x +c_my^{d_m}, y)$ and $\tilde \psi_m := (F_m \circ F_{m-1} \circ \cdots \circ F_1)_*(\tilde \phi_\delta)$, then for each $m$,
\begin{itemize}
\item the formal characteristic exponents of $\tilde \psi_m$ are the same as those of $\tilde \phi_\delta$;
\item for all $\beta > \beta'_m$, the coefficients of $x^\beta$ in $\tilde \psi_m$ and $\tilde \psi_{m-1}$ are equal (here $\tilde \psi_0$ is defined to be $\tilde \phi_\delta$);
\item the coefficient of $x^{\beta'_m}$ in $\tilde \psi_m$ is $c_m$. 
\end{itemize}
\Cref{c-thm} then implies that there is a unique $(c_1, \ldots, c_M)$ such that after the change of coordinates by $F_m \circ \cdots \circ F_1$, property \eqref{generally-omega} is satisfied with $\scrE_{\vec \omega}$ replaced by $\scrE_{\vec \omega}\setminus \{0\}$. In the case that $\omega_1 > 0$, a subsequent change of coordinates of the form $(x,y) \mapsto (x, y + b)$ then ensures that property \eqref{generally-omega} is completely satisfied. Since none of these changes of coordinates affect $\omega_0$ or $\omega_1$ (so that properties \eqref{generaly->=1}--\eqref{generally-alpha} continue to hold), this completes the proof.  \qed \\

In \cref{unique-section} we use the following lemma, which follows immediately from the preceding discussion.

\begin{lemma} \label{c-lemma}
Let $\delta$ be a semidegree on $\cc[x,y]$ with generic descending Puiseux series $\tilde \phi_\delta(x,\xi) = \sum_\beta a_\beta x^\beta + \xi x^r$ and key sequence $\vec \omega := (\omega_0, \ldots, \omega_{n+1})$ in $(x,y)$ coordinates. Assume $\omega_1 \neq 0$. Then the following are equivalent: 
\begin{enumerate}
\item $\vec \omega$ satisfies \eqref{generally-omega}.
\item $a_\beta = c_{\tilde \phi_\delta}(\beta) $ for all $\beta \in \scrE_{\vec \omega}$ (where $c_{\tilde \phi_\delta}(\beta) $ is defined as in \cref{c-thm}). \qed
\end{enumerate}
\end{lemma}  

\subsection{Proof of \cref{normal-morphism} and assertion \eqref{unique-assertion} of \cref{normal-thm-0} (uniqueness of normal forms and automorphisms that preserve the normal form)} \label{unique-section}
In this section we prove \cref{normal-morphism} and assertion \eqref{unique-assertion} of \cref{normal-thm-0} simultaneously. More precisely, we assume the following:
\begin{itemize}
\item $\delta$ is a divisorial semidegree on $\cc[x,y]$ such that $\delta(x)  > 0$ and the  key sequence $\vec \omega := (\omega_0, \ldots, \omega_{n+1})$ of $\delta$ in $(x,y)$ coordinates is in the normal form. T
\item $F: \cc[x,y] \to \cc[x,y]$ is an automorphism such that the key sequence $\vec \omega'$ of $\delta$ with espect to $(x',y') := (F(x), F(y))$ coordinates is also in the normal form.
\end{itemize}
Under these assumptions we show that each of the assertions \eqref{trivial-F}--\eqref{unique-puiseux} of \cref{normal-morphism} holds, and that in each case $\vec \omega' = \vec \omega$. \\

Consider a factorization
\begin{align*}
F = F_k \circ \cdots \circ F_1 
\end{align*}
as in \eqref{jung-factorization} such that the length $k$ of the factorization is the {\em minimum}. It then follows that
\begin{enumerate}[label=(min-\arabic{enumi})]
\item \ref{affine} and \ref{triangular} maps alternate, i.e.\ $F_j$ is \ref{affine} iff $F_{j+1}$ is \ref{triangular} for all $j$.
\item \label{tr>1} For each $F_j$ of \ref{triangular}, $F_j(u,v) =  (au + f(v), bv + c) $ for some $f$ with $\deg(f) > 1$, and 
\item \label{non-trivially-affine} if there exists $j$ such that $F_j$ is \ref{affine} of the form $(u,v) \mapsto (au + bv + c, a'u + b'v + c')$ with $a' = 0$, then in fact $j=k=1$ and $F = F_1$.
\end{enumerate}
Now we show that $k = 1$. For all $j$, $1 \leq j \leq k$, let $(x_j,y_j) := (F_j \circ \cdots \circ F_1)(x,y)$. Let $\vec \omega^j := (\omega^j_0, \ldots, \omega^j_{n_j+1})$ be the key sequence of $\delta$ and $\tilde \phi_j(x_j, \xi) = \phi_j(x_j) + \xi x_j^{r_j}$ be the generic descending Puiseux series of $\delta$ in $(x_j,y_j)$ coordinates. 

\begin{claim} \label{initially-bad}
Assume $k > 1$. Then there exists $j$, $1 \leq j \leq k$, such that
\begin{enumerate}
\item $F_j$ is \ref{triangular},
\item Let $s_j := 1/\deg_{x_j}(\tilde \phi_j) = \omega^j_0/\omega^j_1$ is an integer $\geq 2$.
\item $\phi_j \neq 0$, or equivalently, $n_j \geq 1$.
\end{enumerate}
\end{claim}

\begin{proof}
If $F_1$ is \ref{affine}, then \ref{non-trivially-affine} implies that $a' \neq 0$, so that $\delta(y_1) = \delta(x) \geq \delta(x_1)$. Then \ref{tr>1} and assertion \eqref{typeIIchange0} of \cref{change-of-coordinates-0} imply that $j=2$ satisfies the first two assertions of \cref{initially-bad}. On the other hand, assertion \eqref{typeIchange2} implies that the number of formal Puiseux pairs of $\phi_2$ is one more than that of $\phi_1$. In particular, $\phi_2 \neq 0$, which proves \cref{initially-bad} in this case.\\

Now assume $F_1$ is \ref{triangular}. If $\deg(f)\delta(y) > \delta(x)$, then $j=1$ satisfies the claim (assertion \eqref{typeIIchange0} of \cref{change-of-coordinates-0}). So assume $\deg(f)\delta(y) \leq  \delta(x)$. Since $\deg(f) \geq 2$ (\ref{tr>1}), the normality of $\vec \omega$ implies that $\deg(f)\delta(y) <  \delta(x)$. 
It follows that $\delta(x_1) = \delta(x)$ and $\delta(y_1) = \delta(y)$. Since $F_2$ is \ref{affine}, \ref{tr>1} then implies that 
\begin{prooflist}
\item \label{eq0} if $a = 0$, then $\omega^2_1 > \omega^2_0$;
\item \label{neq0} if $a \neq 0$, then the highest degree term of $\tilde \phi_2(x, \xi)$ is $\frac{a'}{a}x_2$, so that $\omega^2_0 = \omega^2_1$ and $\phi_2 \neq 0$. 
\end{prooflist} 
In particular, $\vec \omega^2$ is not in the normal form. It follows that $k \geq 3$. Now the same arguments as in the first paragraph of the proof of this claim shows that $j = 3$ satisfies the claim. 
\end{proof}

\begin{claim} \label{continually-bad}
Let $j$ be as in \cref{initially-bad}. Then $k \geq j+2$. Moreover $j+2$ also satisfies \cref{initially-bad}.
\end{claim}

\begin{proof}
Let $j$ be as in \cref{initially-bad}. Then $\vec \omega^j$ is not in the normal form, so that $k \geq j+1$. $F_{j+1}$ is \ref{affine} with $a' \neq 0$, so that $\delta(y_{j+1}) = \delta(x_j) = s_j\delta(y_j)$. At first assume $a= 0$. Then $\delta(x_{j+1}) = \delta(y_j)$, so that $\omega^{j+1}_1 = s_j \omega^{j+1}_0 > 0$. Consequently, $\vec \omega^{j+1}$ violates \eqref{trivially-omega} and \eqref{generally-less}; therefore it is not in the normal form. In particular, $k > j +1$. The same arguments as in the first paragraph of the proof of \cref{initially-bad} then ensure that $j+2$ also satisfies \cref{initially-bad}. On the other hand, if $a \neq 0$, then, as in observation \ref{neq0} from the proof of \cref{continually-bad}, it follows that 
$\omega^{j+1}_0 = \omega^{j+1}_1 > 0$ and $n_{j+1} \geq 1$. Since $ \vec \omega^{j+1}$ is not in the normal form, we have that $k \geq j+2$. The same arguments as in the first paragraph of the proof of \cref{initially-bad} then show that $j+2$ satisfies \cref{initially-bad}, as required.
\end{proof}

Since $(x_k,y_k) = (x',y')$ and $\vec \omega^k = \vec \omega'$ is in the normal form, \cref{initially-bad,continually-bad} imply that $k = 1$, i.e.\ $F = F_1$. This immediately implies that
\begin{prooflist}[start=3]
\item if $\omega_0 = \omega_1 = 1$, then assertion \eqref{trivial-F} of \cref{normal-morphism} holds, and $\vec \omega' = \vec \omega$;
\item if $\omega_0 = 1$ and $\omega_1 = 0$, then assertion \eqref{trivial-F-zero} of \cref{normal-morphism} holds, and $\vec \omega' = \vec \omega$;
\item \label{general-observation} otherwise $F : (x,y) \mapsto (\bar a x + f(y), \bar b y + c)$ where $\bar a, \bar b \in \cc^*$, $c \in \cc$ and $f(y) \in \cc[y]$. 
\end{prooflist}
Now we prove assertion \eqref{positive-F} of \cref{normal-morphism}. So assume $\omega_0 > \omega_1 > 0$. Let $d := \deg(f)$. The normality of $\vec \omega$ implies that $d\omega_1 \neq \omega_0$. If $d\omega_1 > \omega_0$, then observation \ref{general-observation} would imply that $\omega'_0 < \omega'_1$, contradicting the normality of $\vec \omega'$. It follows that $d\omega_1 < \omega_0$. %
Assertion \eqref{same-charac} of \cref{change-of-coordinates-1} implies that $\vec \omega$ and $\vec \omega'$ have the same formal characteristic pairs, so that the sets $\Eomega$ and $\Eomegaprime$ (defined as in \eqref{exponents-omega}) are identical. Since both $\vec \omega$ and $\vec \omega'$ are normal, \cref{c-lemma} implies that 
\begin{align}
a_\beta = c_{\tilde \phi_\delta}(\beta),\ a'_\beta = c_{\tilde \psi_\delta}(\beta) 
\label{c-phi-psi-0}
\end{align}
for all $\beta \in \Eomega$. Define $\beta_1, \beta'_1, \beta'$ as in \cref{change-of-coordinates-1}. Assertion \ref{positive-F} of \cref{normal-morphism} is a straightforward consequence of \cref{beta'-claim} below.

\begin{claim} \label{beta'-claim}
$\beta' \leq \chi_{l+1}$ (where $\chi_{l+1}$ is as in assertion \eqref{positive-F} of \cref{normal-morphism}).
\end{claim}

\begin{proof}
Assertion \eqref{>beta'} of \cref{change-of-coordinates-1} and assertion \eqref{c-2} of \cref{c-thm} imply that 
\begin{align}
c_{\tilde \psi_\delta}(\beta') = ba^{-\bar \omega_0\beta'}c_{\tilde \phi_\delta}(\beta') \label{c-phi-psi-1}
\end{align} 
where $b := \bar b$, $a$ is a primitive $\bar \omega_0$-th root of $\bar a$, and $\bar \omega_0$ is the polydromy order of $\sum_\beta a_\beta x^\beta$. Now assume to  the contrary of the claim that $\beta' > \chi_{l+1}$. The normality of $\vec \omega$ implies that $\beta'_1 \neq 0$, so that either $\beta' = \beta'_1 > 0$, or $\beta' = 0 > \beta'_1$. At first consider the case that $\beta' = \beta'_1 > 0$. Then assertion \eqref{beta'_1} of \cref{change-of-coordinates-1} implies that
\begin{align}
a'_{\beta'_1} &=
	b a^{-\bar \omega_0\beta'_1}a_{\beta'_1} - \beta_1ba^{-\bar \omega_0\beta_1(d+1)}a_{\beta_1}^{d+1}c_d
	\label{a'}
\end{align}
On the other hand, since $\beta'_1 > \chi_{l+1}$, it follows that $\beta'_1 \in \Eomega$. Identities \eqref{c-phi-psi-0} and \eqref{c-phi-psi-1} then imply that $a'_{\beta'_1} =
b a^{-\bar \omega_0\beta'_1}a_{\beta'_1} $, so that \eqref{a'} imply that $c_d = 0$, which is impossible. Now assume $\beta' = 0 > \beta_1$. This implies in particular that $c \neq 0$. Now assertion \eqref{0} of \cref{change-of-coordinates-1} implies that 
\begin{align*}
a'_0 =	c + ba_0
\end{align*}
On the other hand, since $0 \in \Eomega$, identities \eqref{c-phi-psi-0} and \eqref{c-phi-psi-1} imply that $a'_{0} = b a_{0} $, which implies in turn that $c = 0$. This gives the desired contradiction and concludes the proof of the claim. 
\end{proof}

Assertion \eqref{negative-F} of \cref{normal-morphism} follows from exactly the same argument as in the proof of assertion \eqref{positive-F} with $d$ replaced by $\ord(f)$. In particular, \cref{beta'-claim} remains true in the case that $\omega_0 > \omega_1 > 0$. Assertion \eqref{unique-puiseux} of \cref{normal-morphism} then follows from assertion \eqref{>beta'} of \cref{change-of-coordinates-1}. \\

Note that we proved all the assertions of \cref{normal-morphism} under the weaker assumption that $\vec \omega$ and $\vec \omega'$ are both in normal form (i.e.\ we did {\em not} assume that $\vec \omega' = \vec \omega$). \Cref{theta-thm} and assertion \eqref{unique-puiseux} of this stronger version of \cref{normal-morphism} then imply that $\vec \omega' = \vec \omega$. This completes the proof of assertion \ref{unique-assertion} of \cref{normal-thm-0}. \qed

\section{Automorphisms preserving a semidegree} \label{normal-proof-2}
In this section we prove \cref{preserving-thm}. Let the notation be as in \cref{normal-morphism,preserving-thm}. If $n = 0$, then $\phi(x) = 0$, and therefore assertion \eqref{n=0} of \cref{preserving-thm} is an immediate consequence of assertion \eqref{unique-puiseux} of \cref{normal-morphism}. So assume $n \geq 1$. \Cref{normal-morphism} implies that the generic descending Puiseux series of $\delta$ in $(x',y')$ coordinates is
\begin{align}
\tilde \psi_\delta(x',\xi) = b\sum_\beta a_\beta a^{-\bar \omega_0\beta} x'^{\beta} + \xi x'^{r} \label{psi-delta}
\end{align}
where $\tilde \phi_\delta(x,\xi) = \sum_\beta a_\beta x^{\beta} + \xi x^r$ is the descending Puiseux series of $\delta$ in $(x,y)$ coordinates. Let $\scrG$ be the group consisting of all $F$ described in assertion \eqref{n>=1} of \cref{preserving-thm}. It suffices to show that $F \in \scrG$ iff $\tilde \psi_\delta(x,\xi)$ is conjugate to $\tilde \phi_\delta(x,\xi)$. At first assume $F \in \scrG$. Then 
\begin{align*}
\tilde \psi_\delta(x',\xi) 
	= \sum_\beta a_\beta a^{\bar \omega_1-\bar \omega_0\beta} x'^{\beta} + \xi x'^{r}
	= a_{\beta_1}x'^{\beta_1} +  \sum_{\beta < \beta_1} a_\beta a^{\bar \omega_1-\bar \omega_0\beta} x'^{\beta} + \xi x'^{r}
\end{align*}
where $\beta_1 := \deg_x(\tilde \phi_\delta) = \omega_1/\omega_0$. If $n = 1$, then $\phi(x)$ has only one term $a_{\beta_1}x^{\beta_1}$, so that $\tilde \psi_\delta(x, \xi) = \tilde \phi_\delta(x,\xi)$. On the other hand, if $n \geq 2$, then \cref{gcd-difference} and the definition of $a$ from assertion \eqref{n>=1} of \cref{preserving-thm} implies that $a^{\bar \omega_1-\bar \omega_0\beta} = 1$ for all $\beta$ such that $a_\beta \neq 0$. It follows again that $\tilde \psi_\delta(x, \xi) = \tilde \phi_\delta(x,\xi)$ and completes the proof of the ($\im$) implication.\\

Now assume that $\tilde \psi_\delta(x,\xi)$ is conjugate to $\tilde \phi_\delta(x,\xi)$. Then it follows from \eqref{psi-delta} that there is an $\bar \omega_0$-th root of unity $\zeta$ such that $ba^{-\bar \omega_0\beta} = \zeta^{\bar \omega_0\beta}$ for all $\beta$ such that $a_\beta \neq 0$. Since $n \geq 1$, it follows that $\phi \neq 0$. In particular, $a_{\beta_1} \neq 0$, where $\beta_1 := \omega_1/\omega_0 = \bar \omega_1/\bar \omega_0$. Therefore $b = (\zeta a)^{\bar \omega_0\beta_1} =  (\zeta a)^{\bar \omega_1} = a'^{\bar \omega_1}$, where $a' := \zeta a$. Since $a'^{\bar \omega_0} = a^{\bar \omega_0}$, we can replace $a$ by $a'$ and assume that
\begin{align*}
b = a^{\bar \omega_0\beta}\ \text{for all $\beta$ such that $a_\beta \neq 0$.} 
\end{align*}
In particular, $b = a^{\bar \omega_1}$, which implies that $F \in \scrG$ in the case that $\phi$ has only one monomial term, or equivalently, if $n = 1$. On the other hand if $\phi$ has more than one monomial term, then it follows that $a^{\bar \omega_0\beta -\bar \omega_1 } = 1$ for all $\beta$ such that $a_\beta \neq 0$. \Cref{gcd-difference} then implies that $a$ is an $\baromegastar$-th root of unity ($\baromegastar$ being as in assertion \ref{aut-general}), so that $F \in \scrG$ in this case as well. This completes the proof of assertion \eqref{n>=1} of \cref{preserving-thm}.

\bibliographystyle{alpha}
\bibliography{bibi}



\end{document}